\documentclass[onefignum,onetabnum]{siamart171218}

\usepackage[%
left=1.5in,
right=1.5in,
bottom=0.75in,
top=1in,
footskip=0.5in,
]{geometry}

\usepackage{csquotes,graphicx,comment}
\usepackage{epstopdf}
\usepackage[mathscr]{eucal}
\usepackage{subcaption}
\usepackage[font=footnotesize]{caption}
\usepackage{amsfonts,amsmath,amssymb,mathtools,bm} 
\usepackage{enumitem}
\usepackage{mathalfa,colonequals}
\usepackage{cleveref}
\crefname{proposition}{\textup{Proposition}}{\textup{Assumptions}}
\crefname{assumption}{\textup{Assumption}}{\textup{Assumptions}}
\crefname{lemma}{\textup{Lemma}}{\textup{Lemmas}}
\crefname{algorithm}{\textup{Algorithm}}{\textup{Algorithms}}
\crefname{theorem}{\textup{Theorem}}{\textup{Theorems}}
\crefname{remark}{\textup{Remark}}{\textup{Remarks}}
\crefname{example}{\textup{Example}}{\textup{Examples}}
\crefname{corollary}{\textup{Corollary}}{\textup{Corollaries}}
\crefname{subsection}{\textup{Section}}{\textup{Subsections}}
\crefname{section}{\textup{Section}}{\textup{Sections}}

\usepackage{graphicx,color,xcolor}
\usepackage{algorithm} 
\usepackage{algorithmic} 

\usepackage{cite}

\makeatletter
\let\@citexOld\@citex
\def\@citex[#1]#2{\textup{\@citexOld[#1]{#2}}}
\makeatother

\usepackage{placeins}
\usepackage{csquotes}
\ifpdf
  \DeclareGraphicsExtensions{.eps,.pdf,.png,.jpg}
\else
  \DeclareGraphicsExtensions{.eps}
\fi

\usepackage{ulem}
\headers{Adjoint-based calibration of nonlinear SDEs}{J. Bartsch, R. Denk, and S. Volkwein}

\title{Adjoint-based calibration of nonlinear stochastic differential equations}

\author{
    Jan Bartsch\thanks{Department of Mathematics and Statistics, University of Konstanz, Germany (\email{\{jan.bartsch, robert.denk, stefan.volkwein\}@uni-konstanz.de}).}
    \and Robert Denk\footnotemark[1]
    \and Stefan Volkwein\footnotemark[1]
}

\newcommand{\ua}{{u_{\mathsf a}}}

\newcommand{\ub}{{u_{\mathsf b}}}

\newcommand{\X}{{\mathscr X}}

\newcommand{\U}{{\mathscr U}}
\newcommand{\Uad}{{\U_{\mathsf{ad}}}}
\newcommand{\Uadh}{{\U_{\mathsf{ad}}^h}}

\newcommand{\Uc}{{\U^{\mathsf{c}}}}
\newcommand{\UadCalib}{{\U^{\mathsf{c}}_{\mathsf{ad}}}}
\newcommand{\cd}{{c^{\mathsf d}}}
\newcommand{\cdIndex}[1]{{c_{#1}^{\mathsf d}}}
\newcommand{\cdT}{{c^{\mathsf d}_T}}
\newcommand{\sd}{{\sigma^{\mathsf d}}}
\newcommand{\sdT}{{\sigma^{\mathsf d}_T}}
\newcommand{\ed}{{\eta^{\mathsf d}}}
\newcommand{\edT}{{\eta^{\mathsf d}_T}}

\newcommand{\Npart}{{N_{\mathrm{p}}}}

\newcommand{\rmd}{{\, \mathrm{d}}}

\newtheorem{example}[theorem]{Example}

\newtheorem{remark}[theorem]{Remark}

\newtheorem{assumption}[theorem]{Assumption}

\newcommand{\bA}{{\bm A}}
\newcommand{\bB}{{\bm B}}

\newcommand{\bU}{{\bm U}}
\newcommand{\bUt}{{\bm{\tilde U}}}
\newcommand{\bX}{{\bm X}}
\newcommand{\bXt}{{\bm{\tilde X}}}
\newcommand{\bXb}{{\bm{\bar X}}}
\newcommand{\bY}{{\bm Y}}
\newcommand{\bZ}{{\bm Z}}
\newcommand{\bL}{{\bm\Lambda}}

\newcommand{\Cm}{{\mathcal C}}

\newcommand{\Om}{{\mathcal O}}
\newcommand{\Sm}{{\mathcal S}}

\newcommand{\Eb}{{\mathbb E}}
\newcommand{\Lb}{{\mathbb L}}
\newcommand{\Nb}{{\mathbb N}}
\newcommand{\Pb}{{\mathbb P}}
\newcommand{\Rb}{{\mathbb R}}

\newcommand{\Vb}{{\mathbb V}}
\newcommand{\Xb}{{\mathbb X}}

\newcommand{\Fs}{{\mathscr F}}

\newcommand{\Ms}{{\mathscr M}}
\newcommand{\Ns}{{\mathscr N}}

\newcommand{\Fr}{{\mathrm F}}

\newcommand{\standardNorm}[1]{\left\lVert #1 \right\rVert}

\def\LF{\mathcal{L}}

\allowdisplaybreaks

\begin{document}

\maketitle

\begin{abstract}
	To study the nonlinear properties of complex natural phenomena, the evolution of the quantity of interest can be often represented by systems of coupled nonlinear stochastic differential equations (SDEs). These SDEs typically contain several parameters which have to be chosen carefully to match the experimental data and to validate the effectiveness of the model. In the present paper the calibration of these parameters is described by nonlinear SDE-constrained optimization problems. In the optimize-before-discretize setting a rigorous analysis is carried out to ensure the existence of optimal solutions and to derive necessary first-order optimality conditions. For the numerical solution a Monte-Carlo method is applied using parallelization strategies to compensate for the high computational time. In the numerical examples an Ornstein-Uhlenbeck and a stochastic Prandtl-Tomlinson bath model are considered. 
\end{abstract}




\section{Introduction}

Natural processes inherit noise and uncertainties and therefore are often modeled using stochastic differential equations (SDEs). Such equations consist of a deterministic part and a stochastic part, where the latter one is usually modeled by Brownian motions; see, e.g., \cite{Kampen1976SDE,Kampen1981StochasticProcesses}. The application of SDEs ranges from physics and biology to finance; cf. \cite{Cresson2018DerivationSDEBiology, Jagla2018SPT, SachsSchu2013GradienComputationModelCalibration}, for instance. All these SDEs include parameters that need to be calibrated. For this calibration (also called parameter identification), usually real-world data is used and parameters are searched that lead to the best agreement of measured and predicted data. Therefore, parameter identification and optimal control of such processes are of interest to many communities and have been the subject of extensive research. 

The main motivation for our work is the calibration of models that investigate the rheological properties of fluids using the microrheology technique. Specifically, the study of the behavior of fluids using a so-called tracer particle suspended in a fluid.
Key tools to study properties of complex fluids are nonlinear bath models, for example, the Stochastic Prandtl-Tomlinson (SPT) model \cite{Jain2021microStochPrandtlTomlison,Jain2021twoStepRheolog,Muller2020NonLinBath}.
While linear models have been investigated in the Markovian framework \cite{Dhont1996introductionColloids, Sekimoto1998MarkovianLangevin,TanimuraYoshitakaWolynes1991QuantumClassicalFPMarkov}, less is known for the nonlinear case together with memory effects in which the Markovian framework is no longer applicable.
To study the properties of fluids, systems of coupled nonlinear SDEs representing the movement of particles are studied. 
The SDEs in this case typically possess the structure of (generalized) nonlinear Langevin equations \cite{Kawasaki1973DerivationGLE, Kubo1966FluctuationDissipation, Zwanzig1973NonlinearGeneralizedLangevin}.

In this paper, we express the calibration problems as SDE-constrained optimization problems of the form
\begin{align}
	\label{Intro:P}
	\left\{
	\begin{aligned}
		&\min J(\bX,u)=j(\bX)+j_T(\bX(T))+\frac{\kappa}{2}\,{\|u\|}_\U^2\\
		&\hspace{1mm}\text{subject to }\bX\text{ satisfies an SDE on }[0,T]\text{ for parameter or control }u\in\Uad\subset\U
	\end{aligned}
	\right.
\end{align}
with $\kappa\ge0$, a parameter or control space $\U$ and with properly chosen functionals $j$, $j_T$ which will be specified in \cref{sec:stateEquation,sec:Optimization_problem}, respectively. In particular, the tracking type functional
\begin{align}
	\label{Intro:CostCal} 
	\begin{aligned}
		J(\bX,u)&=\frac{1}{2}\bigg(\int_0^T{\|\Cm(\Eb[\bX(t)])-\cd(t)\|}_{\Rb^{\ell}}^2\,\mathrm dt+{\|\Cm(\Eb[\bX(T)])-\cdT\|}_{\Rb^{\ell}}^2+\kappa\,{\|u\|}_\U^2\bigg)
	\end{aligned}
\end{align}
is included. Notice that the objective defined in \eqref{Intro:CostCal} differs from the one typically used in the stochastic optimal control setting (see, e.g., {\em\cite{Fabbri2017StochasticOptimalControlInInfi,Pham2009StochasticModellingOptimization}}), but is present in literature for calibration problems; see, e.g., {\em\cite{FabrizioMonetti2015MethodologiesCalibration,Estevao2000FunctionalCalibration,Gilli2011calibratingPricingModels,Kaebe2009AdjointMonteCarloCalibra, LoHaslam2008algorithmParamEstSDE}}).

To solve \eqref{Intro:P} let us mention three different approaches:
\begin{itemize}
	\item One can lift the problem to the level of partial differential equations (PDEs). 
    This lift can be viewed as taking the limit in several senses; see, e.g., \cite{PaulTrelat2024microscopicMacroMeso} for a review in the deterministic setting.
    One way is to define a probability density function that contains information on the probability of the state of the model being in a certain configuration at a certain timestep. The PDE that governs the evolution of the probability density is given by the Fokker-Planck equation \cite{BreitenbachBorzi2020PMPFPOCP, AnnunziatoBorzi2018FPFrameworkSOC}. 
    In this setting, one can apply tools from PDE-constrained optimization; see, e.g., \cite{HPUU08,Treoltzsch2010OCPPDE}.
    If one wants to account for the structure of the space of probability measures,
		we refer to \cite{BensoussanFrehseYam2013MeanFieldGamesControl, BonnetFrankowska2021NecessaryOCPWasserstein, Daudin2023OCPStateConstraintsFP,Frankowska2019NecessaryConditionsSDEStateConstraint}. 
    However, one encounters certain challenges with this lift to the level of PDEs. 
    One of them is the exponential growth of numerical complexity with the dimension of the model since we encounter a high-dimensional PDE.
    This is also known as the \emph{curse of dimensionality}.
	\item In another approach one stays on the microscopic level and directly characterizes the optimal control or calibration using SDEs \cite{Bismut1973ConjugateConvexOSC,Bismut1978IntroOCPStoch, LuZhang2014GeneralPMPBSEE}. Here, no probability density functions come into play and hence no PDEs have to be solved. However, their huge weakness is their quite slow convergence \cite{GrahamTalay2013StochSimMC}.
	\item Further, hybrid methods are utilized, where the models are solved on the microscopic level, but -- to calculate the gradient -- the density functions are assembled. These methods suffer from the time-consuming bottleneck of assembling probability density functions in hybrid methods  \cite{Bartsch2021MOCOKI, Bartsch2020OCPKS}.
\end{itemize}

In the present paper we study \eqref{Intro:P} as an infinite dimensional optimization problem; see, e.g., \cite{HPUU08,Lue69,Treoltzsch2010OCPPDE}. 
The first objective is to derive sufficient conditions for the existence of optimal solutions and to carefully set up the correct framework in terms of solution spaces and control or parameter spaces. 
After we guarantee sufficient differentiability of the cost functional and the constraints, we derive first-order necessary optimality conditions by an adjoint calculus. 
This allows us to characterize an optimal control or optimal parameters; see, e.g., \cite{HPUU08}. 
Up to the authors' knowledge, the rigorous derivation of the reduced gradient for the calibration problem using adjoint calculus is not present in the literature. 
Nevertheless, there are some works heading in a similar direction, in particular, \cite{KosmolPavon1993LagrangeOCP} and very recently \cite{ClevenhausTotzeckEhrhardt2023GradientCalibrationHeston}.

To solve \eqref{Intro:P} numerically, one possibility is to start with deriving first-order necessary optimality conditions for the optimal control problem on the continuous level. After this, the system is discretized and optimization schemes are applied. This approach is also known as \emph{first-optimize-then-discretize}. 
Alternatively, in the \emph{first-discretize-then-optimize} approach \eqref{Intro:P} is discretized first and then the first-order optimality system is derived for the discretized optimization problem. 
Utilizing the Lagrangian framework, we derive an optimality system consisting of the state equation, the adjoint equation and the optimality condition. Both approaches are compared. Moreover, convergence of the value of the discretized cost functional to the value of the continuous one is proved. 
Notice that the exploiting of adjoint-based models and staying on the microscopic level by using Monte-Carlo computations have several advantages, in particular, they are easy to implement and inherit huge flexibility \cite{GrahamTalay2013StochSimMC}.

The paper is structured as follows: In \cref{Sec:Notations} we introduce the notation and recall some preliminaries when dealing with SDEs. In \cref{sec:stateEquation}, a general state equation is introduced that governs the optimization problem that is defined in \cref{sec:Optimization_problem}, where we prove the existence of solutions to the optimization problem. In \cref{sec:Discretization} we introduce a discretization for \eqref{Intro:P} and derive the associated first-order necessary optimality system which is the basis for a gradient-based method utilized for our numerical tests. The Fr\'echet differentiability of the constraints and the objective as well as the first-order optimality system for \eqref{Intro:P} are proved in \cref{sec:Gradient}. The work is completed in \cref{sec:Numerical_experiments}, where we formulate our optimization strategy and validate it using two different examples. At first, we apply our method to the mean-reverting Ornstein-Uhlenbeck process where we try to find a time-dependent control function that drives the evolution of the SDE to follow a desired trajectory.
Then, we solve the calibration model governed by the SPT model with one and two particles.


\section{Notation and preliminaries}
\label{Sec:Notations}

In this section, we introduce our setting and a general model for calibration and stochastic control problems. 
We start by recalling some preliminaries. 
For more details, we refer the reader to \cite{Mao2008SDE, YongZhou1999StochasticControls}.

Let $(\Omega,\Fs,\Pb)$ be a probability space with a non-empty event set $\Omega$, a $\sigma$-algebra $\Fs$ and a probability measure $\Pb:\Fs\to[0,1]$.
For given $d\in\Nb$ let $\bX=(X_1,\ldots,X_d):\Omega\to\Rb^d$ be $\Fs$-measurable. 
Then, $\bX$ is called a \emph{$d$-dimensional random variable} and its \emph{expected value} is defined as
\begin{align*}
	\Eb[\bX]\colonequals\int_\Omega\bX\,\mathrm d\Pb=\int_\Omega\bX(\omega)\,\mathrm d\Pb(\omega) \in \Rb^d.
\end{align*}
A $d$-dimensional \emph{stochastic process} is a family $\{\bX(t)\}_{t\ge 0}$ of $d$-dimensional random variables $\bX(t)\colon \Omega\to \Rb^d$.
A \emph{filtration} of the probability space $(\Omega,\Fs,\Pb)$ is a family $\{\Fs(t)\}_{t\ge 0}$ of $\sigma$-algebras $\Fs(t)\subset\Fs$ with $\Fs(s)\subset \Fs(t)$ for all $s\in[0,t]$.
A $d$-dimensional stochastic process  $\{\bX(t)\}_{t\ge0}$ is called
\emph{adapted} with respect to the filtration $\{\Fs(t)\}_{t\ge0}$ if each $\bX(t)$ is $\Fs(t)$-measurable.
Furthermore, it is called \emph{progressively measurable} if for every $t\ge0$ the function $\bX\vert_{[0,t]}:\Omega \times [0,t]\to\Rb^d,\, (\omega, t)\mapsto \bX(\omega, t)$, is $\Fs(t)\otimes{\mathscr B}([0,t])$-measurable, where $\mathscr B([0,t])$ is the $\sigma$-algebra of all Borel subsets of $[0,t]$ and $\otimes$ denotes the product-$\sigma$-algebra.
The stochastic It\^o-integral will be considered in the $L^2$-framework; cf., e.g., \cite[Section~4.2]{Evans2013SDE}. For $T>0$ we endow $L^2(0,T; \Rb^d)$ with the standard norm
\begin{align*}
	{\|\varphi\|}_{L^2(0,T;\Rb^d)} \colonequals  \bigg( \int_0^T{\|\varphi(t)\|}_{\Rb^d}^2 \,\mathrm dt\bigg)^{1/2},
\end{align*}
with the Euclidean norm $\|\cdot\|_{\Rb^d}$ in $\Rb^d$. In the same way, the space $L^2(\Omega;\Rb^d)$ of all equivalence classes of square-integrable random variables is endowed with the norm
\begin{align*}
	{\|\bX_\circ\|}_{L^2(\Omega;\Rb^d)}\colonequals\left(
	\Eb\left[{\|\bX_\circ\|}_{\Rb^d}^2\right]\right)^{1/2}=\left( \int_{\Omega}{\|\bX_\circ(\omega)\|}_{\Rb^d}^2 \,\mathrm d\Pb(\omega)\right)^{1/2}.
\end{align*}
Finally, $L^2(\Omega\times (0,T);\Rb^d)\cong L^2(\Omega; L^2(0,T;\Rb^d))$ denotes the space of  equivalence classes of $\Rb^d$-valued $\Fs\otimes{\mathscr B}([0,T])$-measurable and square-integrable processes $\bX$, with the norm 
\begin{align*}
	{\|\bX\|}_{L^2(\Omega\times (0,T);\Rb^d)}
	\colonequals\bigg(\Eb\bigg[\int_0^T{\|\bX(t)\|}_{\Rb^d}^2\,\mathrm dt\bigg]\bigg)^{1/2}.
\end{align*}
We define 
\begin{align*}
	\Lb^2_{\Fs}(\Rb^d)\colonequals\{ \bX\in L^2(\Omega\times (0,T);\Rb^d): \bX\text{ 
		progressively measurable}\}.
\end{align*}
More precisely, $\Lb^2_{\Fs}(\Rb^d)$ consists of all equivalence classes in 
$L^2(\Omega\times(0,T);\Rb^d)$ which contain at least one progressively measurable 
process. The space $\Lb^2_{\Fs}(\Rb^d)$ will be considered as a subspace of $L^2(\Omega\times(0,T);\Rb^d)$ supplied by the topology of $L^2(\Omega\times(0,T);\Rb^d)$, i.e., we set
\begin{align*}
	{\|\bX\|}_{\Lb^2_{\Fs}(\Rb^d)}\colonequals {\|\bX\|}_{L^2(\Omega\times (0,T);\Rb^d)}.
\end{align*}
In addition to the above Hilbert spaces, we will also consider processes with continuous paths, endowed with the $\sup$-norm. For this, we define the norm
\begin{align*}
	{\|\bX\|}_{C([0,T]; L^2(\Omega;\Rb^d))} := \sup_{t\in [0,T]} {\|\bX(t)\|}_{L^2(\Omega;\Rb^d)}= \sup_{t\in [0,T]}\left(\Eb\big[{\|\bX(t)\|}^2_{\Rb^d}\big]\right)^{1/2}
\end{align*}
on the Banach space of all processes $\bX$ which are continuous functions of $t$ with values in the space $L^2(\Omega;\Rb^d)$.

Let now $m\in\Nb$ be given.
By $\{\bB(t)=(B_i(t))_{1\le i\le m}\,\vert\,t\ge0\}$ we denote an $m$-dimensional Brownian motion (or Wiener process); see, e.g., \cite[Definition 4.1, Chapter 1]{Mao2008SDE}.
We introduce the family of $\sigma$-algebras $\{\Fs^\bB(t)\}_{t\ge0}$, where $\Fs^\bB(t)$ is generated by $\{\bm B(s)\,\vert\,0\le s\le t\}$ for any $t\ge 0$. Recall that $\{\Fs^\bB(t)\}_{t\ge0}$ is said to be the \emph{natural filtration} generated by the Brownian motion $\{\bB(t)\}_{t\ge0}$. By $\{\Fs(t)\}_{t\ge0}$ we introduce the filtration given by $\Fs(t)=\Fs^\bB(t)\cup\Ns$, where the set $\Ns$ is the set of all $\Pb$-null sets.

We fix a sub-$\sigma$-algebra $\Fs_\circ\subset\Fs$ which is independent of $\Fs(T)$, and define the extended filtration $\Fs_\circ(t)\colonequals\sigma(\Fs(t) \cup \Fs_\circ )$. Let $\X$ denote the vector space  of all $\{\Fs_\circ(t)\}_{t \in [0,T]}$-adapted  measurable processes of the form
\begin{align}
	\label{Yprocess}
	\bY(t)=\Phi(\bY_\circ,\mathtt a,\mathtt b)\colonequals\bY_\circ+\int_0^t \mathtt a(s)\,\mathrm ds+\int_0^t \mathtt b(s)\,\mathrm d\bB(s)\quad\text{for every }t\in[0,T]
\end{align}
satisfying 
\begin{align}
	{\|\bY\|}_\X\colonequals \left(
	\Eb\left[{\|\bY_\circ\|}_{\Rb^d}^2\right]
	+\Eb\left[\int_0^T{\|\mathtt a(s)\|}_{\Rb^d}^2+{\|\mathtt b(s)\|}_\Fr^2\,\mathrm ds\right]\right)^{1/2}<\infty
	\label{eq:definition_norm}
\end{align}
with $\bY_\circ\in L^2(\Omega;\Rb^d)$ being $\Fs_\circ$-measurable, $\mathtt a\in\Lb^2_{\Fs}(\Rb^d)$ and $\mathtt b=[\mathtt b_1\vert\ldots\vert\mathtt b_m]\in \Lb^2_{\Fs}(\Rb^{d\times m})$. In \eqref{eq:definition_norm}, we denote by $\|\cdot\|_\Fr$ the Frobenius (matrix) norm induced by the inner product
\begin{align}
	{\langle \mathfrak b,\tilde{\mathfrak b} \rangle}_\Fr
	= \sum_{j=1}^m \mathfrak b_j^\top \tilde{\mathfrak b}_j 
	= \sum_{i=1}^d\sum_{j = 1}^m\mathfrak b_{ij}\tilde{\mathfrak b}_{ij}\quad\text{for }\mathfrak b=[\mathfrak b_1\vert\ldots\vert\mathfrak b_m],\tilde{\mathfrak b}=[\tilde{\mathfrak b}_1\vert\ldots\vert\tilde{\mathfrak b}_m]\in\Rb^{d\times m}.
	\label{eq:definition_frobenius_scalarproduct}
\end{align}
Then the stochastic process $\bY$ in \eqref{Yprocess} is adapted  with respect to the filtration $\{\Fs_\circ(t)\}_{t \in [0,T]}$, and for almost all $\omega\in\Omega$ the associated path $[0,T]\ni t\mapsto \bY(\omega,t)\in\Rb^d$ is continuous. 
In particular, $\bY$ is progressively measurable; cf. \cite[Proposition~1.13]{Karatzas1988BMSDE}. 
We consider on the space $\X$ the inner product
\begin{align}
	{\langle\bY,\bZ\rangle}_\X = \Eb\left[\bY_\circ^\top\bZ_\circ\right]
	+\Eb\bigg[
	\int_0^T\mathtt a(s)^\top\tilde{\mathtt a}(s)\,\mathrm ds+ \int_0^T{\langle\mathtt b(s), \tilde{\mathtt b}(s)\rangle}_\Fr\, \mathrm d s\bigg],
	\label{eq:inner_product}
\end{align}
where $\bZ := \Phi(\bZ_\circ,  \tilde{\mathtt a}, \tilde{\mathtt b} )$ with $\bZ_\circ\in L^2(\Omega;\Rb^d)$ being $\Fs_\circ$-measurable and with processes $\tilde{\mathtt a}\in\Lb_\Fs^2(\mathbb R^d)$, $\tilde{\mathtt b}=[\tilde{\mathtt b}_1\vert\ldots\vert\tilde{\mathtt b}_m]\in\Lb_\Fs^2(\Rb^{d\times m})$. 
Throughout this work, the symbol `$\top$' stands for the transpose of vectors or matrices. 

To show that $\X$ is a Hilbert space, we consider the product space
\begin{equation}\label{eq-defXb}
	\Xb=\big\{ (\bY_\circ,\mathtt a,\mathtt b)\in L^2(\Omega;\Rb^d)\times\Lb^2_\mathscr F(\Rb^d)\times\Lb^2_\mathscr F(\Rb^{d\times m}): \bY_\circ 
	\text{ is $\Fs_0$-measurable}\big\} .
\end{equation}
In \cref{Lem:IsoX} we show that the space $\X$ is a Hilbert space with the inner product given by the right-hand side of \eqref{eq:inner_product}, and the induced norm is given by the right-hand side of \eqref{eq:definition_norm}. In the proof of the lemma we utilize the notion of a martingale which is defined, e.g., in \cite[Section~2.7]{Evans2013SDE}.

\begin{lemma}
	\label{Lem:IsoX}
	The map $\Phi\colon \Xb\to \X$ is injective, i.e., the representation of $\bY\in\X$ of the form \eqref{Yprocess} is unique, and therefore the norm \eqref{eq:definition_norm} is well defined in  $\X$. With this norm in $\X$, $\Phi$ is an isometric isomorphism between $\Xb$ and $\X$. In particular, $\X$ is a Hilbert space. 
\end{lemma} 

\begin{proof}
	Assume $\bY = \Phi(\bY_\circ,\mathrm a,\mathrm b)=0$. 
	Then $\bY_\circ = \bY(0) =0$ (note that $\bY$ has continuous paths), and  
	\begin{align*}
		M_t\colonequals\int_0^t \mathtt b(s)\,\mathrm d\bB(s) = - \int_0^t \mathtt a(s)\,\mathrm ds,\quad t\in[0,T],
	\end{align*}
	is a continuous martingale with finite variation and therefore constant, see \cite[Theorem~III.12]{Protter05}, and we obtain $M_t=M_0=0$ for all $t\in [0,T]$. As the integral representation of square integrable martingales with continuous paths is unique (see \cite[Theorem~4.15]{Karatzas1988BMSDE}), this yields $\mathtt b=0$ in $\Lb^2_{\Fs}(\Rb^{d\times m})$. Moreover, we have $-M_t = \int_0^t \mathtt a(s)\,\mathrm ds = 0$ for all $t\in[0,T]$ which implies $\mathtt a=0$ in $\Lb^2_{\Fs}(\Rb^{d})$. Consequently,  the map $\Phi$ is injective, which implies that the norm in \eqref{eq:definition_norm} is well defined. 
	By construction, $\Phi$ is an isometric isomorphism between $\Xb$ and $\X$, which also yields that $\X$ is a Hilbert space.
\end{proof}

As the following lemma shows, we obtain continuous embeddings of $\X$ into standard spaces
of processes. Part 2) of the  lemma will be useful later in \cref{Sec:RedGradCont}, when we derive  backward equations. 

\begin{lemma}
	\label{Le:ExtXb}
	\begin{enumerate}
		\item [\em 1)] For $C:= \sqrt{3\max\{T,1\}}$ we have 
		\begin{align*}
			{\|\bY\|}_{C([0,T]; L^2(\Omega;\Rb^d))}\le C {\|\bY\|}_{\X}\quad\text{for every }\bY \in \X.    
		\end{align*}
		In particular, 
		\begin{equation}
			\label{embedding-X}
			\X\subset C([0,T]; L^2(\Omega;\Rb^d))\subset L^2(\Omega\times (0,T); \Rb^d)
		\end{equation}
		holds with continuous embeddings, and we obtain $\X\subset \Lb^2_\Fs(\mathbb R^d)$.
		\item [\em 2)] Let $\bL\colonequals\Phi(\bL_\circ,  \tilde{\mathtt a}, \tilde{\mathtt b} )\in\X$. Then the map
		\begin{equation}\label{def-functional}
			\ell_{\bL}\colon \X\to \Rb,\quad  \bY \mapsto \Eb\bigg[ \bY(T)^\top \bL(T)-\int_0^T \bY(s)^\top\tilde{\mathtt a}(s)\,\mathrm ds\bigg]
		\end{equation}
		defines a continuous linear functional on $\X$. Moreover, for $\bY = \Phi(\bY_\circ,\mathtt a, \mathtt b)\in\X$, we have
		\begin{equation}\label{formula-functional}
			\ell_{\bL}(\bY) = \big\langle (\bY_\circ, \mathtt a,\mathtt b), (\bL_\circ,\bL,\tilde{\mathtt b})\big\rangle_{\Xb}\,, 
		\end{equation}
		where $\langle\cdot\,,\cdot\rangle_{\Xb}$ stands for the inner product in $\Xb$.
	\end{enumerate}
\end{lemma}

\begin{proof}
	\begin{enumerate}
		\item [1)] Let $\bY\in\X$ be given.  Then \eqref{Yprocess} holds with a random vector $\bX_\circ\in L^2(\Omega;\Rb^d)$ and processes $\mathtt a\in\Lb_\Fs^2(\Rb^d)$ and $\mathtt b\in\Lb_\Fs^2(\Rb^{d\times m})$. From \eqref{Yprocess} and the inequality $(s_1+s_2+s_3)^2\leq 3(s_1^2+s_2^2+s_3^2)$ for positive $s_1,s_2,s_3$, we obtain for every $t\in [0,T]$
		\begin{align*}
			\Eb\left[{\|\bY(t)\|}_{\Rb^d}^2\right] \le 3\Big( \Eb \left[{\|\bY_\circ\|}_{\Rb^d}^2\right]+\Eb\Big[ \Big\|\int_0^t \texttt a(s)\,\mathrm ds\Big\|_{\Rb^d}^2\Big]+ \Eb\Big[ \Big\|\int_0^t \texttt b(s)\,\mathrm d\bB(s)\Big\|_{\Rb^d}^2\Big]\Big).    
		\end{align*}
		An application of the Cauchy-Schwarz inequality yields
		\begin{align*}
			\Big\|\int_0^t \texttt a(s)\,\mathrm ds\Big\|_{\Rb^d}^2 \le \Big( \int_0^T{\|\texttt a(s)\|}_{\Rb^d}\,\mathrm ds\Big)^2\le T\int_0^T{\|\texttt a(s)\|}_{\Rb^d}^2\,\mathrm ds.
		\end{align*}
		For the stochastic integral, we use the It\^o-isometry (cf., e.g., \cite[Chapter~1, Theorem~5.21]{Mao2008SDE})
		\begin{align*}
			\Eb \Big[ \Big\|\int_0^t \texttt b(s)\,\mathrm d\bB(s)\Big\|_{\Rb^d}^2\Big] = \Eb \Big[\int_0^t \|\texttt b(s)\|_\Fr^2\,\mathrm d s\Big]\le \Eb \Big[\int_0^T \|\texttt b(s)\|_\Fr^2\,\mathrm d s\Big].
		\end{align*}
		Therefore, recalling that $\bY$ has continuous paths, we obtain
		\begin{align*}
			{\|\bY\|}_{C([0,T]; L^2(\Omega;\Rb^d))}^2=\sup_{t\in[0,T]}\Eb\big[{\|\bY(t)\|}_{\Rb^d}^2\big] \le 3 \max\{T,1\} \,{\|\bY\|}_{\X}^2    
		\end{align*}
		which shows the first embedding in \eqref{embedding-X}. 
        The second embedding follows from the obvious estimate
		\begin{align*}
			{\|\bX\|}_{L^2(\Omega\times (0,T); \Rb^d)}\le \sqrt{T}\,{\|\bX\|}_{C([0,T]; L^2(\Omega;\Rb^d))}.
		\end{align*}
		As every element of $\X$ is progressively measurable, we obtain $\X\subset \Lb^2_\Fs(\mathbb R^d)$.
		\item [2)] By the Cauchy-Schwarz inequality and the embeddings from 1), we can estimate 
		\begin{align*}
			|\ell_{\bL}(\bY)|&\le {\|\bY\|}_{C([0,T];L^2(\Omega;\Rb^d))}{\|\bL\|}_{C([0,T]; L^2(\Omega;\Rb^d))}+{\|\bY\|}_{\Lb^2_\Fs(\mathbb R^d)}{\|\bL\|}_{\Lb^2_\Fs(\mathbb R^d)}\\
			&\le C'{\|\bY\|}_{\X}{\|\bL\|}_{\X} 
		\end{align*}
		for some constant $C'>0$. Therefore, $\ell_{\bL}$ is a continuous linear functional on $\X$. The proof of \eqref{formula-functional} is based on the It\^o lemma applied on $\bY(T)^\top \bL(T)$, see \cite[Lemma~2.2.16]{Gross2015applications}, which yields
		\begin{align*}
			\bY(T)^\top \bL(T)& = \bY(0)^\top \bL(0) + \int_0^T \Big( \bY(t)^\top \tilde{\mathtt a}(t)
			+ \mathtt a(t)^\top\bL(t)+{\langle\mathtt b(t), \tilde{\mathtt b}(t)\rangle}_\Fr
			\Big) \,\mathrm dt \\
			&\quad+\int_0^T\big(\bY(t)^\top  \tilde{\mathtt b}(t)+ \bL(t)^\top \mathtt b(t)\big)\, \mathrm d\bB(t).
		\end{align*}
		Taking the expectation and noting that the expectation of the stochastic integral equals zero, we obtain
		\begin{align*}
			\ell_{\bL}(\bY) = \Eb\Big[\bY_\circ^\top \bL_\circ+ \int_0^T \Big(\mathtt a(t)^\top \bL(t)+{\langle\mathtt b(t), \tilde{\mathtt b}(t)\rangle}_\Fr\Big)\,\mathrm dt\Big],
		\end{align*}
		which shows part 2).
	\end{enumerate}
\end{proof}

\begin{remark}\label{Re:EstExpVal}
	Let $\bY\in\X$ hold. Then,
	\begin{align*}
		\sup_{t\in [0,T]}{\|\Eb[\bY(t)]\|}_{\Rb^d}\le\sup_{t\in [0,T]} \Eb\left[{\|\bY(t)\|}_{\Rb^d}\right]\le\sup_{t\in [0,T]}\Big(\Eb\left[{\|\bY(t)\|}_{\Rb^d}^2\right]\Big)^{1/2}\le  C\,{\|\bY\|}_\X
	\end{align*}
	follows from H\"older's inequality and \cref{Le:ExtXb}.
\end{remark}

\section{The state equation}
\label{sec:stateEquation}

For $r \in \Nb$ we introduce the Hilbert space $\U\coloneqq L^2(0,T;\Rb^r)$ and the set of \emph{admissible (deterministic) controls} which is given as
\begin{align}
	\Uad\colonequals\big\{u\in\U\;\big| \; u(t)\in[\ua,\ub]\text{ for almost all (f.a.a.) }t\in[0,T]\big\}
	\label{eq:admissable_controls}
\end{align}
with $\ua,\ub\in\Rb^r$ satisfying $\ua \le\ub $ in $\Rb^r$ (i.e., component by component) and $[\ua,\ub]\colonequals\{u\in\Rb^r\,\big|\,\ua\le u\le\ub\text{ in }\Rb^r\}$.

\begin{remark}
	\label{rem:Uad_nonemptyConvexClosed}
	\begin{enumerate}[label=\text{\textup{\arabic*)}}]
		\item Notice that $\Uad$ defined in \eqref{eq:admissable_controls} is nonempty, convex and closed.
		\item \label{item:remark_calibration_example} 
		In one numerical example carried out in \cref{sec:SPT_example}, we consider calibration problems, where the unknowns are time-independent parameters. In that case we have $\U^\mathrm c=\Rb^r$ and $\UadCalib=[\ua,\ub]\subset\U^\mathrm c$, i.e., the set $\UadCalib$ is even compact.
	\end{enumerate}
\end{remark}

Next let us start with the introduction of the coefficient functions that constitute the equation, where the following standard hypothesis on the coefficient functions is taken from \cite[p.~44]{YongZhou1999StochasticControls}.
\begin{assumption}
	\label{Assumption:coefficients}
	We are given measurable \emph{drift} and \emph{diffusion coefficient functions}
	\begin{align*}
		a:\Rb^d\times\Rb^r\times[0,T]\to\Rb^d\quad\text{and}\quad b:\Rb^d\times \Rb^r\times[0,T]\to\Rb^{d\times m},
	\end{align*}
	respectively. 
	Further, there exists a (Lipschitz) constant $L_U>0$ such that for every $x,\tilde x\in\Rb^d$, $u,\tilde u\in[\ua,\ub]$ and f.a.a $t\in[0,T]$ there exists a constant $L_U>0$ such that the \emph{Lipschitz condition}
	\begin{align}
		\label{eq:LipCondition}
		{\|a(x,u,t)-a(\tilde x,\tilde u,t)\|}_{\Rb^d}+{\|b(x,u,t)-b(\tilde x, \tilde u,t)\|}_\Fr\le L_U\,\left({\|x-\tilde x\|}_{\Rb^d}+{\|u-\tilde u\|}_{\Rb^r}\right)
	\end{align}
	holds. 
	Moreover,
	\begin{align}
		\label{Assumption:coefficients_integral}
		\int_0^T{\|a(0,u(t),t)\|}^2_{\Rb^d}+{\|b(0,u(t),t)\|}_\Fr^2\,\mathrm dt\le L_U^2        
	\end{align}	
	is valid for all $u\in \Uad$.
\end{assumption}

Suppose that $\bX_\circ\in L^2(\Omega;\Rb^d)$ is a fixed given initial condition being $\Fs_\circ$-measurable. For a given (deterministic) input  function  $u\in\Uad$, the associated stochastic process $\bX$ is an $\Rb^d$-valued solution to the stochastic differential equation (SDE):
\begin{align}
	\left\{
	\begin{aligned}
		\mathrm d\bm X(t)&=a(\bm X(t),u(t),t)\,\mathrm dt+b(\bm X(t),u(t),t)\,\mathrm d\bm B(t)\quad\text{for all }t \in(0,T],\\
		\bm X(0)&=\bm X_\circ.
	\end{aligned}
	\right.
	\label{SDE}
\end{align}
The next definition specifies the concept of a solution and uniqueness in our context.

\begin{definition}
	\label{Definition:SDE}
	\begin{enumerate}[label={\em\arabic*)}]
		\item If $\bX$ belongs to $\X$ and satisfies
		\begin{align*}
			\bm X(t)=\bm X_\circ+\int_0^ta(\bm X(s),u(s),s)\,\mathrm ds+\int_0^tb(\bm X(s),u(s),s)\,\mathrm d\bm B(s)\quad\text{in }L^2(\Omega;\Rb^d)
		\end{align*}
		for all $t\in[0,T]$, we call $\bX$ a \emph{solution} to \eqref{SDE}.
		\item We say that a solution $\bX\in\X$ to \eqref{SDE} is \emph{unique} if for any other solution $\bm{\tilde X}\in\X$ to \eqref{SDE} we have
		\begin{align*}
			\mathbb P\left(\big\{\omega\in\Omega\,\big|\,\bm X(t)=\bm{\tilde X}(t)\text{ for all }t\in[0,T]\big\}\right)=1.
		\end{align*}
	\end{enumerate}
\end{definition}

\begin{remark}
	\label{rem:linear_example}
	For the example introduced in \cref{rem:Uad_nonemptyConvexClosed}-\ref{item:remark_calibration_example}, the SDE \eqref{SDE} is given as
	\begin{align}
		\left\{
		\begin{aligned}
			\mathrm d\bm X(t)&=a(\bm X(t),u,t)\,\mathrm dt+b(\bm X(t),u,t)\,\mathrm d\bm B(t)\quad\text{for all }t \in(0,T],\\
			\bm X(0)&=\bm X_\circ
		\end{aligned}
		\right.
		\label{SDE-Cal}
	\end{align}
	for a given (parameter) vector $u\in\UadCalib$. 
	In that case, its solution $\bX\in\X$ satisfies
	\begin{align*}
		\bm X(t)=\bm X_\circ+\int_0^ta(\bm X(s),u,s)\,\mathrm ds+\int_0^tb(\bm X(s),u,s)\,\mathrm d\bm B(s)\quad\text{in }L^2(\Omega;\Rb^d)\text{ for all }t\in[0,T].
	\end{align*}
\end{remark}

Now, we recall the existence result from \cite[Corollary~1.6.4]{ YongZhou1999StochasticControls}.

\begin{theorem}
	\label{Theorem:ExUni}
	Let \cref{Assumption:coefficients} hold and $\bX_\circ\in L^2(\Omega;\Rb^d)$ be an $\Fs_\circ$-measurable random variable. Then for every $u\in \Uad$ there exists a unique solution $\bX\in\X$ of \eqref{SDE}.
\end{theorem}

\begin{remark}
	To prove \cref{Theorem:ExUni} we do not need the Lipschitz continuity of the coefficient functions $a$ and $b$ with respect to the control variable $u$. Moreover, for the unique solvability, the Lipschitz condition (with respect to $x$) can be weakened by a local one; cf. {\em\cite[Chapter~2, Theorem~3.4]{Mao2008SDE}}.
\end{remark}

Let $\bX_\circ\in L^2(\Omega;\Rb^d)$ be an $\Fs_\circ$-measurable random variable. 
Due to \cref{Theorem:ExUni} we introduce the non-linear solution operator
\begin{align}
	\label{OpSm}
	\Sm:\Uad\to\X,\quad\bX=\Sm(u)\text{ is the unique solution to \eqref{SDE} for }u\in\Uad.
\end{align}

\begin{theorem}
	\label{Theorem:Apriori}
	Let \cref{Assumption:coefficients} hold and $\bX_\circ\in L^2(\Omega;\Rb^d)$ be an $\Fs_\circ$-measurable random variable. 
	Then, for all $u,\tilde u\in \Uad$, we have 
	\begin{align}
		\label{Lipschitz1}
		{\| \Sm(u)-\Sm(\tilde u)\|}_{C([0,T]; L^2(\Omega;\Rb^d))} & \le C_1\,{\|u-\tilde u\|}_{\U},\\
		\label{Lipschitz2}
		{\|\Sm(u)\|}_{C([0,T]; L^2(\Omega;\Rb^d))}&\le C_2
	\end{align}
	with non-negative constants $C_1 = C_1(L_U,T)$ and $C_2 = C_2(L_U,T,\bX_\circ)$. 
\end{theorem}

\begin{proof}
	To show the estimates \eqref{Lipschitz1} and \eqref{Lipschitz2}, we follow a standard Gronwall approach (see \cite[proof of Theorem~5.2.1]{Oksendal2000SDEIntroApplication}). For this, let $u,\tilde u\in \Uad$ be given. We set $\bX:=\Sm(u)$ and $\bm{\tilde X}:=\Sm(\tilde u)$. 
	In the same way as in the proof of \cref{Le:ExtXb}, we obtain for fixed $t\in [0,T]$  
	\begin{align}
		\label{Lemma:Est-1}
		\begin{aligned}
			\Eb\Big[{\|\bX(t)-\bm{\tilde X}(t)\|}_{\Rb^d}^2\Big]&\le 2\,\Eb\bigg[\Big\|\int_0^ta(\bX(s),u(s),s)-a(\bm{\tilde X}(t),\tilde u(s),s)\,\mathrm d s\Big\|_{\Rb^d}^2\bigg]\\
			&\quad+2\,\mathbb E\bigg[\Big\|\int_0^tb(\bm X(s),u(s),s)-b(\bm{\tilde X}(s),\tilde u(s),s)\,\mathrm d\bB(s)\Big\|_{\Rb^d}^2\bigg] \\
			\qquad\qquad\qquad&\le 2T \,\Eb\bigg[\int_0^t {\|a(\bX(s),u(s),s)-a(\bm{\tilde X}(t),\tilde u(s),s)\|}_{\Rb^d}^2\,\mathrm d s\bigg]\\
			&\quad+ 2 \,\Eb\bigg[\int_0^t{\|b(\bX(s),u(s),s)-b(\bm{\tilde X}(s),\tilde u(s),s)\|}^2_\Fr\,\mathrm ds\bigg].
		\end{aligned}
	\end{align}
	Now the Lipschitz condition of \cref{Assumption:coefficients} yields 
	\begin{align*}
		&\int_0^t{\|a(\bX(s),u(s),s)-a(\bm{\tilde X}(t),\tilde u(s),s)\|}_{\Rb^d}^2\,\mathrm ds
		\\
		&\qquad\qquad
		\le \int_0^t\Big(L_U\big({\|\bX(s)-\bm{\tilde X}(s)\|}_{\Rb^d}+{\|u(s)-\tilde u(s)\|}_{\Rb^r}\big)\Big)^2\,\mathrm ds
		\\
		&\qquad\qquad
		\le 2L_U^2\int_0^t{\|\bX(s)-\bm{\tilde X}(s)\|}_{\Rb^d}^2+{\|u(s)-\tilde u(s)\|}_{\Rb^r}^2\,\mathrm ds
	\end{align*}
	for $t\in [0,T]$. 
	We proceed analogously for the second term on the right-hand side of \eqref{Lemma:Est-1}. 
	Consequently, we obtain
	\begin{align*}
		\Eb\left[{\|\bX(t)-\bm{\tilde X}(t)\|}_{\Rb^d}^2\right]
		\le c_1\int_0^t\Eb\left[{\|\bX(s)-\bm{\tilde X}(s)\|}_{\Rb^d}^2\right]+{\|u(s)-\tilde u(s)\|}_{\Rb^r}^2\,\mathrm ds
	\end{align*}
	for $t\in [0,T]$ and $c_1=4\max\{T,1\}L_U^2$. 
	Now, an application of Gronwall's inequality (cf., e.g., \cite[p.~92]{Evans2013SDE}) gives
	\begin{align}
		\label{ProfLemma}
		\Eb\left[{\|\bX(t)-\bm{\tilde X}(t)\|}_{\Rb^d}^2\right]\le c_1e^{\int_0^tc_1\,\mathrm ds}\int_0^T{\|u(s)-\tilde u(s)\|}_{\Rb^r}^2\,\mathrm ds 
	\end{align}
	for $t\in [0,T]$, which yields \eqref{Lipschitz1} with  $C_1(L_U,T) = \sqrt{c_1e^{c_1T}}$.\hfill\\
	To show \eqref{Lipschitz2} we proceed similarly as in the proof of \cref{Le:ExtXb} and utilize again \cref{Assumption:coefficients}. 
	Note that
	\begin{align*}
		\Eb\big[{\|\bX(t)\|}_{\Rb^d}^2\big]&\le3\bigg(\Eb\big[{\|\bX_\circ\|}_{\Rb^d}^2\big]+T\Eb\bigg[\int_0^t {\|a(\bX(s),u(s),s)\|}_{\Rb^d}^2\,\mathrm d s\bigg]\\
		&\hspace{8mm}+\Eb\bigg[\int_0^t{\|b(\bX(s),u(s),s)\|}^2_\Fr\,\mathrm ds\bigg]\bigg)
	\end{align*}
	for $t\in [0,T]$. 
	For brevity, we use the notation $\mathfrak a={\|a(0,u(\cdot),\cdot\,)\|}_{\Rb^d}$ and $\mathfrak b={\|b(0,u(\cdot),\cdot\,)\|}_\Fr$. 
	Note that
	\begin{align}
		\label{a-Estimate}
		\begin{aligned}
			&\Eb\bigg[\int_0^t {\|a(\bX(s),u(s),s)\|}_{\Rb^d}^2\,\mathrm d s\bigg]\\
			&\le 2\Eb\bigg[\int_0^t {\|a(\bX(s),u(s),s)-a(0,u(s),s)\|}_{\Rb^d}^2+{\|a(0,u(s),s)\|}_{\Rb^d}^2\,\mathrm d s\bigg]\\
			&\le 2\Eb\bigg[\int_0^tL_U^2\,{\|\bX(t)\|}_{\Rb^d}^2+|\mathfrak a(s)|^2\,\mathrm d s\bigg]=2\int_0^tL_U^2\Eb\big[{\|\bX(s)\|}_{\Rb^d}^2\big]+|\mathfrak a(s)|^2\,\mathrm ds.
		\end{aligned}
	\end{align}
	Similarly, we infer that
	\begin{align}
		\label{b-Estimate}
		\Eb\bigg[\int_0^t {\|b(\bX(s),u(s),s)\|}_\Fr^2\,\mathrm d s\bigg]\le 2\int_0^tL_U^2\Eb\big[{\|\bX(s)\|}_{\Rb^d}^2\big]+{|\mathfrak b(s)|}^2\,\mathrm ds.
	\end{align}
	By \cref{Assumption:coefficients}, we obtain 
	\begin{align*}
		\Eb\big[{\|\bX(t)\|}_{\Rb^d}^2\big]&\le3\bigg(\Eb\big[{\|\bX_\circ\|}_{\Rb^d}^2\big]+2\max\{T,1\}\int_0^T|\mathfrak a(s)|^2+{|\mathfrak b(s)|}^2\,\mathrm ds\\
		&\hspace{8mm}+2(T+1)L_U^2\int_0^t\Eb\big[{\|\bX(s)\|}_{\Rb^d}^2\big]\,\mathrm ds\bigg)\\
		&\le3\Eb\big[{\|\bX_\circ\|}_{\Rb^d}^2\big]+6\max\{T,1\}L_U^2+\hat c_2\int_0^t\Eb\big[{\|\bX(s)\|}_{\Rb^d}^2\big]\,\mathrm ds\\
		&\le\tilde c_2+\hat c_2\int_0^t\Eb\big[{\|\bX(s)\|}_{\Rb^d}^2\big]\,\mathrm ds
	\end{align*}
	for $t\in[0,T]$ with $\hat c_2=6(T+1)L_U^2$ and $\tilde c_2=3\Eb[\|\bX_\circ\|_{\Rb^d}^2]+6\max\{T,1\}L_U^2$. 
	By Gronwall's inequality and \cref{Assumption:coefficients}, this implies
	\begin{align*}
		\Eb\big[{\|\bX(t)\|}_{\Rb^d}^2\big]\le \tilde c_2\exp\bigg(\int_0^t\hat c_2\mathrm d\tau\bigg)\le\tilde c_2\exp\big(\hat c_2T\big)\equalscolon c_2
	\end{align*}
	where $c_2$ depends on $L_U$, $T$ and $\bX_\circ$. 
	Therefore, \eqref{Lipschitz2} holds with $C_2(L_U,T,\bX_\circ) =c_2^{1/2}$.
\end{proof}

\begin{corollary}
	\label{Lemma:ExUni_Uad}
	In the situation of \cref{Theorem:ExUni}, we have for all $u,\tilde u\in \Uad$  
	\begin{align}
		\label{Lipschitz3}
		{\|\Sm(u)-\Sm(\tilde u)\|}_{\X}&\le C'_1 {\|u-\tilde u\|}_{\U},\\
		{\| \Sm(u)\|}_{\X} & \le C'_2\label{Lipschitz4}
	\end{align}
	with constants $C'_1 = C'_1(L_U,T)$ and $C'_2 = C'_2(L_U,T,\bX_\circ)$. 
	Therefore, $\Sm\colon \Uad\to \X$ is Lip\-schitz continuous and has bounded range.
\end{corollary}

\begin{proof}
	Again we set $\mathfrak a=\|a(0,u(\cdot),\cdot\,)\|_{\Rb^d}$ and $\mathfrak b=\|b(0,u(\cdot),\cdot\,)\|_\Fr$. 
	Utilizing the definition of the norm in $\X$, \eqref{a-Estimate}, \eqref{b-Estimate} and \cref{Assumption:coefficients}, we find that
	\begin{align*}
		{\|\bX\|}_\X^2
		&=\Eb\big[{\|\bX_\circ\|}_{\Rb^d}^2\big]
		+\Eb\bigg[\int_0^T{\|a(\bX(s),u(s),s)\|}_{\Rb^d}^2\,\mathrm ds\bigg]+\Eb\bigg[\int_0^T{\|b(\bX(s),u(s),s)\|}_\Fr^2\,\mathrm ds\bigg]\\
		&\le\Eb\big[{\|\bX_\circ\|}_{\Rb^d}^2\big]+\int_0^T4L_U^2\Eb\big[{\|\bX(s)\|}_{\Rb^d}^2\big]+{|\mathfrak a(s)|}^2+{|\mathfrak b(s)|}^2\,\mathrm ds\\
		&\le \max\{1,4L_U^2\}{\|\bX\|}_{C([0,T]; L^2(\Omega;\Rb^d))}^2+L_U^2.
	\end{align*}
	Now \eqref{Lipschitz4} follows directly from \eqref{Lipschitz2} in Theorem~\ref{Theorem:Apriori}. 
	To show \eqref{Lipschitz3}, let $u,\tilde u\in \Uad$ and set $\bX:=\Sm(u)$ and $\bm{\tilde X}:= \Sm(\tilde u)$. 
	In the same way as above, the Lipschitz continuity condition in \cref{Assumption:coefficients} leads to
	\begin{align*}
		{\|\bX-\bm{\tilde X}\|}_{\X}^2 \le 4L_U^2\big({\|\bX-\bm{\tilde X}\|}_{C([0,T]; L^2(\Omega;\Rb^d))}^2+{\|u-\tilde u\|}_\U^2\big).    
	\end{align*}
	Therefore, \eqref{Lipschitz3} follows from \eqref{Lipschitz1} in \cref{Theorem:Apriori}.
\end{proof}

\section{The optimization problem}
\label{sec:Optimization_problem}
%
Next, we focus on the optimization problem itself. 
We start by introducing the cost functional
\begin{align}
	\label{GeneralCost}
	J(\bX,u) \coloneqq j(\bX) + j_T(\bX(T)) + \frac{\kappa}{2}\,{\|u\|}_\U^2,
\end{align}
where $j$ and $j_T$ fulfill the following hypothesis:
\begin{assumption}
	\label{Assumption:continuity_C}
	The functionals $j:\X\to\Rb$ and $j_T:L^2(\Omega;\Rb^d)\to\Rb$ are convex, lower semicontinuous and bounded from below.
\end{assumption}
Notice that $j$ and $j_T$ take random variables and map them to deterministic outputs. 

\begin{example}
	\label{Ex:Cost}
	For example, we may consider the following non-negative functional of tracking type for which \cref{Assumption:continuity_C} holds:
	\begin{align}
		\label{eq:objective_calib} 
		\begin{aligned}
			J(\bX,u)&=\frac{1}{2}\bigg(\int_0^T{\|\Cm(\Eb[\bX(t)])-\cd(t)\|}_{\Rb^{\ell}}^2\,\mathrm dt+{\|\Cm(\Eb[\bX(T)])-\cdT\|}_{\Rb^{\ell}}^2+\kappa\,{\|u\|}_\U^2\bigg)
		\end{aligned}
	\end{align}
	for $(\bm X,u)\in\X\times\U$. In \eqref{eq:objective_calib} and $\kappa>0$ is the control weight. 
	We suppose that the (non-linear) mapping $\Cm:\Rb^d\to\Rb^\ell$ is given in such a way that $J$ is convex and lower semicontinuous. 
	Notice that the objective defined in \eqref{eq:objective_calib} differs from the one typically used in the stochastic optimal control setting (see, e.g., {\em\cite{Fabbri2017StochasticOptimalControlInInfi,Pham2009StochasticModellingOptimization}}), but is present in literature for calibration problems; see, e.g., {\em\cite{FabrizioMonetti2015MethodologiesCalibration,Estevao2000FunctionalCalibration,Gilli2011calibratingPricingModels,Kaebe2009AdjointMonteCarloCalibra, LoHaslam2008algorithmParamEstSDE}}).
\end{example}

Now we can formulate our stochastic optimization problem:
\begin{align}
	\label{eq:SDE_Opti_OU}
	\min J(\bm X,u)\quad\text{subject to (s.t.)}\quad(\bm X,u)\in\X\times \Uad\text{ satisfies \eqref{SDE}.}
\end{align}
Utilizing the solution operator $\Sm$ (introduced in \eqref{OpSm}), we define the reduced cost functional
\begin{align}
	\label{ReducedCost}
	\hat J(u)\colonequals J(\Sm(u),u)\quad\text{for }u\in\Uad.
\end{align}
The associated reduced problem reads
\begin{align}
	\label{eq:SDE_Opti_OU_red}
	\min\hat J(u)\quad\text{s.t.}\quad u\in\Uad.
\end{align}

In \cref{Th:OptExSol}, we deal with the question of the existence of optimal solutions to \eqref{eq:SDE_Opti_OU_red}. 
In view of our numerical example later on, we restrict ourselves to the case of a linear SDE in the case of $\U=L^2(0,T;\Rb^r)$; see also the discussion after the proof of \cref{Th:OptExSol}.

\begin{assumption}
	\label{A:LinSDE}
	Suppose that \eqref{SDE} is a linear SDE. 
	More precisely, the coefficients $a$ and $b$ are of the following form
	\begin{align}
		\label{SettingLinear}
		\begin{aligned}
			a(x,u,t)&=\mathfrak a(t)+\mathfrak A_1(t)x+\mathfrak A_2(t)u&&\text{for }(x,u,t)\in\Rb^d\times\Rb^r\times[0,T],\\
			b(x,u,t)&=\mathfrak b(t)+\mathfrak B_1(t)x+\mathfrak B_2(t)u&&\text{for }(x,u,t)\in\Rb^d\times\Rb^r\times[0,T]
		\end{aligned}
	\end{align}
	with (deterministic) functions
	\begin{align*}
		\mathfrak a&\in L^2(0,T;\Rb^d),&\mathfrak A_1&\in L^\infty(0,T;\Rb^{d\times d}),&\mathfrak A_2&\in L^\infty(0,T;\Rb^{d\times r}),\\
		\mathfrak b&\in L^2(0,T;\Rb^{d\times m}),&\mathfrak B_1&\in L^\infty(0,T;\Rb^{d\times m\times d}),&\mathfrak B_2&\in L^\infty(0,T;\Rb^{d\times m\times r})
	\end{align*}
	and
	\begin{align*}
		\left.
		\begin{aligned}
			\big(\mathfrak B_1(t)x\big)_{ij}&=\sum_{l=1}^d\mathfrak B_{1,ijl}(t)x_l,\quad\\
			\big(\mathfrak B_2(t)u\big)_{ij}&=\sum_{l=1}^r\mathfrak B_{2,ijl}(t)x_l
		\end{aligned}
		\right\}\quad\text{for }1\le i\le d\text{ and }1\le j\le m.
	\end{align*}
\end{assumption}

\begin{remark}
	\label{Rem:LinSDE}
	\begin{enumerate}[label=\text{\textup{\arabic*)}}]
		\item \label{item:linear_uniqueSolution} It follows from \cref{A:LinSDE} that
		\begin{align}
			\label{LinSDE}
			\left\{
			\begin{aligned}
				\mathrm d\bX(t)&=\big(\mathfrak a(t)+\mathfrak A_1(t)\bX(t)+\mathfrak A_2(t)u(t)\big)\,\mathrm dt\\
				&\quad+\big(\mathfrak b(t)+\mathfrak B_1(t)\bX(t)+\mathfrak B_2(t)u(t)\big)\,\mathrm d\bm B(t)\quad\text{for all }t \in(0,T],\\
				\bX(0)&=\bX_\circ
			\end{aligned}
			\right.
		\end{align}
		possesses a unique solution $\bX\in\X$ for every $u\in\U$ and  every $\mathscr F_\circ$-measurable $\bX_\circ\in L^2(\Omega;\Rb^d)$; cf. \cite[Theorem 6.14]{YongZhou1999StochasticControls}.
		\item Notice that \cref{A:LinSDE} implies \cref{Assumption:coefficients}. 
		Thus, the estimates of \cref{Theorem:Apriori} and \cref{Lemma:ExUni_Uad} are satisfied also in the linear setting.
		\item \label{item:linearSDE_propertiesSolutionOperator} Suppose that \cref{Assumption:coefficients,A:LinSDE} hold. 
		For given initial condition $\bX_\circ:\Omega\to\Rb^d$ let $\hat\bX\in\X$ be the unique solution to the linear SDE
		\begin{align}
			\mathrm d\hat\bX(t)=\mathfrak a(t)\,\mathrm dt+\mathfrak b(t)\,\mathrm d\bm B(t)\text{ for all }t \in(0,T],\quad\hat\bX(0)=\bm X_\circ.
			\label{SDElin}
		\end{align}
        Moreover, for given control $u\in\Uad$ we denote by $\bX_\mathsf h\in\X_\circ$ the unique solution to
		\begin{align}
			\left\{
			\begin{aligned}
				\mathrm d\bX_\mathsf h(t)&=\big(\mathfrak A_1(t)\bX_\mathsf h(t)+\mathfrak A_2(t)u(t)\big)\,\mathrm dt\\
				&\quad+\big(\mathfrak B_1(t)\bX_\mathsf h(t)+\mathfrak B_2(t)u(t)\big)\,\mathrm d\bm B(t)\quad\text{for all }t \in(0,T],\\
				\bX_\mathsf h(0)&=0,
			\end{aligned}
			\right.
			\label{SDEhom}
		\end{align}
        where we set $\X_\circ=\{\bY\in\X \; | \; \bY(0)=0\text{ in }L^2(\Omega;\Rb^d)\big\}$.
		Then the solution operator $\mathcal S_\mathsf h:\Uad\to\X_\circ$, where $\bX_\mathsf h=\mathcal S_\mathsf h(u)$ solves \eqref{SDEhom}, is well defined, linear and bounded. 
		Moreover, $\mathcal S(u)=\hat\bX+\mathcal S_\mathsf h(u)$ solves \eqref{SDE} with the setting \eqref{SettingLinear}.
	\end{enumerate}
\end{remark}

\begin{theorem}
	\label{Th:OptExSol}
	Let \cref{Assumption:coefficients,Assumption:continuity_C,A:LinSDE} hold. 
	Then there exists a (global) optimal solution $\bar u\in\Uad$ to \eqref{eq:SDE_Opti_OU_red}.
\end{theorem}

\begin{proof}
	We follow the main ideas of the proofs in \cite[Theorem 5.2, Chapter 2]{YongZhou1999StochasticControls}. 
	First, we observe that the reduced cost $\hat J$ is bounded from below. Thus, there exists a minimizing sequence $\{u^k\}_{k\in\Nb}\subset\Uad$ such that
	\begin{align}
		\label{MinSeq}
		-\infty< \bar J=\inf\big\{\hat J(u)\,\vert\,u\in\Uad\big\}=\lim_{k\to\infty}\hat J(u^k)<\infty.
	\end{align}
	We set $\bX^k=\Sm(u^k)$ for $k\in\Nb$. 
	Due to \cref{Rem:LinSDE}-\ref{item:linear_uniqueSolution}, the sequence $\{\bX^k\}_{k\in\Nb}\subset\X$ is well defined. 
	Since $\Uad$ is bounded, the sequence $\{u^k\}_{k\in\mathbb N}$ is bounded as well. 
	Hence, there exists a subsequence (which is still labeled by $u^k$) $\{u^{k}\}_{k\in\Nb}\subset\Uad$ and an element $\bar u\in\Uad$ such that $u^{k}\rightharpoonup\bar u$ in $\U$ as $k\to\infty$. We set $\bar X=\mathcal S(\bar u)$. 
	By Mazur's theorem \cite[Lemma 10.19]{RenardyRogers2004IntroPDE}, we get the existence of a function $N:\Nb \rightarrow \Nb$ and a sequence of sets of real numbers $\{\alpha_{ik}\}_{i=k}^{N(k)}$ such that the sequence of convex combinations
	\begin{align}
		\label{XtildeConv}
		\widetilde{u}^k 
		\colonequals \sum_{i = k}^{N(k)} \alpha_{ik}u^{i} \in\Uad
		\qquad\qquad
		\text{with } \alpha_{ik} \geq 0\text{ and }\sum_{i = k}^{N(k)} \alpha_{ik} =1
	\end{align}
	converges strongly to $\bar{u}$ in $\U$, i.e.
	\begin{align}
		\widetilde{u}^k\to\bar u\quad\text{in }\U.
		\label{eq:convergence_v^k_u}
	\end{align}
	Since the set $\Uad$ is convex and closed, we have $\bar{u} \in \Uad$. Furthermore, if $\widetilde{\bX}^k$ is the state corresponding to the control $\widetilde{u}^k$, then we have the strong convergence (see \cref{Theorem:ExUni} and \cref{Lemma:ExUni_Uad})
	\begin{align}
		\label{ConvXtilde}
		\widetilde{\bX}^k\to\bar{\bX}\quad\text{ in }\X.
	\end{align}
	Using \cref{Rem:LinSDE}-\ref{item:linearSDE_propertiesSolutionOperator}, we find that
	\begin{align}
		\label{X_k}
		\begin{aligned}
			\widetilde{\bX}^k&=\mathcal S\big(\widetilde{u}^k\big)=\hat\bX+\mathcal S_\mathsf h\big(\widetilde{u}^k\big)=\hat\bX+\mathcal S_\mathsf h\bigg(\sum_{i = k}^{N(k)} \alpha_{ik}u^{i}\bigg)=\sum_{i = k}^{N(k)} \alpha_{ik}\big(\hat\bX+\mathcal S_\mathsf h(u^{i})\big)\\
			&=\sum_{i=k}^{N(k)} \alpha_{ik}\bX^i.            
		\end{aligned}
	\end{align}
	Due to \cref{Assumption:continuity_C} both $j$ and $j_T$ are lower semicontinuous and convex. 
	Thus, we infer from \eqref{XtildeConv}, \eqref{ConvXtilde} and \eqref{X_k} that
	\begin{align}
		\label{CostCont1}
		\begin{aligned} 
			j(\bar\bX)+j_T(\bar\bX)&\le\lim_{k\to\infty}j\big(\widetilde{\bX}^k\big)+\lim_{k\to\infty}j_T\big(\widetilde{\bX}^k\big)\\
			&=\lim_{k\to\infty}
			j\bigg(\sum_{i=k}^{N(k)} \alpha_{ik}\bX^i\bigg)
			+ \lim_{k\to\infty}j_T\bigg(\sum_{i=k}^{N(k)} \alpha_{ik}\bX^i(T)\bigg)\\
			&\le\lim_{k\to\infty}\sum_{i=k}^{N(k)}\alpha_{ik}\left(j\big(\bX^{i}\big)+ j_T\big(\bX^{i}(T)\big))\right).
		\end{aligned}
	\end{align}
	Since $u\mapsto\|u\|_\U^2$ is continuous and convex in $\U$, we infer from \eqref{eq:convergence_v^k_u}
	\begin{align}
		\label{CostCont2}
		\begin{aligned}
			{\|\bar u\|}_\U^2&=\lim_{k\to\infty}{\|\widetilde{u}^k\|}_\U^2= \lim_{k\to\infty} \Big\|\sum_{i = k}^{N(k)} \alpha_{ik}u^{i}\Big\|_\U^2\leq  \lim_{k\to\infty}\sum_{i = k}^{N(k)} \alpha_{ik}\,{\| u^{i}\|}_\U^2.
		\end{aligned}
	\end{align}
	Combining \eqref{MinSeq}, \eqref{CostCont1} and \eqref{CostCont2}, we obtain
	\begin{align}
		\label{eq:existence_limit}
		\begin{aligned}
			\hat{J}(\bar{u})&=j\left(\bar{\bX}\right)+j_T\left(\bar{\bX}(T)\right) + \frac{\kappa}{2}\,{\|\bar{u}\|}^2_\U\\
			&\leq\lim_{k\to\infty}\sum_{i=k}^{N(k)}\alpha_{ik}\left(j\big(\bX^{i}\big)+ j_T\big(\bX^{i}(T)\big)+\frac{\kappa}{2}\,{\|u^{i}\|}_\U^2\right)=\lim_{k\to\infty} \sum_{i = k}^{N(k)} \alpha_{ik}\, \hat{J}(u^{i}).
		\end{aligned}
	\end{align}
	To pass to the limit in the last equality in \eqref{eq:existence_limit}, we perform the following estimates. 
	Due to \eqref{MinSeq} we have $\lim_{k\to\infty}\hat J(u^k)=\bar J$. 
	Thus, for every $\varepsilon>0$ there exists a $K=K(\varepsilon)\in\Nb$ such that $|\hat J(u^k)-\bar J|<\varepsilon$ for all $k\ge K$. 
	Moreover, it holds that $\sum_{i = k}^{N(k)} \alpha_{ik}=1$ and $0\le\alpha_{ik}\le1$ for $i=k,\ldots,N(k)$ and for any $k$. Consequently,
	\begin{align*}
		\sum_{i=k}^{N(k)}\alpha_{ik}\,\hat{J}(u^{i})
		&=\sum_{i=k}^{N(k)}\alpha_{ik}\,\bar J+\sum_{i=k}^{N(k)}\alpha_{ik}\,\big(\hat{J}(u^{i})-\bar J\big)
		\\
		&\le\bar J
		\sum_{i=k}^{N(k)}\alpha_{ik}
		+\sum_{i=k}^{N(k)}
		\alpha_{ik}\,
		\big|\hat{J}(u^{i})-\bar J\big|\le\bar J+\varepsilon\quad\text{for all }k\ge K
	\end{align*}
	and analogously
	\begin{align*}
		\sum_{i=k}^{N(k)}\alpha_{ik}\,\hat{J}(u^{i})&\ge\bar J\sum_{i=k}^{N(k)}\alpha_{ik}-\sum_{i=k}^{N(k)}\alpha_{ik}\,\big|\hat{J}(u^{i})-\bar J\big|\ge\bar J-\varepsilon\quad\text{for all }k\ge K.
	\end{align*}
	Since $\varepsilon>0$ is chosen arbitrarily, we conclude that $\lim_{k\to\infty}\sum_{i=k}^{N(k)}\alpha_{ik}\,\hat{J}(u^{i})=\bar J$ holds true. 
	Thus, it follows from \eqref{eq:existence_limit} that $\bar u$ solves \eqref{eq:SDE_Opti_OU_red}.
\end{proof}

\begin{remark}
	\label{rem:existence_general}
	Notice that the assumption of the linearity of the SDE is quite restrictive. However, it is in general not possible to prove existence of optimal controls for general SDE in the strong sense. On the contrary, one has to consider so-called \emph{relaxed controls} for which the probability space can not be fixed a-priori \cite{Buckdahn2010existenceStochasticControl, YongZhou1999StochasticControls}. In \cite{ElKarouiNguyen1987ExistenceOSCP}, the authors deal with the question of existence of optimal stochastic control for a quite general form of stochastic optimal controls problems with a state equation of the form \eqref{SDE}.
\end{remark}

Let us come back to the case when the control is time-independent and just a vector in $\U^\mathrm c=\Rb^k$; see \cref{rem:Uad_nonemptyConvexClosed,rem:linear_example}. In this case, the proof of the existence of optimal parameters can also be executed for a non-linear SDE.

\begin{theorem}
	\label{Th:OptExSol_calibration}
	Let \cref{Assumption:coefficients,Assumption:continuity_C} hold for $\U=\U^\mathrm c$ and $\Uad=\UadCalib$. 
	Then there exists a (global) optimal solution $\bar u\in \UadCalib$ to \eqref{eq:SDE_Opti_OU_red}.
\end{theorem}
\begin{proof}
	First, notice as in the proof of \cref{Th:OptExSol} that the reduced cost $\hat J$ is non-negative and the existence of a minimizing minimizing sequence $\{u^k\}_{k\in\Nb}\subset\UadCalib$. 
	We use again the notation $\bX^k=\Sm(u^k)$ for $k\in\Nb$. 
	Due to \cref{Theorem:ExUni}, the sequence $\{\bX^k\}_{k\in\Nb}\subset\X$ is well defined. 
	Since $\UadCalib$ is compact, there exists a subsequence of the minimizing sequence (which is still labeled by $u^k$) $\{u^{k}\}_{k\in\Nb}\subset\UadCalib$ and an element $\bar u\in\UadCalib$ such that we have the strong convergence $u^{k}\rightarrow\bar u$ in $\U$ as $k\to\infty$. We set $\bXb=\Sm(\bar u)\in\X$. It follows from \cref{Theorem:Apriori} that $\bX^k(T) \to \bXb(T)$ in $L^2(\Omega,\Rb^d)$ as $k\to\infty$. Due to \cref{Lemma:ExUni_Uad} we have  $\bX^k \to \bXb$ in $\X$ as $k\to\infty$. Now, with the notation of \eqref{MinSeq} and \cref{Assumption:continuity_C}, it holds that
	\begin{align*}
		\hat{J}(\bar{u})&=j\left(\bXb\right)+j_T\left(\bXb(T)\right) + \frac{\kappa}{2}\,{\|\bar{u}\|}^2_\U
		\leq\lim_{k\to\infty}\big( j\left(\bX^{k}\right)+j_T(\bX^{k}(T))\big)+\frac{\kappa}{2}\,{\|u^{k}\|}_\U^2\big)\\
		&= \inf\{\hat J(u)\,\big|\,u \in \UadCalib\big\}.
	\end{align*}
\end{proof}


\section{Discretization of the optimization problem}
\label{sec:Discretization}

In this section, we introduce a discretization for \eqref{eq:SDE_Opti_OU} and derive the associated first-order necessary optimality system which is the basis for a gradient-based method utilized for our numerical tests in \cref{sec:Numerical_experiments}. 
For that reason, the following hypotheses are required.

\begin{assumption}
	\label{Assumption:Diffbarkeit}
	\begin{enumerate}[label=\text{\textup{\arabic*)}}]
		\item \label{AssumptionCost} The mapping $\Cm:\Rb^d\to\Rb^\ell$ is continuously differentiable and the function $\cd:[0,T]\to \Rb^\ell$  is continuous.
		\item \label{AssumptionState} The functions $a\colon \Rb^d\times \Rb^r \times [0,T]\to \Rb^d$ and $b\colon \Rb^d\times \Rb^r \times [0,T]\to \Rb^{d\times m}$ are continuous and satisfy \cref{Assumption:coefficients}. Moreover,
		\begin{align*}
			\Rb^d&\ni x\mapsto a(x,u(t),t)\in\Rb^d,&\Rb^d&\ni x\mapsto b(x,u(t),t)\in\Rb^{d \times m}&&(\text{for any }u\in\Uad),\\
			\Rb^r&\ni u\mapsto a(x,u,t)\in\Rb^d,&\Rb^r&\ni u\mapsto b(x,u,t)\in\Rb^{d\times m}&&(\text{for any }x\in\Rb^d)
		\end{align*}
		are continuously differentiable for all $t\in[0,T]$.
	\end{enumerate}
	
\end{assumption}

We remark that if \cref{Assumption:Diffbarkeit}-2) holds, then \eqref{eq:LipCondition} holds for all $t\in[0,T]$, by continuity of $a$ and $b$. In a first discretization step, we consider a semidiscrete version of the calibration problem and prove a convergence result. For the numerical implementation, we then follow \cite{Kaebe2009AdjointMonteCarloCalibra} and apply a Monte-Carlo discretization for \eqref{SDE} to get a discretized and high-dimensional, but now deterministic constrained optimization problem. 
Then, we can apply a Lagrangian framework to get optimality conditions.

\subsection{The semidiscrete calibration problem}
\label{sec:D6.1}

In this section, we study the calibration problem (cf. \cref{rem:Uad_nonemptyConvexClosed}-\ref{item:remark_calibration_example}), where the control space $\Uc=\Rb^r$ is already finite-dimensional and needs not to be discretized:
\begin{align}
	\label{OptCal}
	\min \hat J(u)=J(\Sm(u),u)\quad\text{s.t.}\quad u\in\UadCalib=[\ua,\ub]\subset\Uc,
\end{align}
where $\bX=\Sm(u)\in\X$ is the solution to \eqref{SDE-Cal}. In order to keep the presentation clear and concise, we restrict ourselves in this section to the form of the functional given in \eqref{eq:objective_calib}. 
However, all results can be extended to functionals that satisfy \cref{Assumption:continuity_C} and \cref{Assumption:Diffbarkeit}.

For the semidiscrete approximation of \eqref{SDE-Cal}, we  split the time interval $[0,T]$ in $N$ intervals of equal length $\Delta t=T/N$. We set 
\begin{align*}
	t_\nu \colonequals \nu\Delta t \quad\text{for }\nu=0,\ldots,N
\end{align*}
and define an elementary time interval as
\begin{align*}
	\mathscr I_\nu\colonequals [t_\nu, t_{\nu+1}]= [t_\nu, t_\nu+\Delta t ]\subset[0,T]\quad\text{for }\nu=0,\ldots,N-1.
\end{align*}
We will approximate continuous functions $f\colon [0,T]\to \Rb^d$ by step functions of the form
\begin{align*}
	\sum_{\nu=0}^{N-1} f(t_\nu) \chi_{\mathscr I^\nu}(t),
\end{align*}
where  $\chi_{\mathfrak{S}}$ denotes  the indicator function for the set $\mathfrak{S}$, i.e.
\begin{equation}
	\label{def:indicator_function}
	\begin{aligned}
		\chi_{\mathfrak{S}}(\mathfrak{s}) \colonequals\begin{cases}
			1 \qquad &\text{if } \mathfrak{s} \in \mathfrak{S},\\
			0 \qquad &\text{otherwise}.
		\end{cases}
	\end{aligned}
\end{equation}

Following, e.g., \cite{Higham2001IntroNumSDE}, we discretize \eqref{SDE-Cal} with respect to time and obtain an explicit semidiscrete scheme of Euler--Maruyama type. For $u\in \UadCalib$, we define $\bX^N=\{\bX_\nu^N\}_{\nu=0}^{N}\subset L^2(\Omega;\Rb^d)$ as the solution of
\begin{equation}
	\begin{aligned}
		\label{SDEDiscNeu}
		\bX^N_{\nu+1}&=\bX^N_\nu +a(\bX^N_\nu,u,t_\nu)\,\Delta t+b(\bX^N_\nu,u,t_\nu)\,\Delta\bB_\nu\quad\text{for }\nu=0,\ldots,N-1,\\
		\bX_0^N&= \bX_\circ,
	\end{aligned}    
\end{equation}
where $\Delta \bB_\nu\colonequals\bB(t_{\nu+1})-\bB(t_\nu)\in L^2(\Omega;\Rb^m)$. 
Note that  $(\Delta\bB_\nu)_i\sim\mathcal N(0,\Delta t)$ for $1\le i\le m$. Here, $\mathcal{N}(0,\Delta t)$ defines a normal distribution with mean $0$ and standard deviation $\Delta t^{1/2}$. 
For the explicit realization of sampling Brownian motion increments, we refer to \cref{sec:Num_approx}.
As the scheme \eqref{SDEDiscNeu} is explicit, there exists a unique solution $\bX^N$, and $\bX^N_\nu$ is $\Fs_\circ(t_\nu)$-measurable for $\nu=0,\dots,N$. 
We introduce the discrete solution space 
\begin{align*}
	\X^N\colonequals\big\{\bY^N = \{\bY_\nu^N\}_{\nu=0}^{N}\,\big|\,\bY_\nu^N\in L^2(\Omega;\Rb^d)\big\} = \big(L^2(\Omega;\Rb^d)\big)^{N+1}
\end{align*}
with norm
\begin{align*}
	{\|\bY^N\|}_{\X^N} \colonequals \max_{\nu=0,\dots,N} {\|\bY^N_\nu\|}_{L^2(\Omega,\Rb^d)}
\end{align*}
and the discrete solution operator 
\begin{align*}
	\Sm^N:\UadCalib\to\X^N,\quad u\mapsto \Sm^N(u) \colonequals \bX^N,
\end{align*}
where $\bX^N$ denotes the solution of \eqref{SDEDiscNeu}.

The following result is the time-discrete analogon of \cref{Theorem:Apriori}.

\begin{lemma}
	\label{Lemma:semidiscrete-a-priori}
	Let \cref{Assumption:Diffbarkeit}-\ref{AssumptionState} hold,  and let $\bX_\circ\in L^2(\Omega;\Rb^d)$ be an $\Fs_\circ$-measurable random variable. Then, for all $u,\tilde u\in \UadCalib$, we have 
	\begin{align}
		\label{Lipschitz1-semi}
		{\| \Sm^N(u)-\Sm^N(\tilde u)\|}_{\X^N} & \le C_1\,{\|u-\tilde u\|}_{\Uc},
        \\
		\label{Lipschitz2-semi}
		{\|\Sm^N(u)\|}_{\X^N}&\le C_2
	\end{align}
	with non-negative constants $C_1 = C_1(L_U,T)$ and $C_2 = C_2(L_U,T,\bX_\circ)$ independent of $N$.
\end{lemma} 

\begin{proof}
	Let $u, \tilde u\in \UadCalib$, and set $\bX^N\colonequals\Sm^N(u)$ and $\bXt^N\colonequals\Sm^N(\tilde u)$. We consider  the differences $\bm\Delta^N_\nu\colonequals\bX^N_\nu-\bXt_\nu^N\in L^2(\Omega;\Rb^d)$, $0\le\nu\le N$. Obviously, we have $\bm \Delta^N_0=0$. Iteratively, we obtain
	\begin{align*}
		\bm\Delta^N_{\nu+1}&=\bm\Delta^N_\nu +\big(a(\bX^N_\nu,u,t_\nu)-a(\bXt^N_\nu,\tilde u,t_\nu)\big)\,\Delta t+\big(b(\bXt
		^N_\nu,u,t_\nu)-b(\bXt^N_\nu,\tilde u,t_\nu)\big)\,\Delta\bB_\nu\\
		&=\sum_{l=0}^\nu\Big(\big(a(\bX^N_l,u,t_l)-a(\bXt^N_l,\tilde u,t_l)\big)\,\Delta t+\big(b(\bX
		^N_l,u,t_l)-b(\bXt^N_l,\tilde u,t_l)\big)\,\Delta\bB_l\Big)
	\end{align*} 
	for $\nu=0,\ldots,N-1$. From this, we see that
	\begin{equation}\label{eq-5-7}
		\begin{aligned}
			\Eb\big[{\|\bm\Delta^N_{\nu+1}\|}_{\Rb^d}^2\big]&\le 2\,\Eb\bigg[\Big\|\sum_{l=0}^\nu\big(a(\bX^N_l,u,t_l)-a(\bXt^N_l,\tilde u,t_l)\big)\,\Delta t\Big\|_{\Rb^d}^2\bigg]\\
			&\quad+2\, \Eb\bigg[\Big\|\sum_{l=0}^\nu\big(b(\bX
			^N_l,u,t_l)-b(\bXt^N_l,\tilde u,t_l)\big)\,\Delta\bB_l\Big\|_{\Rb^d}^2\bigg]
		\end{aligned}
	\end{equation}
	for $\nu=0,\ldots,N-1$. With \cref{Assumption:Diffbarkeit}-\ref{AssumptionState}, $\Delta t = T/N$ and $\sum_{l=0}^\nu 1\le N$, we get
	\begin{align*}
		&\Eb\bigg[\Big\|\sum_{l=0}^\nu\big(a(\bX^N_l,u,t_l)-a(\bXt^N_l,\tilde u,t_l)\big)\,\Delta t\Big\|_{\Rb^d}^2\bigg]\\
		& \qquad \le(\Delta t)^2\,\Eb\bigg[\Big(\sum_{l=0}^\nu\big\|a(\bX^N_l,u,t_l)-a(\bXt^N_l,\tilde u,t_l)\big\|_{\Rb^d}\Big)^2\bigg]\\
		&\qquad\le(\Delta t)^2L_U^2\Eb\bigg[\Big(\sum_{l=0}^\nu\big({\|\bm\Delta_l^N\|}_{\Rb^d}+{\|u-\tilde u\|}_\Uc\big)\Big)^2\bigg]\\
		& \qquad\le\Delta tL_U^2T\,\Eb\bigg[\sum_{l=0}^\nu\big({\|\bm\Delta_l^N\|}_{\Rb^d}+{\|u-\tilde u\|}_\Uc\big)^2\bigg]\\
		&\qquad\le2L_U^2T\bigg(\Delta t\sum_{l=0}^\nu\Eb\big[{\|\bm\Delta_l^N\|}_{\Rb^d}^2\big]+{\|u-\tilde u\|}_\Uc^2\bigg).
	\end{align*}
	For the second expectation on the right-hand side of \eqref{eq-5-7},  
	we set
	\begin{align*}
		G_l = (G_l^{ji})_{j=1,\dots,d,\,i=1,\dots,m} \colonequals b(\bX
		^N_l,u,t_l)-b(\bXt^N_l,\tilde u,t_l) \quad \text{for }l=0,\dots,\nu
	\end{align*}
	and get (cf. also \cite[Subsection~4.2.2]{Evans2013SDE})
	\begin{align*}
		&\Eb\bigg[\Big\|\sum_{l=0}^\nu G_l\,\Delta\bB_l\Big\|_{\Rb^d}^2\bigg]=\Eb\bigg[\sum_{l,\ell=0}^\nu\Big(G_l\Delta\bB_l\Big)^\top\Big(G_\ell\,\Delta\bB_\ell\Big)\bigg]=\sum_{l,\ell=0}^\nu\Eb\Big[\Delta\bB_l^\top G_l^\top 
		G_\ell\,\Delta\bB_\ell\Big]\\
		&\quad=\sum_{l,\ell=0}^\nu\sum_{i=1}^m\sum_{j=1}^d\sum_{k=1}^m\Eb\Big[\Delta B_{li}G_l^{ji}G_\ell^{jk}\,\Delta B_{\ell k}\Big].
	\end{align*}
	If $\ell>l$, then $\Delta B_{\ell k}$ is independent of $\Delta B_{li}G_l^{ji}G_\ell^{jk}$, and we get 
	\[ \Eb\Big[\Delta B_{li}G_l^{ji}G_\ell^{jk}\,\Delta B_{\ell k}\Big] 
	= \Eb\Big[\Delta B_{li}G_l^{ji}G_\ell^{jk}\Big]\, \Eb\Big[\Delta B_{\ell k}\Big]=0.\]
	In the same way, the expectation vanishes for $\ell<l$ and for $i\not=k$. With this and $\Eb[ (\Delta B_{li})^2]= t_{i+1}-t_i = \Delta t$, we obtain
	\begin{align*}
		\Eb\bigg[\Big\|\sum_{l=0}^\nu G_l\,\Delta\bB_l\Big\|_{\Rb^d}^2\bigg]&=\sum_{l=0}^\nu \sum_{i=1}^m\sum_{j=1}^d\Eb\big[(G_l^{ji})^2\big]\Eb\big[ (\Delta B_{li})^2\big]=\Delta t\sum_{l=0}^\nu\Eb\bigg[\sum_{i=1}^m\sum_{j=1}^d(G_l^{ji})^2\bigg]\\
		&=\Delta t\sum_{l=0}^\nu\Eb\big[{\|G_l\|}_\Fr^2\big].
	\end{align*}
	From the definition of $G_l$ and the Lipschitz continuity of $b$ (see \cref{Assumption:Diffbarkeit}-\ref{AssumptionState}) it follows that
	\begin{align*}
		\Eb\bigg[\Big\|\sum_{l=0}^\nu\big(b(\bX
		^N_l,u,t_l)-b(\bXt^N_l,\tilde u,t_l)\big)\,\Delta\bB_l\Big\|_{\Rb^d}^2\bigg]\le2 L_U^2\bigg(\Delta t\sum_{l=0}^\nu\Eb\big[{\|\bm\Delta_l^N\|}_{\Rb^d}^2\big]+{\|u-\tilde u\|}_\Uc^2\bigg).
	\end{align*}
	Inserting this into \eqref{eq-5-7} yields
	\begin{align*}
		\Eb\big[{\|\bm\Delta^N_{\nu+1}\|}_{\Rb^d}^2\big]\le c_1\Delta t\sum_{l=0}^\nu\Eb\big[{\|\bm\Delta_l^N\|}_{\Rb^d}^2\big]+c_1\,{\|u-\tilde u\|}_\Uc^2\quad\text{for } \nu=0,\dots,N-1
	\end{align*}
	with $c_1\colonequals 4(T+1)L_U^2$. Now an application of the discrete Gronwall inequality (see \cite[Theorem 3.2]{ConlanWang1992ContDiscGronwall}) gives the desired estimate
	\begin{align*}
		\max_{\nu=0,\dots,N} \Eb\big[{\|\bX_\nu^N-\bXt_\nu^N\|}_{\Rb^d}^2\big]\le C_1\,{\|u-\tilde u\|}_\Uc^2
	\end{align*}
	with a constant $C_1 = C_1(T, L_U)$, which shows \eqref{Lipschitz1-semi}. 
	The proof of \eqref{Lipschitz2-semi} follows the same lines, writing
	\begin{align*}
		a(\bX^N_l, u, t_l) = \Big( a(\bX^N_l, u, t_l)- a(0, u, t_l)\Big) + a(0, u, t_l)
	\end{align*}
	(see the proof of \cref{Theorem:Apriori}). We obtain
	\begin{align*}
		\Eb\big[\|\bX^N_\nu\|^2\big]\le c_2 \Delta t \sum_{l=0}^\nu \Eb\big[\|\bX^N_l\|^2\big]+c_2 + c_2 \sum_{l=0}^{N-1}  \Delta t \Big( {\|a(0,u,t_l)\|}_{\Rb^d}^2 + {\|b(0,u,t_l)\|}_{\Fr}^2\Big),
	\end{align*}
	with a constant $c_2 = c_2(L_U,T,\bX_\circ)$. By \cref{Assumption:Diffbarkeit}-\ref{AssumptionState}, the (deterministic) coefficient $a(0,u,\cdot)$ is  continuous and therefore uniformly continuous in the interval $[0,T]$. Therefore, $\sum_{\nu=0}^{N-1} \|a(0,u,t_l)\|_{\Rb^d}^2 \chi_{\mathscr I^\nu}(\cdot)$ converges uniformly to $\|a(0,u,\cdot )\|^2_{\Rb^d}$ for $N\to\infty$. Integration with respect to $t\in [0,T]$ shows that
	\begin{align*}
		\sum_{l=0}^{N-1}\Delta t\,{\|a(0,u,t_l)\|}_{\Rb^d}^2\to \int_0^T{\|a(0,u,\cdot )\|}^2_{\Rb^d}\,\mathrm d t\le L_U^2,
	\end{align*}
	and, in particular, this sum is bounded by a constant for all $N\in\Nb$. Treating the coefficient $b$ in the same way, we get
	\begin{align*}
		\Eb\big[ \|\bX^N_{\nu+1}\|^2\big] \le c_2 \Delta t \sum_{l=0}^\nu \Eb\big[\|\bX^N_l\|^2\big] + c_2'\quad\text{ for }\nu=0,\dots,N-1
	\end{align*}
	for some constant $c_2'=c_2'(L_U,T,\bX_\circ)$. Now the discrete Gronwall inequality yields \eqref{Lipschitz2-semi}.
\end{proof}

We set $\cdIndex{\nu}\colonequals\cd(t_\nu)$ for $\nu=0,\dots,N$ and introduce the discrete cost functional
\begin{align*}
	J^N(\bX^N,u)\colonequals&\frac{\Delta t}{2}\sum_{\nu=0}^{N-1}\left\|\Cm\big(\Eb\big[\bX^N_\nu\big]\big)- \cdIndex{\nu}\right\|_{\Rb^\ell}^2+\frac{1}{2}\left\|\Cm\big(\Eb\big[\bX^N_N\big]\big)- \cdIndex{N}\right\|_{\Rb^\ell}^2+\frac{\kappa}{2}\,{\|u\|}_\Uc^2 
\end{align*}
for $\bX^N\in \X^N$ and $u\in \UadCalib$ as well as the 
discrete reduced cost functional 
\begin{align*}
	\hat J^N(u)\colonequals J^N(\Sm^N(u),u)\quad\text{for }u\in\UadCalib.
\end{align*}
The following is the main result on  convergence of the semidiscrete approximation.

\begin{theorem}
	\label{thm:convergence_discrete_continuous}
	Let \cref{Assumption:Diffbarkeit} hold.
	\begin{enumerate}[label=\text{\textup{\arabic*)}}]
		\item The map $\hat J^N\colon\UadCalib\to \Rb$ is Lipschitz continuous with Lipschitz constant independent of $N$, i.e., there exists a constant $L_d>0$ such that for all $N\in\Nb$ and all $u, \tilde u\in \UadCalib$ we have
		\begin{align*}
			\big|\hat J^N(u )-\hat J^N(\tilde u )\big|\le L_d\,{\|u -\tilde u \|}_{\Uc}.
		\end{align*}
		\item For any $u\in \UadCalib$, we have $\hat J^N(u)\to \hat J(u)$ for $N\to \infty$.
		\item Let  $\{u^{k}\}_{k\in\Nb}\subset\UadCalib$ be a sequence  with $u^{k}\to u\in \UadCalib$ for $k\to\infty$, and let $\{N_k\}_{k\in\Nb}\subset \Nb$ be a sequence with $N_k\to\infty$ for $k\to\infty$. Then 
		\begin{align*}
			\hat J^{N_k}(u^k)\to \hat J(u)\quad \text{for }k\to\infty.
		\end{align*}
	\end{enumerate}
\end{theorem}

\begin{proof}
	1) For $\nu=0,\dots,N$ and $\bX^N,\bXt^N\in\X^N$, we have
	\begin{align*}
		\big\| \Eb\big[\bX^N_\nu\big] - \Eb\big[\bXt^N_\nu\big] \big\|_{\Rb^d}&= \big\|  \Eb\big[\bX^N_\nu-\bXt^N_\nu\big]\big\|_{\Rb^d}\le\big\|\bX^N_\nu -\tilde \bX^N_\nu\big\|_{L^1(\Omega;\Rb^d)}\\
		&\le\sqrt T\,\big\|\bX^N_\nu -\tilde \bX^N_\nu\big\|_{L^2(\Omega;\Rb^{d})},
	\end{align*} 
	which shows that the map $\X^N \to\Rb^d,\;\bX^N\mapsto\Eb[\bX^N_\nu]$ is Lipschitz continuous. 
    Since \cref{Assumption:Diffbarkeit}-\ref{AssumptionState} holds, we can apply \cref{Lemma:semidiscrete-a-priori}. 
    We conclude that $\Sm^N\colon \UadCalib\to\X^N$ is Lipschitz continuous with constant independent of $N$ and that $\Sm^N(\UadCalib)$ is bounded by virtue of \eqref{Lipschitz2-semi}.
	Therefore, the same holds for the map
	\begin{align*}
		\phi_\nu\colon\UadCalib\to\Rb^d\quad u\mapsto \Eb\big[\bX_\nu^N\big]\quad\text{with }\bX^N\colonequals\Sm^N (u).
	\end{align*}
	By \cref{Assumption:Diffbarkeit}-\ref{AssumptionCost}, 
	\begin{align}\label{eq:function-for-cost-functional}
		\Rb^d\to \Rb,\quad x\mapsto \left\|\Cm(x)- \cdIndex{\nu}\right\|_{\Rb^\ell}^2
	\end{align}
	is continuously differentiable and therefore Lipschitz on the bounded range of $\phi_\nu$. As the composition of Lipschitz functions is again a Lipschitz function, we see that 
	\begin{align*}
		\UadCalib \to \Rb,\quad u\mapsto \left\|\Cm\big(\Eb\big[\bX^N_\nu\big]\big)- \cdIndex{\nu}\right\|_{\Rb^\ell}^2\quad\text{with }\bX^N\colonequals \Sm^N (u)
	\end{align*}
	is Lipschitz continuous with a Lipschitz constant $L>0$ independent of $N$. 
	As $u\mapsto \|u\|^2_{\Uc}$ is continuously differentiable and therefore Lipschitz with some constant $L'>0$ on the compact set $\UadCalib$, we obtain for $u,\tilde u\in \UadCalib$ and for every $N\in\Nb$
	\begin{align*}
		\big|\hat J^N(u)-\hat J^N(\tilde u)\big| 
		&\le \frac{\Delta t}2\sum_{\nu=0}^{N-1}L\,{\|u-\tilde u\|}_{\Uc}+\frac12 L\,{\|u-\tilde u\|}_{\Uc} + \frac\kappa 2 L'\|u-\tilde u\|_{\Uc} 
		\\
		&\le L_d\,{\|u-\tilde u\|}_{\Uc} 
	\end{align*}
	for $L_d\colonequals((T+1)L+\kappa L')/2$,  where we used $\Delta t=T/N$, which gives part 1).
	
	2) Let $u\in\UadCalib$ and set $\bX \colonequals \Sm(u)$ and $\bX^N \colonequals \Sm^N(u)$ for $N\in\Nb$. By \cref{Theorem:Apriori}, we have $\bX\in C([0,T]; L^2(\Omega;\Rb^d))$ which implies that $\Eb[\bX(\cdot)]\in C([0,T]; \Rb^d)$. From this and the continuity of $\Cm$ and $\cd$, we see that the integrand in the cost functional $J(\bX,u)$ (see \eqref{eq:objective_calib}), which is given as $\|\Cm(\Eb[\bX(\cdot)])-\cd(\cdot)\|_{\Rb^{\ell}}^2$, is a continuous function on $[0,T]$ and can be approximated uniformly by the step function 
	\begin{align*}
		t\mapsto \sum_{\nu=0}^{N-1} \big\|\Cm(\Eb[\bX(t_\nu)])-\cdIndex{\nu}\big\|_{\Rb^{\ell}}^2 \chi_{\mathscr I^\nu}(t).
	\end{align*}
	Integrating this with respect to $t\in[0,T]$, we obtain 
	\begin{equation}
		\label{eq:convergence-semidiscrete1}
		J^N\big((\bX(t_\nu))_{\nu=0,\dots,N},u\big)\to J(\bX,u)\quad \text{for }N\to\infty.
	\end{equation}
	Now we use the fact that $\bX^N$ is computed with the Euler-Murayama scheme \eqref{SDEDiscNeu} which is known to be of convergence order $1/2$ (see \cite[Section~1.1.5]{Milstein-Tretyakov21}). Therefore, we can apply the convergence result from \cite[Theorem~1.1.1]{Milstein-Tretyakov21} which tells us that 
	\begin{align*}
		\sup_{\nu=0,\dots,N}\Eb\Big[ \big\|\bX^N_\nu-\bX(t_\nu) \big\|^2_{\Rb^d}\Big]\to 0\quad\text{for }N\to\infty.
	\end{align*}
	This implies $\Eb[\bX^N_\nu]\to\Eb[\bX(t_\nu)]$ for $N\to\infty$ and, by uniform continuity of the function \eqref{eq:function-for-cost-functional} on bounded subsets, 
	\begin{align*}
		\sup_{\nu=0,\dots,N} \bigg|\big\|\Cm(\Eb[\bX^N_\nu])-\cdIndex{\nu}\big\|_{\Rb^{\ell}}^2- \big\|\Cm(\Eb[\bX(t_\nu)])-\cdIndex{\nu}\big\|_{\Rb^{\ell}}^2\bigg| \to 0 \quad \text{for }N\to\infty.
	\end{align*}
	Summing up over $\nu$ and using $\Delta t=T/N$ again, we get
	\begin{align}
		\label{eq:convergence-semidiscrete2}
		\Big|J^N\big(\bX^N, u\big)-J^N\big( (\bX(t_\nu))_{\nu=0,\dots,N},u\big)\Big|\to 0\quad\text{for }N\to\infty. 
	\end{align}
	As $J^N(\bX^N,u)=\hat J^N(u)$ and $J(\bX,u)=\hat J(u)$, the statement in 2) now follows from \eqref{eq:convergence-semidiscrete1} and \eqref{eq:convergence-semidiscrete2}.
	
	3) This is an immediate consequence of parts 1) and 2). In fact, for $\varepsilon>0$ we first choose $k_0\in\Nb$ such that for all $k\ge k_0$, we have $\|u^k-u\|_{\Uc} \le \varepsilon/(2L_d)$ with $L_d$ from 1). 
    Then $|\hat J^N(u^k) - \hat J^N(u)|\le \varepsilon/2$ holds by 1) for all $N\in\Nb$. 
    Now we apply 2) and $N_k\to\infty$ to see that there exists $k_1\ge k_0$ such that for all $k\ge k_1$ we obtain $|\hat J^{N_k}(u) - \hat J(u)| \le \varepsilon/2$. This yields $|\hat J^{N_k}(u^k) - \hat J(u)|\le \varepsilon$ for $k\ge k_1$, which shows 3).
\end{proof}

\subsection{The fully discretized problem}
\label{sec:Num_approx}

In this section, we explain how we solve \eqref{OptCal} numerically.
In \cref{sec:D6.1}, we already described the discretization with respect to time and formulated a semi-discrete scheme \eqref{SDEDiscNeu}.
In order to obtain a fully discrete scheme, we use a Monte Carlo approximation.
More specifically, we consider $M \gg 1$ realizations of the SDE and for every realization and every time instant $t_\nu\in[0,T]$, we precompute samples of the Brownian motion. We denote these samples by
\begin{align*}
	\Delta \bB^\mu=\big[\Delta \bB_1^\mu\big|\ldots\big|\Delta \bB_N^\mu\big]\in\Rb^{m\times N}\quad\text{with }\Delta \bB_\nu^\mu=\big(\Delta B^\mu_{\nu,j}\big)_{j=1}^m\text{ and }\Delta B^\mu_{\nu,j}\sim \mathcal{N}(0,\Delta t)
\end{align*}
for $\mu=1,\ldots,M$. In order to simplify the notation, we set $h=(M,N)$.

We can now introduce the discretization of the state equation \eqref{SDE} as follows: For any given control $u\in\UadCalib$ find
\begin{align*}
	\bX^h\in\X^h\colonequals\big\{\bXt^\mu_\nu\in\Rb^d\,\big|\,\nu=0,\ldots,N
	\text{ and }\mu=1,\ldots,M\big\}
\end{align*}
satisfying
\begin{align}
	\label{SDEDisc}
	\begin{aligned}
		\bX^\mu_{\nu+1}&=\bX^\mu_\nu +a(\bX^\mu_\nu,u,t_\nu)\Delta t+b(\bm X^\mu_\nu,u,t_\nu)\Delta \bm B_\nu^\mu,
		&&\hspace{-2mm}
		\mu=1,\ldots,M,\nu=0,\ldots,N-1,\\
		\bX_0^\mu&= \bX_\circ,&&\hspace{-2mm}\mu=1,\ldots,M.
	\end{aligned}
\end{align}    
Notice that the solution in \eqref{SDEDisc} depends on the pre-computed samples of the Brownian motion.
For the discrete $\bX^h$, we define the discrete expectation value as the arithmetic mean, i.e., for $\nu \in \{0,\ldots,N\}$
\begin{align}
	\label{ExpValDisc}
	\Eb^{M}\left[\bX^h_\nu\right] 
	\colonequals \frac{1}{M}\sum_{\mu=1}^M \bX_\nu^\mu.
\end{align}
The finite-dimensional (Hilbert) space $\X^h$ is isomorphic to $\Rb^{(N+1)M}$ and is supplied by the weighted inner product
\begin{align*}
	{\langle\bX^h,\bXt^h\rangle}_{\X^h}=\frac{\Delta t}{M}\sum_{\mu=1}^M\sum_{\nu=0}^N\big(\bX_\nu^\mu\big)^\top\bXt_\nu^\mu=\Eb^M\bigg[\Delta t\sum_{\nu=0}^N\big(\bX_\nu^h\big)^\top\bXt_\nu^h\bigg]\quad\text{for }\bX^h,\bXt^h\in\X^h,
\end{align*}
where we set $\bX^h_\nu=\{\bX_\nu^\mu\}_{\mu=1}^M$ and $\bXt^h_\nu=\{\bXt_\nu^\mu\}_{\mu=1}^M$ for $\nu\in\{0,\ldots,N\}$.
Notice that \eqref{SDEDisc} is an explicit difference scheme. Thus, the existence of the sequence $\bX^h$ satisfying \eqref{SDEDisc} is clear. In that case the discrete non-linear solution operator
\begin{align*}
	\Sm^h:\UadCalib\to\X^h,
    \quad
    &\bX^h=\Sm^h(u)
	\text{ is the unique solution to \eqref{SDEDisc} for }
	u\in\UadCalib 
    \\
    &\text{ given the $M$ pre-computed Brownian increments}
\end{align*}
is well defined. We formulate a discrete version of \eqref{OptCal} as follows:
\begin{align}
	\min \label{eq:SDE_Opti_OU_discrete}
	\hat J^h(u)\quad\text{s.t.}\quad u\in\UadCalib,
\end{align}
where the discrete reduced cost functional is defined as
\begin{align*}
	\hat J^h(u)\colonequals J^h(\bX^h,u)
	\quad \text{for } u\in\UadCalib
	\text{ and }\bX^h=\Sm^h(u)
\end{align*}
with
\begin{align*}
	J^h(\bX^h,u)\colonequals
	&\frac{1}{2}\bigg(\Delta t\sum_{\nu=0}^{N-1}\left\|\Cm\big(\Eb^M\big[\bX^h_\nu\big]\big) - \cdIndex{\nu}\right\|_{\Rb^\ell}^2+\left\|\Cm\big(\Eb^M\big[\bX^h_N\big]\big) - \cdIndex{N}\right\|_{\Rb^\ell}^2+\kappa\,{\|u\|}_\Uc^2\bigg)
	\\
	=&\frac{\Delta t}{2}\sum_{\nu=0}^{N-1}
	\bigg\|\Cm\Big(\frac{1}{M}
	\sum_{\mu=1}^M\bX^h_\nu\Big) - \cdIndex{\nu}\bigg\|_{\Rb^\ell}^2+\frac{1}{2}\bigg\|\Cm\Big(\frac{1}{M}\sum_{\mu=1}^M\bX^\mu_N\Big)- \cdIndex{N}\bigg\|_{\Rb^\ell}^2+\frac{\kappa}{2}\sum_{i=1}^r\big|u_i|^2.
\end{align*}

The next step is to introduce the associated
Lagrange functional for the discretized problem \eqref{eq:SDE_Opti_OU_discrete}. It is given by
\begin{align}
	\label{eq:Lagrand_SDE_Opti_discrete}
	\begin{aligned}
		&\LF(\bX^h,u,\bL^h)\colonequals\\
		&\qquad\frac{1}{2}\bigg(\Delta t\sum_{\nu=0}^{N-1}
		\left\|\Cm\big(\Eb^M\big[\bm X^h_\nu\big]\big) - \cdIndex{\nu}\right\|_{\Rb^\ell}^2+\left\|\Cm\big(\Eb^M\big[\bX^h_N\big]\big) - \cdIndex{N}\right\|_{\Rb^\ell}^2+\kappa\,{\|u\|}_\U^2\bigg)
		\\
		&\qquad+\frac{1}{M}\sum_{\mu=1}^M \sum_{\nu=0}^{N-1} \big[\bX^\mu_{\nu+1}-\bX^\mu_\nu- a(\bX^\mu_\nu,u,t_\nu)\,\Delta t-b(\bX^\mu_\nu,u,t_\nu)\,\Delta\bB_\nu^\mu\big]^\top\bL^\mu_{\nu+1}
		\\
		&\qquad+\frac{1}{M}\sum_{\mu=1}^M \big[\bX^\mu_0-\bX_\circ^\mu\big]^\top\bL^\mu_0
	\end{aligned}
\end{align}
for $\bX^h=\{\bX^\mu_\nu\}\in\X^h$, and where we have introduced the Lagrange multiplier
\begin{align*}
    \bL^h=\{\bL_\nu^\mu \in \Rb^d \, | \, \nu = 0,\ldots,N \text{ and }\mu = 1,\ldots,M\}\in \X^h.
\end{align*}
In particular, $\bL_0^\mu$ is the Lagrange multiplier associated with the initial condition.
Note that
\begin{align*}
	\underbrace{b(\bX^\mu_\nu,u,t_\nu)}_{\in\Rb^{d\times m}}\underbrace{\Delta\bB_\nu^\mu}_{\in\Rb^m}=\sum_{j=1}^m\Delta B_{\nu,j}^\mu\underbrace{b_j(\bX^\mu_\nu,u,t_\nu)}_{\in\Rb^d}
	\quad\text{for }
	\mu=1,\ldots,M,\; \nu=0,\ldots,N-1.
\end{align*}
Next, we derive the adjoint equation by computing the partial derivative $\LF_\bX$ of $\LF$ with respect to $\bX^h$. 
Let us refer to \cite{Kaebe2009AdjointMonteCarloCalibra} for a similar approach, but in our case, we consider data that is given over the whole time interval $[0,T]$. 
Notice that in the discretize-before-optimize approach, the solution of the adjoint problem is only the variable $\bL^h$.
This is a crucial difference to the optimize-before-discretize approach, where the solution of the adjoint problem is in fact a tuple of stochastic processes (cf. \cref{Sec:RedGradCont} and \cite{YongZhou1999StochasticControls}). 
For any direction $\bXt^h=\{\bXt^\mu_\nu\}\in\X^h$, 
we obtain the following expression 
\begin{align}
	\label{eq:Derivative_L_X}
	\begin{aligned}
		\LF_\bX(\bX^h,u,\bL^h)\bXt^h
		&=\Delta t\sum_{\nu=0}^{N-1}
		\Big[\Cm\big(\Eb^M\big[\bX^h_\nu\big]\big) - \cdIndex{\nu}\Big]^\top\Cm'\big(\Eb^M\big[\bm X^h_\nu\big]\big)\Eb^M\big[\bXt^h_\nu\big]
		\\
		&\quad+\Big[\Cm\big(\Eb^M\big[\bX^h_N\big]\big) - \cdIndex{\nu}\Big]^\top\Cm'\big(\Eb^M\big[\bX^h_N\big]\big)\Eb^M\big[\bXt^h_N\big]
		\\
		&\quad+\frac{1}{M}\sum_{\mu=1}^M \sum_{\nu=0}^{N-1} \Big[\bXt^\mu_{\nu+1}-\bXt^\mu_\nu-\big(a_x(\bX^\mu_\nu,u,t_\nu)\bXt_\nu^\mu\big)\,\Delta t\Big]^\top\bL^\mu_{\nu+1}
		\\
		&\quad-\frac{1}{M}\sum_{\mu=1}^M \sum_{\nu=0}^{N-1} \bigg[\sum_{j=1}^m\Delta B_{\nu,j}^\mu\big(b_{jx}(\bX^\mu_\nu,u,t_\nu)\bXt_\nu^\mu\big)\bigg]^\top\bL^\mu_{\nu+1} 
        \\
        &\quad 
        + \frac1M \sum_{\mu=1}^M \big[\bXt^\mu_0\big]^\top \bL^\mu_0
	\end{aligned}
\end{align}
with $b_{jx}(\bX^\mu_\nu,u,t_\nu)\in\Rb^{d\times d}$. Using an index shift, we can write
\begin{align}
	\label{Eq:DiscOS-2}
	\begin{aligned}
		&\sum_{\mu=1}^M\sum_{\nu=0}^{N-1}\big[\bXt^\mu_{\nu+1}-\bXt^\mu_\nu\big]^\top\bL^\mu_{\nu+1}
  + \sum_{\mu=1}^M \big[\bXt^\mu_0\big]^\top \bL^\mu_0 \\
		& =\sum_{\mu=1}^M\sum_{\nu=0}^{N-1}\Big(\big[\bXt^\mu_{\nu+1}\big]^\top\bL^\mu_{\nu+1}-\big[\bXt^\mu_\nu\big]^\top\bL^\mu_{\nu+1}\Big)+ \sum_{\mu=1}^M \big[\bXt^\mu_0\big]^\top \bL^\mu_0
		\\
		&=\sum_{\mu=1}^M\sum_{\nu=0}^N\big[\bXt^\mu_\nu\big]^\top\bL^\mu_\nu-\sum_{\mu=1}^M\sum_{\nu=0}^{N-1} \big[\bXt^\mu_\nu\big]^\top\bL^\mu_{\nu+1}
		\\
		&=\sum_{\mu=1}^M\sum_{\nu=0}^{N-1} \big[\bL^\mu_\nu-\bL^\mu_{\nu+1}\big]^\top\bXt^\mu_\nu+\sum_{\mu=1}^M\big[\bL^\mu_N\big]^\top\bXt^\mu_N.
	\end{aligned}
\end{align}
Furthermore, we can calculate
\begin{align}
	\label{Eq:DiscOS-3}
	\begin{aligned}
		&\sum_{\mu=1}^M\sum_{\nu=0}^{N-1} \big[\big(a_x(\bX^\mu_\nu,u,t_\nu)\bXt_\nu^\mu\big)\,\Delta t\big]^\top\bL^\mu_{\nu+1}=
  \sum_{\mu=1}^M \sum_{\nu=0}^{N-1} \big[\big(a_x(\bX^\mu_\nu,u,t_\nu)^\top\bL^\mu_{\nu+1}\big)\,\Delta t\big]^\top\bXt_\nu^\mu
	\end{aligned}
\end{align}
and
\begin{align}
	\label{Eq:DiscOS-4}
	\begin{aligned}
		&\sum_{\mu=1}^M\sum_{\nu=0}^{N-1} \bigg[\sum_{j=1}^m\Delta B_{\nu,j}^\mu\big(b_{jx}(\bX^\mu_\nu,u,t_\nu)\bXt_\nu^\mu\big)\bigg]^\top\bL^\mu_{\nu+1}\\
		&\qquad=\sum_{\mu=1}^M \sum_{\nu=0}^{N-1}\bigg[\sum_{j=1}^m\Delta B_{\nu,j}^\mu\big(b_{jx}(\bX^\mu_\nu,u,t_\nu)^\top\bL^\mu_{\nu+1}\big) \bigg]^\top\bXt_\nu^\mu.
	\end{aligned}
\end{align}
Using \eqref{ExpValDisc}, inserting \eqref{Eq:DiscOS-2}-\eqref{Eq:DiscOS-4} into \eqref{eq:Derivative_L_X} and setting $\LF_\bX(\bX,u,\bL)\bXt=0$, we derive
\begin{align}
	\label{eq:Variational_ineq}
	\begin{aligned}
		&\frac{\Delta t}{M} \sum_{\mu=1}^M \sum_{\nu=0}^{N-1} \Big[\Cm\big(\Eb^M\big[\bX^h_\nu\big]\big)- \cdIndex{\nu}\Big]^\top\Cm'\big(\Eb^M\big[\bm X^h_\nu\big]\big)\bXt^\mu_\nu
		\\
		&+\frac{1}{M} \sum_{\mu = 1}^M\Big[\Cm\big(\Eb^M\big[\bX^h_N\big]\big) - \cdIndex{\nu}\Big]^\top\Cm'\big(\Eb^M\big[\bX^h_N\big]\big)\bXt^\mu_N
		\\
		&+\frac{1}{M}\sum_{\mu=1}^M \sum_{\nu=0}^{N-1} \Big[\bL^\mu_\nu-\bL^\mu_{\nu+1}-\big(a_x(\bX^\mu_\nu,u,t_\nu)^\top\bL^\mu_{\nu+1}\big)\,\Delta t\Big]^\top\bXt_\nu^\mu+\frac{1}{M}\sum_{\mu=1}^M\big[\bL^\mu_N\big]^\top\bXt^\mu_N
		\\
		&-\frac{1}{M}\sum_{\mu=1}^M \sum_{\nu=0}^{N-1} \bigg[\sum_{j=1}^m\Delta B_{\nu,j}^\mu\big(b_{jx}(\bX^\mu_\nu,u,t_\nu)^\top\bL^\mu_{\nu+1}\big) \bigg]^\top\bXt_\nu^\mu=0     
	\end{aligned}
\end{align}
for any direction $\bXt^h=\{\bXt^\mu_\nu\}\in\X^h$. 
If we choose $\bXt^h$ with $\bXt_N^h=0$, we deduce from \eqref{eq:Variational_ineq}
\begin{subequations}
	\label{Eq:Adjoint}
	\begin{align}
		\label{eq:adjoint_SDE_time}
		\begin{aligned}
			&\bL^\mu_{\nu+1}-\bL^\mu_\nu+a_x(\bX^\mu_\nu,u,t_\nu)^\top\bL^\mu_{\nu+1}\Delta t
			+\sum_{j=1}^m\Delta B_{\nu,j}^\mu\big(b_{jx}(\bX^\mu_\nu,u,t_\nu)^\top\bL^\mu_{\nu+1}\big)
			\\
			&\quad
			=\Delta t \, \Cm'\big(\Eb^M\big[\bX^h_\nu\big]\big)^\top\big(\Cm\big(\Eb^M\big[\bX^h_\nu\big]\big)-\cdIndex{\nu}\big)
			\quad\text{for }\mu=1,\ldots,M,\;\nu=0,\ldots,N-1.
		\end{aligned} 
	\end{align}
	Inserting \eqref{eq:adjoint_SDE_time} into \eqref{eq:Variational_ineq} we obtain the terminal condition
	\begin{align}
		\label{Eq:Adjoint-2}
		\bL^\mu_N=\Cm'\big(\Eb^M\big[\bX^h_N\big]\big)^\top\big(\cdIndex{\nu}-\Cm\big(\Eb^M\big[\bX^h_N\big]\big)\big)
	\end{align}
\end{subequations}
for $\mu=1,\ldots,M$. Next, we compute the partial derivative with respect to $u$ and derive
for $u\in \text{int }\UadCalib$ and for any direction $\tilde u\in\Uc$ such that $u+\tilde u \in \UadCalib$
\begin{align*}
	&\LF_u(\bX^h,u,\bL^h)\tilde u\\
	&=\kappa\,u^\top\tilde u-\frac{1}{M}\sum_{\mu=1}^M\sum_{\nu=0}^{N-1} \bigg(\big(a_u(\bX^\mu_\nu,u,t_\nu)\Delta t\big)\,\tilde u
	+\Big(\sum_{j=1}^m\Delta B_{\nu,j}^\mu b_{ju}(\bX^\mu_\nu,u,t_\nu)\Big)\tilde u\bigg)^\top\bL^\mu_{\nu+1}
	\\
	&=\bigg(\kappa u-\frac{1}{M}\sum_{\mu=1}^M \sum_{\nu=0}^{N-1}\big(a_u(\bX^\mu_\nu,u,t_\nu)^\top\bL^\mu_{\nu+1}\big)\Delta t+\Big(\sum_{j=1}^m\Delta B_{\nu,j}^\mu b_{ju}(\bX^\mu_\nu,u,t_\nu)^\top\bL^\mu_{\nu+1}\Big)\bigg)^\top\tilde u.
\end{align*}
Suppose that for given $u\in\UadCalib$ the state $\bX^h=\bX^h_u$ solves \eqref{SDE-Cal}. Moreover, $\bL^h=\bL_u^h$ is the solution to \eqref{Eq:Adjoint} with $\bX^h=\bX^h_u$. Then, it follows that (cf., e.g., \cite[Section~1.6.4]{HPUU08})
\begin{align}
	\label{GradRedGradDisc}
	\begin{aligned}
		&\nabla\hat J^h(u)=\LF_u(\bX_u^h,u,\bL_u^h)\\
		&=\kappa u-\frac{1}{M}\sum_{\mu=1}^M \sum_{\nu=0}^{N-1} \bigg(a_u(\bX^\mu_\nu,u,t_\nu)^\top\bL^\mu_{\nu+1}\Delta t+\sum_{j=1}^m\Delta B_{\nu,j}^\mu b_{ju}(\bX^\mu_\nu,u,t_\nu)^\top\bL^\mu_{\nu+1}\bigg)
		\\
		&=\kappa u-\Eb^M\bigg[ \sum_{\nu=0}^{N-1} a_u(\bX^\mu_\nu,u,t_\nu)^\top\bL^\mu_{\nu+1}\Delta t+\sum_{j=1}^m\Delta B_{\nu,j}^\mu b_{ju}(\bX^\mu_\nu,u,t_\nu)^\top\bL^\mu_{\nu+1}\bigg].
	\end{aligned}
\end{align}

All together, we can formulate first-order necessary optimality conditions for \eqref{OptCal} in \cref{thm:FirstOrderCalibration}.

\begin{theorem}
    \label{thm:FirstOrderCalibration}
    Suppose that \eqref{SDEDisc} and \eqref{Eq:Adjoint} are uniquely solvable and moreover that $\bar u\in\UadCalib$ is a local solution to \eqref{OptCal}. Then, it follows
%
%
\begin{multline}
	\label{DONOC:Disc}
	\bigg\langle\kappa\bar u-\Eb^M\bigg[ \sum_{\nu=0}^{N-1} a_u(\bX^\mu_\nu,u,t_\nu)^\top\bL^\mu_{\nu+1}\Delta t\bigg],u-\bar u\bigg\rangle_\U
	\\
	-\bigg\langle\Eb^M\bigg[\sum_{\nu=0}^{N-1} \sum_{j=1}^m\Delta B_{\nu,j}^\mu b_{ju}(\bX^\mu_\nu,u,t_\nu)^\top\bL^\mu_{\nu+1}\bigg],u-\bar u\bigg\rangle_\U\ge0\quad\text{for all }u\in\UadCalib,
\end{multline}
%
where $\bX^h\in\X^h$ solves \eqref{SDEDisc} for $u=\bar u$ and $\bL^h\in\X^h$ is the solution to \eqref{Eq:Adjoint} with $u=\bar u$.
\end{theorem}

In \cref{rem:ComparisonGradients}, we compare 
the discrete gradient \eqref{GradRedGradDisc} with the continuous one that will be derived in \cref{Sec:RedGradCont}.

\subsection{Continuous control inputs}
\label{sec:D6.2}

Let us just explain briefly how we derive a discretization of the general optimization problem \eqref{eq:SDE_Opti_OU}. Again, we set $h=(M,N)$ to simplify the notation. We approximate a time-dependent control $u\in\U$ by the piecewise constant function
\begin{align*}
	u^h(t)=
	\sum_{\nu=0}^{N-1} \mathrm u_\nu\chi_{\mathscr I_\nu}(t)\in\Rb^r
	\quad \text{for } t\in[0,T]
\end{align*}
with $\mathrm u_0,\ldots,\mathrm u_{N-1}\in\Rb^r$ and characteristic functions $\chi_{\mathscr I_\nu}:[0,T]\to\{0,1\}$ defined in \eqref{def:indicator_function}. Thus, each control $u^h$ is characterized by a coefficient matrix $\bU^h=[\mathrm u_0|\ldots|\mathrm u_{N-1}]\in\Rb^{r\times N}$. We set $\U^h=\Rb^{r\times N}$ which is a Hilbert space endowed with the (weighted) inner product
\begin{align*}
	{\langle\bU^h,\bUt^h\rangle}_{\U^h}
	=\Delta t\sum_{i=1}^r\sum_{\nu=0}^{N-1} \bU^h_{i\nu}\bUt^h_{i\nu}
	=\Delta t \sum_{\nu=0}^{N-1} \mathrm u_\nu^\top\tilde{\mathrm u}_\nu
\end{align*}
for $\bU^h=[\mathrm u_0|\ldots|\mathrm u_{N-1}]$ and $\bUt^h=[\tilde{\mathrm u}_0|\ldots|\tilde{\mathrm u}_{N-1}]$. The associated induced norm is given as $\|\bU^h\|_{\U^h}=\sqrt{\Delta t}\,\|\bU^h\|_\Fr$ .

Clearly, $u^h\in\Uad$ holds true provided $\mathrm u_\nu\in[\ua,\ub]$ for $\nu=0,\ldots,N-1$. Thus, we define the set of admissible control coefficient matrices as
\begin{align*}
	\Uadh=
	\big\{\bU^h=
	\big[\mathrm u_0|\ldots|\mathrm u_{N-1}\big]\in\U^h\,\big|\,\mathrm u_\nu\in[\ua,\ub]
	\text{ for }\nu=0,\ldots,N-1\big\}.
\end{align*}
%
%
%

Next we introduce a discretization of the state equation \eqref{SDE} as follows: For any given control matrix $\bU^h=[\mathrm u_0|\ldots|\mathrm u_{N-1}]\in\Uadh$ find $\bX^h\in\X^h$ satisfying
\begin{align}
	\label{SDEDisc2}
	\begin{aligned}
		\bX^\mu_{\nu+1}&=\bX^\mu_\nu +a(\bX^\mu_\nu,\mathrm u_\nu,t_\nu)\,\Delta t+b(\bm X^\mu_\nu,\mathrm u_\nu,t_\nu)\,\Delta \bm B_\nu^\mu
		\\
		&\hspace{15mm}
		\text{for }\mu=1,\ldots,M
		\text{ and }\nu=0,\ldots,N-1,\\
		\bX_0^\mu&= \bX_\circ\hspace{5.4mm}\text{for }\mu=1,\ldots,M.
	\end{aligned}
\end{align}    
For any $\bU^h\in\Uadh$ we suppose that the explicit difference scheme \eqref{SDEDisc2} admits a unique solution $\bX^h\in\X^h$. Thus, the discrete non-linear solution operator
\begin{align*}
	\Sm^h:\Uadh\to\X^h,
    \quad&\bX^h=\Sm^h(\bU^h)\text{ is the unique solution to \eqref{SDEDisc2} for }\bU^h\in\Uadh
    \\
    &\text{ given the $M$ pre-computed Brownian increments}
\end{align*}
is well defined. Now let us formulate a discrete version of \eqref{eq:SDE_Opti_OU}:
\begin{align}
	\label{eq:SDE_Opti_OU_discrete2}
	\hat J^h(\bU^h)\quad\text{s.t.}\quad \bU^h\in\Uadh,
\end{align}
where the discrete reduced cost functional $\hat J^h$ is given as
\begin{align*}
	\hat J^h(\bU^h)=J^h(\bX^h,\bU^h)\quad\text{for }\bU^h\in\Uadh\text{ and }\bX^h=\Sm^h(\bU^h)
\end{align*}
with
\begin{align*}
	J^h(\bX^h,\bU^h)
	\colonequals\frac{1}{2}\bigg(\Delta t\sum_{\nu=0}^{N-1} \big\|\Cm\big(\Eb^M\big[\bm X^h_\nu\big]\big) 
	- \cdIndex{\nu}\big\|_{\Rb^\ell}^2+\left\|\Cm\big(\Eb^M\big[\bX^h_N\big]\big) - \cdIndex{N}\right\|_{\Rb^\ell}^2+\kappa\,{\|\bU^h\|}_{\U^h}^2\bigg)
\end{align*}
and $\bX^h_\nu=\{\bX_\nu^\mu\}_{\mu=1}^M$ for $\nu\in\{0,\ldots,N\}$.


\section{First-order optimality conditions}
\label{sec:Gradient}

In this section, we derive a formula for the gradient of the reduced cost functional $\hat J:\Uad\to\Rb$ which allows us to formulate first-order necessary optimality conditions for the continuous problem. Here we focus on the choice $\U=L^2(0,T;\Rb^r)$. The derivation for the calibration problem is analogously.

First let us introduce the constraint function $e:\X\times\Uad\to\X$ by
\begin{align*}
	e(\bX,u) &
	\colonequals\bX(\cdot)-\bX_\circ-\int_0^{(\cdot)}a(\bX(s),u(s),s)\,\mathrm ds-\int_0^{(\cdot)}b(\bX(s),u(s),s)\,\mathrm d\bB(s)
\end{align*}
for $(\bX,u)\in\X\times\Uad$. 
Here, $(\cdot)$ is a placeholder denoting the dependence on $t$.
Using the mapping $e$, the constraint \eqref{SDE} can be expressed as the operator equation
\begin{align}
	\label{E_OE}
	e(\bm X,u)=0\quad\text{in }\X.
\end{align}
Due to \cref{Definition:SDE}, we infer from \eqref{E_OE} that $\bX$ is a solution to \eqref{SDE} for any given $u\in\Uad$. 
Moreover, \eqref{eq:SDE_Opti_OU} can be expressed equivalently as a non-convex, infinite-dimensional, constrained optimization problem as follows:
\begin{align}
	\label{OptConstrained}
	\min J(\bm X,u)\quad\text{s.t.}\quad(\bm X,u)\in\X\times \Uad\text{ and }e(\bm X,u)=0\text{ in }\X.
\end{align}
It follows from \cref{Theorem:ExUni,Th:OptExSol} that \eqref{OptConstrained} admits a (global) optimal solution $(\bar\bX,\bar u)\in\X\times\Uad$ with $\bar X=\bar S(\bar u)$ provided that \cref{Assumption:coefficients,Assumption:continuity_C} hold.

\subsection{Fr\'echet differentiablity of the constraints}

To derive the derivative of the constraint function $e$, we have to ensure \cref{Assumption:Diffbarkeit}-\ref{AssumptionState}.

\begin{remark}
	\label{Rem:TensorProd}
	Let $(x,u)\in\Rb^d\times\Uad$. For any $s\in[0,T]$ we have
	\begin{align*}
		b(x,u(s),s)=\big[b_1(x,u(s),s)\big|\ldots\big|b_m(x,u(s),s)\big]\in\Rb^{d\times m}.
	\end{align*}
	With \cref{Assumption:Diffbarkeit}-\ref{AssumptionState} holding the partial derivative $b_x(x,u(s),s)$ is a linear map from $\Rb^d$ to $\Rb^{d\times m}$ for any $s\in[0,T]$. Consequently, $b_x(x,u(s),s)\tilde x$ is a $(d\times m)$ matrix for all $\tilde x\in\Rb^d$ and we write
	\begin{align*}
		b_x(x,u(s),s)\tilde x&=\big[b_{1x}(x,u(s),s)\tilde x\big|\ldots|b_{mx}(x,u(s),s)\tilde x\big]=\sum_{i=1}^db_{x_i}(x,u(s),s)\tilde x_i\in\Rb^{d\times m}
	\end{align*}
	for any $s\in[0,T]$, where $b_{jx}(x,u(s),s)$ ($1\le j\le m$) denotes the partial derivative of $b_j$ with respect to $x$. Analogously, we write 
	\begin{align*}
		b_u(x,u(s),s)\tilde u&=\big[b_{1u}(x,u(s),s)\tilde u\big|\ldots\big|b_{mu}(x,u(s),s)\tilde u\big]=\sum_{i=1}^mb_{u_i}(x,u(s),s)\tilde u_m\in\Rb^{d\times m}
	\end{align*}
	for any $s\in[0,T]$ and $\tilde u\in\U$.
\end{remark}

\begin{lemma}
	\label{Le:eFrDer}
	Suppose that \cref{Assumption:Diffbarkeit}-\ref{AssumptionState} holds. 
	Then the mapping $e$ is continuously Fr\'echet differentiable on $\X\times\Uad$. 
	For any $(\bX,u)\in\X\times\Uad$ the partial Fr\'echet derivatives $e_\bX$ and $e_u$ are given by
	\begin{align*}
		e_\bX(\bX,u)\bX^\delta&=\bX^\delta(\cdot)-\int_0^{(\cdot)}
		a_x(\bX(s),u(s),s)\bX^\delta(s)\,\mathrm ds\\
		&\quad-\int_0^{(\cdot)}b_x(\bX(s),u(s),s)\bX^\delta(s)\,\mathrm d\bB(s)&&\text{for }\bX^\delta\in\X,
		\\
		e_u(\bX,u)u^\delta&=-\int_0^{(\cdot)}
		a_u(\bX(s),u(s),s)u^\delta(s)\,\mathrm ds\\
		&\quad-\int_0^{(\cdot)}b_u(\bX(s),u(s),s)u^\delta(s)\,\mathrm d\bB(s)&&\text{for }u^\delta\in\U,
	\end{align*}
	respectively, where we have used the notation introduced in \cref{Rem:TensorProd}.
\end{lemma}

\begin{proof}
	The claim follows by similar arguments as the proofs of Lemmas 6.3.3 and 6.3.4 in \cite{Gross2015applications}. 
	In fact, by \cref{Assumption:Diffbarkeit}-\ref{AssumptionState} it holds that the coefficients $a$ and $b$ of the SDE are continuously Fr\'echet differentiable. 
	Now let $\bX^\delta \in \X$ and apply Taylor's formula to $a$ and $b$ with respect to $\bX$
	\begin{align*}
		a(\bX(t)+\bX^\delta(t), u(t),t)=a(\bX(t),u(t),t)+a_x(\bX(t),u(t),t)\bX^\delta(t)+{\scriptstyle \Om}\big({\|\bX^\delta(t)\|}_{\Rb^d}\big)
	\end{align*}
	and
	\begin{align*}
		b(\bX(t)+\bX^\delta(t), u(t),t)=b(\bX(t),u(t),t)+b_x(\bX(t),u(t),t)\bX^\delta(t)+{\scriptstyle \Om}\big({\|\bX^\delta(t)\|}_{\Rb^d}\big).
	\end{align*}
	Now, we obtain for $e_\bX$ as defined in the statement of the theorem
	\begin{align*}
		&{\|e(\bX + \bX^\delta,u) - e(\bX,u) - e_\bX(\bX,u)\,\bX^\delta \|}_\X^2\\
		&=\mathbb E\bigg[\int_0^{(\cdot)}{\|a(\bX(s)+\bX^\delta(s),u(s),s)-a(\bX(s),u(s),s)-a_x(\bX(s),u(s),s)\,\bX^\delta(s)\|}^2_{\Rb^d} \rmd s\\
		&\hspace{10mm}+\int_0^{(\cdot)}{\|b(\bX(s)+\bX^\delta(s),u(s),s)-b(\bX(s),u(s),s)-b_x(\bX(s),u(s),s)\,\bX^\delta(s)\|}^2_{\mathrm{F}} \rmd s\bigg]\\
		&=\mathbb E\bigg[\int_0^{(\cdot)}{\scriptstyle \Om}\big({\|\bX^\delta(s)\|}_{\Rb^d}^2\big)\,\mathrm ds\bigg]={\scriptstyle \Om}\left({\|\bX^\delta\|}_{C([0,T];L^2(\Omega;\Rb^d))}^2\right)={\scriptstyle \Om}\big({\|\bX^\delta\|}_\X^2\big),
	\end{align*}
	where we have used \cref{Le:ExtXb}. An analogous result holds for the Fr\'echet differentiability with respect to $u$. 
	More precisely, for any $u^\delta \in \U$ with $u+u^\delta\in\Uad$ it holds that
	\begin{align*}
		{\|e(\bX,u+u^\delta)-e(\bX,u)- e_u(\bX,u)\,u^\delta\|}_\X= {\scriptstyle \Om}\big({\|u^\delta\|}_\U\big).
	\end{align*}
\end{proof}

In \cref{Co:Bijectivity} we will prove that the operator $e_\bX(\bX,u):\X\to\X$ has a bounded inverse for $(\bX,u)\in \X\times\Uad$. 
For that purpose we set for all $x=(x_i)\in\Rb^d$ and f.a.a. $t\in[0,T]$
\begin{align}
	\label{OpExLinSDE}
	\begin{aligned}
		\mathfrak a(x,t)&=a_x(\bX(\cdot\,,t),u(t),t)x=\sum_{i=1}^da_{x_i}(\bX(\cdot\,,t),u(t),t)x_i\in\Rb^d&&\text{in }\Omega\text{ a.s.},\\
		\mathfrak b(x,t)&=b_x(\bX(\cdot\,,t),u(t),t)x=\sum_{i=1}^db_{x_i}(\bX(\cdot\,,t),u(t),t)x_i\in\Rb^{d\times m}&&\text{in }\Omega\text{ a.s.},        
	\end{aligned}
\end{align}
where we have used the notation introduced in \cref{Rem:TensorProd} and `a.s.' stands for `almost surely'. Note that both $\mathfrak a(\cdot\,,t)$ and $\mathfrak b(\cdot\,,t)$ are linear mappings on $\Rb^d$ f.a.a. $t\in[0,T]$ and in $\Omega$ a.s. However, due to the dependence on $\bX$ the functions $\mathfrak a$ and $\mathfrak b$ contain random-valued coefficients. Therefore, we can not utilize \cref{A:LinSDE} together with \cite[Theorem 6.14]{YongZhou1999StochasticControls} in the proof of \cref{Co:Bijectivity} below; cf. \cref{Rem:LinSDE}. For that reason we introduce the following hypothesis.

\begin{assumption}
	\label{A:5.3}  
	Let $(\bX,u)\in\X\times\Uad$ hold and the mappings $\mathfrak a$, $\mathfrak b$ be given as in \eqref{OpExLinSDE}. Suppose that there exists a constant $L_\mathfrak{ab}\ge0$ satisfying
	\begin{subequations}
		\label{Eq:A:5.3-1}
		\begin{align}
			\label{Eq:A:5.3-1:a}
			{\|a_x(\bX(t),u(t),t)\|}_\Fr&\le L_\mathfrak{ab}&&\text{f.a.a. }t\in[0,T]\text{ and in }\Omega\text{ a.s.},\\
			\label{Eq:A:5.3-1:b}
			\bigg(\sum_{i=1}^d{\|b_{x_i}(\bX(t),u(t),t)\|}_\Fr^2\bigg)^{1/2}&\le L_\mathfrak{ab}&&\text{f.a.a. }t\in[0,T]\text{ and in }\Omega\text{ a.s.}
		\end{align}  
	\end{subequations}
	Further, we suppose that
	\begin{align}
		\label{Eq:A:5.3-2}
		\Eb\bigg[\int_0^T{\|\mathfrak a(0,t)\|}_{\Rb^d}^2+{\|\mathfrak b(0,t)\|}_\Fr^2\,\mathrm dt\bigg]\le L_\mathfrak{ab}^2,
	\end{align}
	i.e., $\mathfrak a(0,\cdot)\in\Lb^2_\mathscr F(\Rb^d)$ and $\mathfrak b(0,\cdot)\in\Lb^2_\mathscr F(\Rb^{d\times m})$ hold.
\end{assumption}

In the next lemma, we prove that \cref{A:5.3} implies condition \textbf{(RC)} from \cite[p.~49]{YongZhou1999StochasticControls} which is given by \eqref{Eq:A:5.3-2} and \eqref{Eq:A:5.3-3} below. Notice that condition \textbf{(RC)} ensures the existence of a unique solution to linear SDEs with random coefficient functions.
This fact will be used in \cref{Co:Bijectivity} below in an essential manner.

\begin{lemma}
	\label{Lem:AA:5.3}
	Let $(\bX,u) \in \X\times\Uad$ hold and the mappings $\mathfrak a$, $\mathfrak b$ be given as in \eqref{OpExLinSDE}. Suppose that \cref{A:5.3} holds true. Then we have
	\begin{subequations}
		\label{Eq:A:5.3-3}
		\begin{align}
			\label{Eq:A:5.3-3:a}
			{\|\mathfrak a(\varphi(t),t)-\mathfrak a(\phi(t),t)\|}_{\Rb^d}&\le L_\mathfrak{ab}\,{\|\varphi-\phi\|}_{C([0,T];\Rb^d)}&&\text{f.a.a. }t\in[0,T]\text{ and in }\Omega\text{ a.s.},\\
			\label{Eq:A:5.3-3:b}
			{\|\mathfrak b(\varphi(t),t)-\mathfrak b(\phi(t),t)\|}_\Fr&\le L_\mathfrak{ab}\,{\|\varphi-\phi\|}_{C([0,T];\Rb^d)}&&\text{f.a.a. }t\in[0,T]\text{ and in }\Omega\text{ a.s.},
		\end{align}
	\end{subequations}
	where $\varphi,\phi$ are chosen arbitrarily in $C([0,T];\Rb^d)$.
\end{lemma}

\begin{proof}
	For arbitrarily $\varphi,\phi\in C([0,T];\Rb^d)$ and $u\in\Uad$ we get f.a.a. $t\in[0,T]$ and in $\Omega$ a.s.
	\begin{align*}
		&{\|\mathfrak a(\varphi(t),t)-\mathfrak a(\phi(t),t)\|}_{\Rb^d}=\left\|a_x(\bX(t),u(t),t)\big(\varphi(t)-\phi(t)\big)\right\|_{\Rb^d}
        \\
		&\quad\le{\|a_x(\bX(t),u(t),t)\|}_\Fr{\|\varphi(t)-\phi(t)\|}_{\Rb^d}\le{\|a_x(\bX(t),u(t),t)\|}_\Fr{\|\varphi-\phi\|}_{C([0,T];\Rb^d)},
	\end{align*}
	where $\bX$ is any element in $\X$. Utilizing \eqref{Eq:A:5.3-1:a} we infer that \eqref{Eq:A:5.3-3:a} is valid. Analogously, we derive f.a.a. $t\in[0,T]$ and in $\Omega$ a.s.
	\begin{align*}
		&{\|\mathfrak b(\varphi(t),t)-\mathfrak b(\phi(t),t)\|}_\Fr=
		\bigg\|\sum_{i=1}^db_{x_i}(\bX(t),u(t),t)\big(\varphi_i(t)-\phi_i(t)\big)\bigg\|_\Fr\\
		&\quad\le\sum_{i=1}^d{\|b_{x_i}(\bX(t),u(t),t)\|}_\Fr\big|\varphi_i(t)-\phi_i(t)\big|\\
		&\quad\le\bigg(\sum_{i=1}^d{\|b_{x_i}(\bX(t),u(t),t)\|}_\Fr^2\bigg)^{1/2}{\|\varphi(t)-\phi(t)\|}_{\Rb^d}\\
		&\quad\le\bigg(\sum_{i=1}^d{\|b_{x_i}(\bX(t),u(t),t)\|}_\Fr^2\bigg)^{1/2}{\|\varphi-\phi\|}_{C([0,T];\Rb^d)},
	\end{align*}
	where again $\bX$ is any element in $\X$. Hence, \eqref{Eq:A:5.3-1:b} implies \eqref{Eq:A:5.3-3:b}.
\end{proof}

\begin{proposition}
	\label{Co:Bijectivity}
	Suppose that \cref{Assumption:Diffbarkeit,A:5.3} hold, and let $(\bX,u)\in\X\times\Uad$. Then the linear operator $e_\bX(\bX,u):\X\to\X$ is bijective and its inverse $e_\bX(\bX,u)^{-1}$ is a bounded operator. 
\end{proposition}

\begin{proof}
	Suppose that $(\bX,u)\in\X\times\Uad$ and $\bY\in\X$ are given arbitrarily. Then there exist a measurable initial condition $\bY_\circ$ and measurable coefficients $\tilde{\mathfrak a}\in\Lb^2_\mathscr F(\Rb^d)$, $\tilde{\mathfrak b}=[\tilde{\mathfrak b}_1|\ldots|\tilde{\mathfrak b}_m]\in\Lb^2_\mathscr F(\Rb^{d\times m})$ with
	\begin{align}
		\label{Y_Repr}
		\bY(t)=\bY_\circ+\int_0^t\tilde{\mathfrak a}(s)\,\mathrm ds+\int_0^t\tilde{\mathfrak b}(s)\,\mathrm d\bB(s)\quad\text{for }t\in[0,T].
	\end{align}
	We have to show that there exists a unique $\bX^\delta\in\X$ satisfying
	\begin{align}
		\label{OpEqEX}
		e_\bX(\bX,u)\bX^\delta=\bY\quad\text{in }\X.
	\end{align}
	By \cref{Le:eFrDer} and \eqref{Y_Repr}, equation \eqref{OpEqEX} is equivalent to
	\begin{align*}
		&\bX^\delta(\cdot)-\int_0^{(\cdot)}
		\underbrace{a_x(\bX(s),u(s),s)\bX^\delta(s)}_{\equalscolon\mathfrak a(\bX^\delta(s),s)\in\Lb^2_\mathscr F(\Rb^d)}\,\mathrm ds
		-\int_0^{(\cdot)}\underbrace{b_x(\bX(s),u(s),s)\bX^\delta(s)}_{\equalscolon\mathfrak b(\bX^\delta(s),s)\in\Lb^2_\mathscr F(\Rb^{d\times m})}\,\mathrm d\bB(s)\\
		&=\bY_\circ+\int_0^t\tilde{\mathfrak a}(s)\,\mathrm ds+\int_0^t\tilde{\mathfrak b}(s)\,\mathrm d\bB(s)\quad\text{for }t\in[0,T].
	\end{align*}
	We set
	\begin{align}
		\label{Eq:Setting_ab}
		\begin{aligned}
			\hat{\mathfrak a}(\bX^\delta(\cdot),\cdot)&=\mathfrak a(\bX^\delta(\cdot),\cdot)+\tilde{\mathfrak a}\in\Lb^2_\mathscr F(\Rb^d),\\
			\hat{\mathfrak b}(\bX^\delta(\cdot),\cdot)&=\mathfrak b(\bX^\delta(\cdot),\cdot)+\tilde{\mathfrak b}\in\Lb^2_\mathscr F(\Rb^{d\times m}).
		\end{aligned}
	\end{align}
	Then $\bX^\delta\in\X$ is the solution to the linear SDE with random coefficients
	\begin{align*}
		\left\{
		\begin{aligned}
			\mathrm d\bX^\delta(t)&=\hat{\mathfrak a}(\bX^\delta(t),t)\,\mathrm dt+\hat{\mathfrak b}(\bX^\delta(t),t)\,\mathrm d\bB(t)\quad\text{for }t\in(0,T],\\
			\bX^\delta(0)&=\bY_\circ.
		\end{aligned}
		\right.
	\end{align*}
	Utilizing \cref{A:5.3} and \cref{Lem:AA:5.3} it follows from \cite[Theorem 6.16]{YongZhou1999StochasticControls} that there exists a unique $\bX^\delta\in\X$ satisfying $\bX^\delta=e_\bX(\bX,u)^{-1}\bY\in\X$. Thus, $e_\bX(\bX,u)^{-1}$ is bijective. Moreover, we infer from \cite[Theorem 6.16]{YongZhou1999StochasticControls} that there is a constant $C>0$ such that
	\begin{align}
		\label{XdeltaBound}
		\sup_{t\in[0,T]}\Eb\big[{\|\bX^\delta(t)\|}_{\Rb^d}^2\big]\le C\big(1+\Eb\big[{\|\bY_\circ\|}_{\Rb^d}^2\big)
	\end{align}
	for all $\bY\in\X$. Applying \eqref{Eq:Setting_ab} and \cref{Lem:AA:5.3} we estimate
	\begin{align*}
		&\int_0^T {\|\hat{\mathfrak a}(\bX^\delta(t),t)\|}_{\Rb^d}^2\,\mathrm dt\le2\int_0^T {\|\tilde{\mathfrak a}(t)\|}_{\Rb^d}^2+{\|\mathfrak a(\bX^\delta(t),t)\|}_{\Rb^d}^2\,\mathrm dt\\
		&\le2\int_0^T {\|\tilde{\mathfrak a}(t)\|}_{\Rb^d}^2+2\,{\|\mathfrak a(\bX^\delta(t),t)-\mathfrak a(0,t)\|}_{\Rb^d}^2+2\,{\|\mathfrak a(0,t)\|}_{\Rb^d}^2\,\mathrm dt\\
		&\le2\int_0^T {\|\tilde{\mathfrak a}(t)\|}_{\Rb^d}^2\,\mathrm dt+4\int_0^TL_\mathfrak{ab}^2\,{\|\bX^\delta(t)\|}_{\Rb^d}^2+{\|\mathfrak a(0,t)\|}_{\Rb^d}^2\,\mathrm dt\quad\text{in }\Omega\text{ a.s.}
	\end{align*}
	Analogously, we find that
	\begin{align*}
		\int_0^T {\|\hat{\mathfrak b}(\bX^\delta(t),t,\cdot\,)\|}_\Fr^2\,\mathrm dt\le2\int_0^T {\|\tilde{\mathfrak b}(t)\|}_\Fr^2\,\mathrm dt+4\int_0^TL_\mathfrak{ab}^2\,{\|\bX^\delta(t)\|}_{\Rb^d}^2+{\|\mathfrak b(0,t)\|}_\Fr^2\,\mathrm dt\text{ in }\Omega\text{ a.s.}
	\end{align*}
	Let $\bY\in\X$ with $\|\bY\|_\X\le1$ be chosen arbitrarily. 
	Then, $\bY$ has a representation as in \eqref{Y_Repr} so that we have
	\begin{align}
		\label{Ybound}{\|\bY\|}_\X^2=\Eb\big[{\|\bY_\circ\|}_{\Rb^d}^2\big]+\Eb\bigg[\int_0^T {\|\tilde{\mathfrak a}(t)\|}_{\Rb^d}^2+{\|\tilde{\mathfrak b}(t)\|}_\Fr^2\,\mathrm dt\bigg]\le 1.
	\end{align}
	Utilizing \cref{A:5.3} as well as equations \eqref{Ybound} and \eqref{XdeltaBound} it follows that
	\begin{align*}
		&{\|e_\bX(\bX,u)^{-1}\bY\|}_\X^2={\|\bX^\delta\|}_\X^2\\
		&=\Eb\big[{\|\bY_\circ\|}_{\Rb^d}^2\big]+\Eb\bigg[\int_0^T {\|\hat{\mathfrak a}(\bX^\delta(t),t)\|}_{\Rb^d}^2+{\|\hat{\mathfrak b}(\bX^\delta(t),t)\|}_\Fr^2\,\mathrm dt\bigg]\\
		&\le2\bigg(\Eb\big[{\|\bY_\circ\|}_{\Rb^d}^2\big]+\Eb\bigg[\int_0^T {\|\tilde{\mathfrak a}(t)\|}_{\Rb^d}^2+{\|\tilde{\mathfrak b}(t)\|}_\Fr^2\,\mathrm dt\bigg]\bigg)+8L_\mathfrak{ab}^2\int_0^T\Eb\big[{\|\bX^\delta(t)\|}_{\Rb^d}^2\big]\,\mathrm dt\\
		&\quad+4\Eb\bigg[\int_0^T{\|\mathfrak a(0,t)\|}_{\Rb^d}^2+{\|\mathfrak b(0,t)\|}_\Fr^2\,\mathrm dt\bigg]\\
		&\le 2+8L_\mathfrak{ab}^2T\sup_{t\in[0,T]}\Eb\big[{\|\bX^\delta(t)\|}_{\Rb^d}\big]+4L_\mathfrak{ab}^2\le c_1+c_2\left(1+\Eb\big[{\|\bY_\circ\|}_{\Rb^d}^2\big]\right)\le c_1+2c_2
	\end{align*}
	for constants $c_1\colonequals 2+4L_\mathfrak{ab}^2$ and $c_2\colonequals 8L_\mathfrak{ab}^2TC$ which do not depend on $\bY$. Consequently,
	\begin{align*}
		{\|e_\bX(\bX,u)^{-1}\|}_{L(\X)}=\sup_{\|\bY\|_\X\le 1}{\|e_\bX(\bX,u)^{-1}\bY\|}_\X\le c_1+2c_2,
	\end{align*}
	where $L(\X)$ denotes the Banach space of all linear and bounded mappings from $\X$ to $\X$ supplied by the usual operator norm. This implies that the bijective mapping $e_\bX(\bX,u)^{-1}$ is also bounded.
\end{proof}

\begin{remark}
	It follows from \cref{Co:Bijectivity} that the Fr\'echet derivative
	\begin{align*}
		e'(\bX,u)=\left(e_\bX(\bX,u) \, \vert \, e_u(\bX,u) \right):\X\times\U\to \X
	\end{align*}
	is surjective. This implies a constraint qualification which allows us to formulate first-order necessary optimality conditions for \eqref{eq:SDE_Opti_OU}; cf., e.g., \cite[p.~243]{Lue69}.
\end{remark}

\subsection{Fr\'echet differentiablity of the objective}

Next, we turn to the cost functional defined in \eqref{Ex:Cost}; c.f. \cite[Lemma 6.3.5]{Gross2015applications}.

\begin{lemma}
	\label{Le:JFrDer}
	Let \cref{Assumption:Diffbarkeit}-\ref{AssumptionCost} be satisfied. 
	Then the cost functional $J$ is continuously Fr\'echet differentiable at every $(\bX,u)\in\X\times\Uad$. Introducing
	\begin{align}
		\label{Formulas_g}
		\begin{aligned}
			\bm g_1(t)&=\Cm'\big(\Eb[\bX(t)]\big)^\top\big(\Cm(\Eb[\bX(t)])-\cd(t)\big)\in\Rb^d\quad \text{for }t\in[0,T],\\
			\bm g_2&=\Cm'\big(\Eb[\bX(T)]\big)^\top\big(\Cm(\Eb[\bX(T)])-\cdT\big)\in\Rb^d
		\end{aligned}
	\end{align}
	we have
	\begin{subequations}
		\begin{align}
			\label{J_PartDerX} 
			{\langle J_\bX(\bX,u),\bX^\delta\rangle}_{\X',\X}&=\Eb\bigg[\int_0^T\bm g_1(t)^\top \bX^\delta(t)\,\mathrm dt\bigg]+\Eb\big[\bm g_2^\top \bX^\delta(T)\big]&&\text{for all }\bX^\delta\in\X,\\
			\label{J_PartDerU} 
			{\langle J_u(\bX,u),u^\delta\rangle}_{\U',\U}&=\kappa\int_0^Tu(t)^\top u^\delta(t)\,\mathrm dt&&\text{for all }u^\delta\in\U.
		\end{align}
	\end{subequations}
\end{lemma}

\begin{proof}
	Notice that the Jacobian $\Cm'(x)$ belongs to $\Rb^{\ell\times d}$ for every $x\in\Rb^d$. 
	Let $(\bX,u)\in\X\times\Uad$ be chosen arbitrarily.	Then, it follows that
	\begin{align*}
		{\langle J_\bX(\bX,u),\bX^\delta\rangle}_{\X',\X}&=\int_0^T\big(\Cm(\Eb[\bX(t)])-\cd(t)\big)^\top \Cm'\big(\Eb[\bX(t)]\big)\Eb\big[\bX^\delta(t)\big]\,\mathrm dt\\
		&\quad+\big(\Cm(\Eb[\bX(T)])-\cdT\big)^\top \Cm'\big(\Eb[\bX(T)]\big)\Eb\big[\bX^\delta(T)\big]\\
		&=\int_0^T\bm g_1(t)^\top \Eb\big[\bX^\delta(t)\big]\,\mathrm dt+\bm g_2^\top \Eb\big[\bX^\delta(T)\big]\\
		&=\Eb\bigg[\int_0^T\bm g_1(t)^\top \bX^\delta(t)\,\mathrm dt\bigg]+\Eb\big[\bm g_2^\top \bX^\delta(T)\big]\quad\text{for all }\bX^\delta\in\X,
	\end{align*}
	which immediately gives \eqref{J_PartDerX}. The proof of \eqref{J_PartDerU} is straightforward. 
\end{proof}

\subsection{Derivation of the gradient for the reduced cost}
\label{Sec:RedGradCont}

In the sequel we follow \cite[Section~1.6.2]{HPUU08} and \cite[Section~6.3]{Gross2015applications}. We suppose that \cref{Assumption:Diffbarkeit,A:5.3} hold. Let $u\in\Uad$ be given and $\bX_u=\mathcal S(u)\in\X$. 
Then, $e(\mathcal S(u),u)=0$ is valid in $\X$. 
Due to \cref{Le:eFrDer}, the mapping $e$ is Fr\'echet differentiable. 
Thus, we have
\begin{align}
	\label{OptCond-1}
	\begin{aligned}
		0=\frac{\mathrm d}{\mathrm du}\big[e(\mathcal S(u),u)\big]u^\delta
		&=
		e_\bX(\mathcal S(u),u)\mathcal S'(u)u^\delta+e_u(\mathcal S(u),u)u^\delta\quad\text{in }\X
	\end{aligned}
\end{align}
for all $u^\delta\in\U$. 
From \cref{Co:Bijectivity} and \eqref{OptCond-1}, we derive the formula
\begin{align}
	\label{OptCond-2}
	\mathcal S'(u)=-e_\bX(\mathcal S(u),u)^{-1}e_u(\mathcal S(u),u)\quad\text{in }L(\U,\X'),
\end{align}
where $L(\U,\X')$ stands for the Banach space of all linear and bounded operators from $\U$ to $\X$. 
Furthermore, it follows that
\begin{align*}
	{\langle\hat J'(u),u^\delta\rangle}_{\U',\U}&=\frac{\mathrm d}{\mathrm dt} \big[J(\mathcal S(u),u)\big]u^\delta={\langle J_\bX(\mathcal S(u),u),\mathcal S'(u)u^\delta\rangle}_{\X',\X}+{\langle J_u(\mathcal S(u),u),u^\delta\rangle}_{\U',\U}\\
	&={\langle \mathcal S'(u)^*J_\bX(\mathcal S(u),u)+J_u(\mathcal S(u),u),u^\delta\rangle}_{\U',\U}
\end{align*}
for all $u^\delta\in\U$, which yields directly
\begin{align}
	\label{OptCond-3}
	\hat J'(u)=\mathcal S'(u)^*J_\bX(\mathcal S(u),u)+J_u(\mathcal S(u),u)\quad\text{in }\U'.
\end{align}
Next, we introduce the Lagrange multiplier $\mathcal G\in\X'$ by
\begin{align}
	\label{OptCond-4}
	e_\bX(\mathcal S(u),u)^*\mathcal G=-J_\bX(\mathcal S(u),u)\quad\text{in }\X',
\end{align}
where $e_\bX(\mathcal S(u),u)^*:\X'\to\X'$ stands for the dual of $e_\bX(\mathcal S(u),u):\X\to\X$ satisfying
\begin{align*}
	{\langle e_\bX(\mathcal S(u),u)^*\mathcal F,\bX^\delta\rangle}_{\X',\X}={\langle\mathcal F,e_\bX(\mathcal S(u),u)\bX^\delta\rangle}_{\X',\X}\quad\text{for all }(\mathcal F,\bX^\delta)\in\X'\times\X.
\end{align*}
It follows from \eqref{OptCond-4} that
\begin{align}
	\label{OptCond-5}
	\mathcal G=-e_\bX(\mathcal S(u),u)^{-*}J_\bX(\mathcal S(u),u)\quad\text{in }\X'.
\end{align}
In \eqref{OptCond-5}, we denote by $e_\bX(\mathcal S(u),u)^{-*}:\X'\to\X'$ the inverse of $e_\bX(\mathcal S(u),u)^*$. 
Moreover, \eqref{OptCond-5} is equivalent to
\begin{align}
	\label{OptCond-6}
	{\langle\mathcal G,\bX^\delta\rangle}_{\X',\X}&=-{\langle e_\bX(\mathcal S(u),u)^{-*}J_\bX(\mathcal S(u),u),\bX^\delta\rangle}_{\X',\X}\quad\text{for all }\bX^\delta\in\X.
\end{align}
Using \eqref{OptCond-2} and \eqref{OptCond-5} we find that
\begin{align}
	\label{OptCond-6A}
	\begin{aligned}
		\mathcal S'(u)^*J_\bX(\mathcal S(u),u)&=\big[-e_\bX(\mathcal S(u),u)^{-1}e_u(\mathcal S(u),u)\big]^*J_\bX(\mathcal S(u),u)\\
		&=e_u(\mathcal S(u),u)^*\big[-e_\bX(\mathcal S(u),u)^{-*}J_\bX(\mathcal S(u),u)\big]\\
		&=e_u(\mathcal S(u),u)^*\mathcal G\quad\text{in }\U',    
	\end{aligned}
\end{align}
where the operator $e_u(\mathcal S(u),u)^*$ maps from $\X'$ to $\U'$. Inserting the last expression in \eqref{OptCond-3} we derive
\begin{align*}
	\hat J'(u)=\mathcal S'(u)^*J_\bX(\mathcal S(u),u)+J_u(\mathcal S(u),u)=e_u(\mathcal S(u),u)^*\mathcal G+J_u(\mathcal S(u),u)\quad\text{in }\U'.
\end{align*}
For arbitrary $\bX^\delta\in\X$ let $\bY^\delta=\bY^\delta(\bX^\delta)\in\X$ be the unique solution to
\begin{align}
	\label{OptCond-7}
	e_\bX(\mathcal S(u),u)\bY^\delta=\bX^\delta\quad\text{in }\X.
\end{align}
In particular, we have
\begin{equation}\label{eq-y-versus-x}
	\bY^\delta(t)-\int_0^ta_x(\bX(s),u(s),s)\bY^\delta(s)\,\mathrm ds-\int_0^tb_x(\bX(s),u(s),s)\bY^\delta(s)\,\mathrm d\bB(s)=\bX^\delta(t)
\end{equation}
for all $t\in[0,T]$, where we have applied the notation introduced in \cref{Rem:TensorProd}. Then, we infer from \eqref{OptCond-6} that
\begin{align}
	\label{OptCond-8}
	\begin{aligned}
		{\langle\mathcal G,\bX^\delta\rangle}_{\X',\X}&=-{\langle e_\bX(\mathcal S(u),u)^{-*}J_\bX(\mathcal S(u),u),e_\bX(\mathcal S(u),u)\bY^\delta\rangle}_{\X',\X}\\
		&=-{\langle J_\bX(\mathcal S(u),u),\bY^\delta\rangle}_{\X',\X}\quad\text{for all }\bX^\delta\in\X        
	\end{aligned}    
\end{align}
and for the associated $\bY^\delta\in\X$ solving \eqref{OptCond-7}. Now we study the specific problem introduced in \cref{Ex:Cost}. 
Using \eqref{J_PartDerX} we have
\begin{align}
	\label{OptCond-9}
	{\langle\mathcal G,\bX^\delta\rangle}_{\X',\X}=\Eb\bigg[\int_0^T\bm g_1(t)^\top \bY^\delta(t)\,\mathrm dt\bigg]+\Eb\big[\bm g_2^\top \bY^\delta(T)\big]
\end{align}
with $\bm g_1$ and $\bm g_2$ from \eqref{Formulas_g}.

In the following, we want to represent the functional $\mathcal G\in\X'$ in the form $\mathcal G = \ell_{\bL}$ (see Lemma~\ref{Le:ExtXb}) with a suitably 
chosen $\bL\in\X$.  
Slightly changing the notation, we write $\bL = \Phi(\bL_\circ, -\bA, \bZ)$ with $(\bL_\circ, -\bA, \bZ)\in\Xb$ (see \eqref{eq-defXb}), and get
\begin{align}
	\label{eq-matrix-inner-product-a}
	\ell_{\bL}(\bX^\delta) = \Eb\bigg[\bX^\delta(T)^\top\bL(T)+\int_0^T\bX^\delta(t)^\top\bA(t)\,\mathrm dt\bigg] = \big\langle (\bX_\circ, \mathtt a,\mathtt b), (\bL_\circ,\bL, \bZ)\big\rangle_{\Xb}
\end{align}
for $\bX^\delta=\Phi(\bX_\circ, \mathtt a, \mathtt b)$, where the second equality follows from \cref{Le:ExtXb}. 
Note that this can be seen as a Riesz representation in the space $\Xb$. 
Let $\bY^\delta  = \Phi(\bY_\circ, \widetilde{\mathtt a},\widetilde{\mathtt b})$. 
From \eqref{eq-y-versus-x}, we know that we have
\begin{align}
	\label{eq-matrix-inner-product-b}
	\bX^\delta = \bY^\delta - \Phi(0, a_x\bY^\delta, b_x\bY^\delta),
\end{align}
where we abbreviated $a_x = a_x(\bX(\cdot), u(\cdot), \cdot)$ and $b_x = b_x(\bX(\cdot), u(\cdot), \cdot)$.  
Recall  that $b(\bX,u,\cdot)$ $= (b_1(\bX,u,\cdot)\vert \ldots \vert b_m(\bX,u,\cdot))$ and  $\bZ = (\bZ_1\vert\ldots\vert\bZ_m)$. 
Therefore, the Frobenius inner product, defined in \eqref{eq:definition_frobenius_scalarproduct},
of $b_x \bY^\delta$ and $\bZ$ is given by
\begin{align}\label{eq-matrix-inner-product}
	{\langle b_x\bY^\delta, \bZ\rangle}_\Fr=\sum_{j=1}^m (b_{jx}\bY^\delta)^\top \bZ_j=(\bY^\delta)^\top \Big( \sum_{j=1}^m b_{jx}^\top\bZ_j\Big).
\end{align}
With \eqref{eq-matrix-inner-product-a}-\eqref{eq-matrix-inner-product}, we obtain
\begin{align*}
	\ell_{\bL}(\bX^\delta) & =  \ell_{\bL}(\bY^\delta)- \big\langle (0, a_x\bY^\delta, b_x\bY^\delta), (\bL_\circ, \bL, \bZ)\big\rangle_{\Xb} \\
	&=\Eb\bigg[\int_0^T\bY^\delta(t)^\top\big(\bA(t)-a_x(\bX(t),u(t),t)^\top\bL(t)-\sum_{j=1}^m b_{jx}(\bX(t),u(t),t)^\top\bZ_j(t)\big)\,\mathrm dt\bigg]\\
	&\quad+\Eb\big[\bY^\delta(T)^\top\bL(T)\big].
\end{align*} 
Comparing this with \eqref{OptCond-9}, we see that  the representation ${\langle\mathcal G,\bX^\delta\rangle}_{\X',\X}= \ell_{\bL}(\bX^\delta)$ holds for all $\bX^\delta\in \X$ if
\begin{align*}
	& \Eb\bigg[\int_0^T\bY^\delta(t)^\top\Big(\bA(t)-a_x(\bX(t),u(t),t)^\top\bL(t)-\sum_{j=1}^m b_{jx}(\bX(t),u(t),t)^\top\bZ_j(t)-\bm g_1(t)\Big)\,\mathrm dt\bigg]\\
	& \qquad +\Eb\big[\bY^\delta(T)^\top\big(\bL(T)-\bm g_2\big)\big]=0
\end{align*}
which yields
\begin{align*}
	\bA(\cdot)&=a_x(\bX(\cdot),u(\cdot),\cdot)^\top\bL(\cdot)+\sum_{j=1}^m b_{jx}(\bX(\cdot),u(\cdot),\cdot)^\top\bZ_j(\cdot)+\bm g_1(\cdot)&&\text{in }\Lb^2_\mathscr F(\Rb^d),\\
	\bL(T)&=\bm g_2.
\end{align*}
As $\bL = \Phi(\bL_\circ, -\bA, \bZ)$, we see that $\bL$ satisfies the following (linear) BSDE (cf., e.g., \cite{MaYong1999FBSDEapplic}):
\begin{subequations}
	\label{OptCond-12}
	\begin{align}
		\label{OptCond-12a}
		\left\{
		\begin{aligned}
			\mathrm d\bL(t)&=-\Big(\bm g_1(t)+a_x(\bX(t),u(t),t)^\top\bL(t)+\sum_{j=1}^m b_{jx}(\bX(t),u(t),t)^\top\bZ_j(t)\Big)\,\mathrm dt\\
			&\quad+\bZ(t)\,\mathrm d\bB(t)\quad\text{for all }t\in[0,T),\\
			\bL(T)&=\bm g_2    .
		\end{aligned}
		\right.
	\end{align}
	Moreover, if the solution pair $(\bL,\bZ)$ to \eqref{OptCond-12a} is computed, then
	\begin{align}
		\label{OptCond-12b}
		\bL_\circ=\bL(0)
	\end{align}
\end{subequations}
holds. Now we turn to the representation of the gradient $\nabla\hat J(u)\in\U$ of the reduced cost functional $\hat J$ at a given admissible control $u\in\Uad$. 
Recall that
\begin{align*}
	{\langle\nabla\hat J(u),u^\delta\rangle}_\U={\langle\hat J'(u),u^\delta\rangle}_{\U',\U}\quad\text{for all }u^\delta\in\U
\end{align*}
holds, i.e., the gradient $\nabla\hat J(u)\in\U$ is the Riesz representant of $\hat J'(u)\in\U'$. 
Utilizing \eqref{OptCond-6A} and $\mathcal G = \ell_\Lambda$, we find 
\begin{align*}
	&{\langle \mathcal S'(u)^*J_\bX(\mathcal S(u),u),u^\delta\rangle}_{\U',\U}
	={\langle e_u(\mathcal S(u),u)^*\mathcal G,u^\delta\rangle}_{\U',\U}
	={\langle \mathcal G,e_u(\mathcal S(u),u)u^\delta\rangle}_{\X',\X}\\
	& \quad = \ell_\bL(e_u(\mathcal S(u),u)u^\delta).
\end{align*}
By Lemma~\ref{Le:eFrDer}, we know that $e_u(\mathcal S(u),u)u^\delta= \Phi(0,\tilde{\mathtt a}, \tilde{\mathtt b})\in \X$ with 
\begin{align*}
	\tilde{\mathtt a}(t) \colonequals -a_u(\bX(t), u(t),t) u^\delta(t), \quad 
	\tilde{\mathtt b}(t) \colonequals -b_u(\bX(t), u(t), t) u^\delta(t),\quad t\in (0,T).
\end{align*}
From this and Lemma~\ref{Le:ExtXb}-2), we obtain
\begin{align}
	&{\langle \mathcal S'(u)^*J_\bX(\mathcal S(u),u),u^\delta\rangle}_{\U',\U} = \ell_\bL(\Phi(0,\tilde{\mathtt a},\tilde{\mathtt b})) =
	\big\langle(0,\tilde{\mathtt a},\tilde{\mathtt b}), (\bL_\circ,\bL,\bZ)\big\rangle_\Xb\nonumber\\
	& \quad = \langle \tilde{\mathtt a}, \bL\rangle_{\Lb^2_\mathscr F(\Rb^d)} + \langle 
	\tilde{\mathtt b}, \bZ\rangle_{\Lb^2_\mathscr F(\Rb^{d\times m})}.\label{eq5-1}
\end{align}
Noting that the control $u$ is deterministic, we  rewrite the first term on the right-hand side  as
\begin{align*}
	{\langle \tilde{\mathtt a}, \bL\rangle}_{\Lb^2_\mathscr F(\Rb^d)} & = {\langle -a_u(\bX(\cdot), u(\cdot),\cdot) u^\delta, \bL\rangle}_{\Lb^2_\mathscr F(\Rb^d)} 
	= {\langle u^\delta, -a_u(\bX(\cdot), u(\cdot),\cdot)^\top \bL\rangle}_{\Lb^2_\mathscr F(\Rb^r)} \\
	& = \Eb \bigg[ 
	\int_0^T  u^\delta(t)^\top \big(-a_u(\bX(t),u(t),t) \big)^\top \bL(t)\,\mathrm dt  
	\bigg]\\
	&= 
	\int_0^T  u^\delta(t)^\top \Eb \Big[ \big(-a_u(\bX(t),u(t),t) \big)^\top \bL(t)\Big]\,\mathrm dt  \\
	& ={\left\langle \Eb \Big[ -a_u(\bX(\cdot),u(\cdot),\cdot)^\top \bL
		\Big],u^\delta\right\rangle}_\U.
\end{align*}
In the same way, we obtain for the second term on the right-hand side of \eqref{eq5-1}
\begin{align*}
	{\langle \tilde{\mathtt b}, \bZ\rangle}_{\Lb^2_\mathscr F(\Rb^{d\times m})}&={\langle 
		-b_u(\bX(\cdot), u(\cdot), \cdot) u^\delta, \bZ\rangle}_{\Lb^2_\mathscr F(\Rb^{d\times m})} 
	\\
	&=\bigg\langle \Eb \Big[ - \sum_{j=1}^m b_{ju}(\bX(\cdot),u(\cdot),\cdot)^\top \bZ_j
	\Big],u^\delta\bigg\rangle_{\U};
\end{align*}
cf. \eqref{eq-matrix-inner-product} for the description of the Frobenius inner product. Inserting this into \eqref{eq5-1} and using \eqref{OptCond-3} and \eqref{J_PartDerU}, we see that
\begin{align}
	{\langle\hat J'(u),u^\delta\rangle}_{\U',\U} & = \langle \mathcal S'(u)^*J_\bX(\mathcal S(u),u)+J_u(\mathcal S(u),u) , u^\delta\rangle_{\U',\U} \notag \\
	& = \bigg\langle \Eb \Big[ -a_u(\bX,u,\cdot)^\top \bL-\sum_{j=1}^m b_{ju}(\bX,u,\cdot)^\top \bZ_j
	\Big]+\kappa u,u^\delta\bigg\rangle_{\U} .
	\label{eq:gradient_u}
\end{align}
Summarizing, we have proved the following theorem.

\begin{theorem}
	\label{Th:GradForm}
	Suppose that \cref{Assumption:Diffbarkeit,A:5.3} hold. Moreover, the cost functional is given by \eqref{eq:objective_calib} and \eqref{OptCond-12} is uniquely solvable. Then, for any $u\in\Uad$ the gradient of $\hat J$ is given as
	\begin{align}
		\label{GradRedGradCont}
		\nabla \hat J(u)=\kappa u-\Eb \Big[a_u(\bX(\cdot),u(\cdot),\cdot)^\top \bL+\sum_{j=1}^m b_{ju}(\bX(\cdot),u(\cdot),\cdot)^\top \bZ_j\Big]\in\U,
	\end{align}
	where the pair $(\bL,\bZ)\in\X\times\Lb^2_\mathscr F(\Rb^{d\times m})$ solves \eqref{OptCond-12}.
\end{theorem}

\begin{remark}
	\label{rem:ComparisonGradients}
	Let us compare $\nabla \hat J$ for the case $\U = \Uc$ with the gradient $\nabla \hat J^h$ (given in \eqref{GradRedGradDisc} for the discretized problem. For this case we can write \eqref{GradRedGradCont} as
	\begin{align}
		\label{GradRedGradContCalibration}
		\nabla \hat J(u)=\kappa u-\Eb \Big[ \int_0^T a_u(\bX(t),u,t)^\top \bL+\sum_{j=1}^m b_{ju}(\bX(t),u,t)^\top \bZ_j \, \mathrm{d} t \Big].
	\end{align}
	Recall that the discrete gradient \eqref{GradRedGradDisc} is
	\begin{align*}
		\nabla\hat J^h(u)
		=\kappa u&-\Eb^M\bigg[ \sum_{\nu=1}^{N-1}a_u(\bX^\mu_\nu,u,t_\nu)^\top\bL^\mu_{\nu+1}\Delta t\bigg]
		\\
		&+\Eb^M\bigg[\sum_{\nu=1}^{N-1}\sum_{j=1}^m b_{ju}(\bX^\mu_\nu,u,t_\nu)^\top \left( \Delta B_{\nu,j}^\mu \bL^\mu_{\nu+1}\right)\bigg].
	\end{align*}
	The first difference between \eqref{GradRedGradDisc} and \eqref{GradRedGradContCalibration} is that the integral over time in the continuous gradient turns to a sum over the time intervals discrete gradient. 
	This is not surprising.
	However, we point out that in the discrete case, the coefficient function $a_u$ and the adjoint variable $\bL$ are not evaluated at the same timestep, but at $t_{\nu}$ and $t_{\nu+1}$, respectively.
	The crucial difference is in the last part. 
	In the continuous case, the stochastic variable $\bZ$ appears that is absent in the discrete gradient. 
	Instead, in the latter case, the (pre-computed) Brownian increments $\Delta B$ multiplied with $\bL$  appear. 
	This is consistent with the numerical scheme to solve backward SDEs presented in \cite[Section 3.3]{Gong2017efficientSGD}.
	Furthermore, recall that the time increment is incorporated in the Brownian increment. 
	Hence, the discretization factor $\Delta t$ only appears in the deterministic part.
    
    Our discussion raises the question whether optimization and discretization commute in the sense that the optimize-before-discretize and discretize-before-optimize approaches lead to similar results for sufficiently fine mesh sizes. 
    We are currently investigating this question.
\end{remark}

The following result follows directly from \cref{Th:GradForm} and \cite[Theorem 1.46]{HPUU08}.

\begin{corollary}
	Suppose that all assumptions of \cref{Th:GradForm} hold. We assume that $\bar u\in\Uad$ is a local solution to \eqref{eq:SDE_Opti_OU_red}. Then, first-order necessary optimality conditions are given by the variational inequality
	\begin{align*}
		\bigg\langle\kappa \bar u+\Eb \Big[ -a_u(\bXb,\bar u,\cdot)^\top \bm{\bar\Lambda}-\sum_{j=1}^m b_{ju}(\bXb,\bar u,\cdot)^\top \bm{\bar Z}_j
		\Big],u-\bar u\bigg\rangle_{\U}\ge0\quad\text{for all }u\in\Uad,
	\end{align*}
	where $\bXb=\Sm(\bar u)$ holds and $(\bL,\bZ)$ solves \eqref{OptCond-12} for $(\bX,u)=(\bXb,\bar u)$.
\end{corollary}


\section{Numerical experiments}
\label{sec:Numerical_experiments}

In this section, we explain our optimization strategy. Afterwards, we perform numerical experiments. In the first one, we find an optimal control such that the mean and variance of the governing process follow precisely the prescribed data. The governing process in this case is the well-known mean-reverting Ornstein-Uhlenbeck process. 
In the second one, we calibrate a model consisting of systems of SDEs by finding optimal parameters. 
The specific model is the so-called Stochastic Prandtl-Tomlinson (SPT) model that is used to study microrheological processes of viscous fluids.

Let us mention that in both examples our cost functionals are not of the specific form introduced in \cref{Ex:Cost}.

\paragraph*{Optimization strategy}

Exploiting the discrete optimality system and using the Euler-Maruyama (EM) scheme \cite{Higham2001IntroNumSDE,Maruyama1955SDE}, we construct a stochastic gradient method that is summarized in \cref{alg:PSGD_SOCP}. During the process of solving the model and adjoint SDE, we utilize the parallelization technique of multi-threading in order to speed up the calculations. 
This is possible since each realization is independent of the other ones. While updating the control, we apply a projection onto $\Uad$ or $\UadCalib$ utilizing the standard projection operator $\mathcal{P}_{\Uad}$ that is based on a componentwise projection.

Hence, we apply a stochastic gradient method using fixed batch sizes $M^\ell$.  Convergence properties can be found in \cite{BottouCurtis2018ReviewStochOpt}. 
In our examples, we use with some initial $s_0 > 0$ the stepsize rule
\begin{align}
	\label{eq:stepsize_rule}
	s_\ell=s_0/\ell
	\quad\text{with }\ell\in\Nb
\end{align}
for the stochastic gradient method.

\begin{algorithm}
	\caption{Stochastic Gradient Method for Stochastic Optimal Control}\label{alg:PSGD_SOCP}
	\begin{algorithmic}[1]
		\REQUIRE Initial discrete control guess $u^0 = (u^0_1,\ldots,u^0_r)^\top$ in case of $\U=[\ua,\ub]$ (calibration problem) or $\bU^{h,0}=[\mathrm u_0^0|\ldots|\mathrm u_{N-1}^0]$ in case of $\U=L^2(0,T;\Rb^r)$ (time-dependent controls), tolerance $tol>0$, 
		maximum iteration depth $\ell_{\max} \in \Nb$;
		\STATE Set $\ell = 0$, initialize $E \gg tol$;
		\WHILE{$E>tol$ \AND $\ell<\ell_{\max}$}
		\STATE Generate $M^{\ell} \in \Nb$ random numbers (call the set of the random numbers $\Ms^{\ell}\in \Nb^{M^{\ell}}$);
		\STATE Solve the discretized state model using the EM scheme;
		\STATE Solve the associated adjoint model using the EM scheme;
		\STATE Assemble the gradient $\nabla\hat J^h(u^\ell)$ and $\nabla\hat J^h(\bU^{h,\ell})$, respectively, of the reduced cost and compute $E=\|\nabla\hat J^h(u^{\ell})\|_\Uc$ and $E=\|\nabla\hat J^h(\bU^{h,\ell})\|_{\U^h}$, respectively;
		\STATE Calculate stepsize $s_\ell$ using \eqref{eq:stepsize_rule};
		\STATE Determine a new control $u^{\ell+1}$ and $\bU^{h,\ell+1}$, respectively, by applying a projected stochastic gradient step;
		\STATE Set $\ell = \ell+1$;
		\ENDWHILE
		\RETURN $u^\ell$ and $\bU^{h,\ell}$, respectively.
	\end{algorithmic}
\end{algorithm}

\begin{remark}
	The random numbers $\Ms^{\ell}$ in \cref{alg:PSGD_SOCP} change in every optimization iteration, but stay the same within an iteration.
\end{remark}

\subsection{Time-dependent Ornstein-Uhlenbeck process}
\label{sec:OU}

In the context of \eqref{SDE} we have $d=1$, $m=1$, $r=2$ and therefore $\U\colonequals L^2(0,T;\Rb^2)$. 
For a fixed parameter $\theta>0$, the stochastic process $X(t)\in L^2(\Omega;\Rb)$ solves the linear SDE
\begin{align}
	\mathrm dX(t)=\theta\big(u_1(t) - X(t)\big)\,\mathrm dt+u_2(t)\,\mathrm dB(t)\text{ for }t\in(0,T],\quad X(0)=X_\circ
	\label{eq:OU_SDE}
\end{align}
with an initial condition $X_\circ\in L^2(\Omega)$ and the time-dependent coefficients $u=(u_1,u_2)\in\U$. 
We set
\begin{align*}
	a(x,u,t)=\theta\big(u_1 - x\big),\quad
	b(x,u,t)=u_2\quad\text{for }(x,u,t)\in\Rb\times\Rb^r\times[0,T]\text{ and }u=(u_1,u_2).
\end{align*}
Then, \eqref{eq:OU_SDE} can be expressed in the form \eqref{SDE}. 
Recall that for $X\in L^2(\Omega)$ the variance is defined as
\begin{align*}
	\Vb[X]\colonequals\Eb\big[X-\Eb[X]^2\big]
	=\int_\Omega\big(X-\Eb[X]\big)^2\,\mathrm d\Pb
	=\int_\Omega\big(X(\omega)-\Eb[X]\big)^2\,\mathrm d\Pb(\omega).
\end{align*}
Now we introduce the cost functional as
\begin{align}
	\label{eq:OU_problem_functional}
	\begin{aligned}
		J(X,u)&\colonequals \frac{1}{2}\,\big\|\Eb[X(\cdot)]-\ed(\cdot)\big\|_{L^2(0,T)}^2+\frac{1}{2} \,\big\|\mathbb{V}[X(\cdot)]-\sd(\cdot)\big\|_{L^2(0,T)}^2
		\\
		&\quad~+\frac{1}{2}\,\big|\Eb[X(T)]-\edT\big|^2+ \frac{1}{2}\,\big|\mathbb{V}[X(T)]-\sdT\big|^2+\frac{\kappa}{2}\,{\|u\|}_\U^2
	\end{aligned}    
\end{align}
with
\begin{align*}
	\ed(t)=\sin\Big(\frac{2\pi t}{T}\Big)-1,\quad\sd(t)=0.2\bigg(\cos\Big(\frac{2\pi t}{T}\Big) + 2\bigg),\quad\edT=\ed(T),\quad\sdT=\sd(T)
\end{align*}
for $t\in[0,T]$. Setting for $t\in[0,T]$ and $X\in\X$
\begin{align*}
	j(X)&=\frac{1}{2}\,\big\|\Eb[X(\cdot)]-\ed\big\|_{L^2(0,T)}^2+\frac{1}{2}\,\big\|\Vb[X(\cdot)]-\sd\big\|^2_{L^2(0,T)},\\
	j_T(X)&=\frac{1}{2}\,\big\|\Eb[X(T)]-\edT\big\|_{L^2(0,T)}^2+\frac{1}{2}\,\big\|\Vb[X(T)]-\sdT\big\|^2_{L^2(0,T)}
\end{align*}
the cost \eqref{eq:OU_problem_functional} and our optimization problem can be expressed as \eqref{GeneralCost} and  \eqref{eq:SDE_Opti_OU}, respectively.

\begin{remark}
	\label{Rem:OU_Ass}
	\begin{enumerate}[label={\em \arabic*)}]
		\item Note that the coefficient functions $a$ and $b$ are independent of $t$ and affine linear in $x$ and $u$. 
		In particular, both coefficients are Lipschitz continuous in $x$ and $u$. Moreover, $a$ and $b$ are time-independent. 
		Furthermore, \eqref{Assumption:coefficients_integral} holds for $u\in\U$. Hence, \cref{Assumption:coefficients,A:LinSDE} are fulfilled.
		\item For the cost functional it holds that it is continuous as a composition of continuous functions and moreover bounded from below by zero. Furthermore, the parts of the functional only containing the expected value are convex as a composition of non-decreasing functions and a convex function. 
		However, since the variance is not a convex function, the parts with the variance are not convex. 
		Consequently, \cref{Assumption:continuity_C} is only partially satisfied.
		\item Note
		that $a$ and $b$ are continuously differentiable and hence \cref{Assumption:Diffbarkeit}-\ref{AssumptionState} is fulfilled. Furthermore, $a_x(x,u(t),t)=-\theta$ and $b_x(x,u(t),t)=0$ hold for $(x,u)\in\Rb\times\U$ and $t\in[0,T]$. Thus, \eqref{Eq:A:5.3-3} is clearly valid by choosing $L_\mathfrak{ab}=\theta$ and $L_\mathfrak{ab}=0$, respectively. 
		Thus, \cref{A:5.3} is fulfilled.
	\end{enumerate}
\end{remark}

Utilizing the notation introduced in \cref{sec:D6.2}, it turns out that for $\mu=1,\ldots,M$ and $\nu=0,\ldots,N-1$, the first-order necessary optimality system of the discrete optimization problem is given with $\mu = 1,\ldots,M$ by
\begin{subequations}
	\label{eq:SDE_opti_OptSys_OU}
	\begin{align}
		&\left\{
		\begin{aligned}
			X^\mu_{\nu+1}&=X^\mu_\nu+\big( \theta(\bU^h_{1\nu}-X^\mu_\nu)\big)\,\Delta t + \bU^h_{2\nu}\,\Delta B^\mu_\nu,
            \qquad\qquad \nu = 0,\ldots,N-1
            \\
			X^\mu_0&=X_\circ^\mu, 
		\end{aligned}
		\right.
        \label{eq:opt_sys_forward}
        \\
		&\left\{
		\begin{aligned}
			\Lambda^\mu_\nu&=\Lambda_{\nu+1}^\mu-\theta\Lambda_{\nu+1}^\mu\,\Delta t+ \Delta t\left( \Eb^M\big[X_\nu^M\big] - \eta^\mathrm{d}_\nu\right)
			\\
			&\quad+2\Delta t \left(\Vb^M\big[X_\nu^M\big]-\sigma_\nu^\mathsf d\right)\left( X^\mu_\nu-\Eb^M\big[X_\nu^M\big]\right),
             \qquad\qquad \nu = N-1,\ldots,0, 
			\\
			\Lambda^\mu_N&=\left( \Eb^M\big[X_N^h\big] - \edT\right)+
           2\left(\Vb^M\big[X_N^h\big]-\sigma_T^\mathsf d\right)\left( X^\mu_N-\Eb^M\big[X_N^h\big]\right),
		\end{aligned}
		\right.
        \label{eq:opt_sys_backward}
        \\
		\label{eq:SDE_opti_OptSys_OU_Gradient}
		&\nabla \hat{J}^h_\nu(\mathrm u)
		\colonequals\kappa \,  \mathrm{u}_\nu - \frac{1}{M}\sum_{\mu=1}^{M}\left( \binom{\theta}{0}\Delta t+\binom{0}{1}\Delta B^\mu_\nu \right)\Lambda_{\nu+1}^\mu\quad\text{for }\nu=0,\ldots,N-1   
	\end{align}
\end{subequations}
with
\begin{align*}
	X^h&=\big\{X_\nu^\mu\in\Rb\,\big|\,\mu=1,\ldots,M,~\nu=0,\ldots,N\big\},\\
	\bU^h&=\big[\mathrm u_0|\ldots|\mathrm u_{N-1}\big]\in\Rb^{2\times N},\quad\mathrm u_\nu=\big(\bU^h_{1\nu},\bU^h_{2\nu}\big)^\top\in\Rb^2~(\nu=0,\ldots,N-1),\\
	\Lambda^h&=\big\{\Lambda_\nu^\mu\in\Rb\,\big|\,\mu=1,\ldots,M,~\nu=0,\ldots,N\big\}.
\end{align*}

In the sequel, we set the regularization parameter to zero, i.e. $\kappa=0$. Notice that \eqref{eq:SDE_opti_OptSys_OU_Gradient} can in this case be written as
\begin{align}
	\nabla \hat{J}^h_\nu(\mathrm u)= \Eb^M\left[ \binom{\theta \Delta t \, \Lambda_{\nu+1}^M}{\Delta B_\nu \, \Lambda^M_{\nu+1}} \right] = \binom{\theta\Delta t \, \Eb^M[\Lambda_{\nu+1}^M]}
	{\Eb^M[\Delta B_\nu \Lambda^M_{\nu+1}]}.
	\label{eq:gradient_EE}
\end{align}
Recall that $\Eb[\Delta B_{\nu+1}]=0$. 
However, since we used the $\Delta B^\mu_\nu$ to derive $X_{\nu+1}$ and with this $\Lambda_{\nu+1}$, it holds that $\Lambda_{\nu+1}$ is not independent of $\Delta B_\nu$. 
Therefore, the second component of the gradient in \eqref{eq:gradient_EE} is non-zero.

\paragraph{Numerical results}

In the next figures, we present the results of our optimization strategy for the Ornstein-Uhlenbeck process \eqref{eq:OU_SDE}. 
For this example, we choose the initial condition $X_\circ$ to obey the normal distribution with mean $\eta^\mathsf d(0)$ and standard deviation $\sigma^\mathsf d(0)$.

Since \eqref{eq:OU_SDE} is a linear SDE, the trajectory of mean and variance can be calculated as (see, e.g., \cite[Chapter 3, Theorem  3.2 and Example 5.2]{Mao2008SDE})
\begin{align}
	\begin{split}
		\Eb\big[X(t)\big]
		&=e^{-\theta t}\Eb[X_\circ] + e^{-\theta t}\theta \int_0^s u_1(s)e^{\theta s}\, \mathrm ds,
		\\
		\mathbb{V}\big[X(t)\big] 
		&= e^{-2\theta t}\mathbb{V}\big[X_\circ\big]+ 
		e^{-2\theta t}\int_0^t u_2^2(s)e^{2\theta s}\,\mathrm ds.
	\end{split}
	\label{eq:mean_var_OU}
\end{align}
Hence, we can derive the best controls $u$ in the sense that they will lead to a perfect tracking of the desired trajectories of mean and variance: 
\begin{align}
	\label{eq:perfect_controls_OU}
	u_1^*(t)=(\eta^\mathsf d)'(t)\frac{1}{\theta}+\eta^\mathsf d(t)
	\quad\text{and}\quad (u_2^*)^2(t)=2\theta(\sigma^\mathsf d)(t) + \sqrt{\sigma^\mathsf d(t)}(\sqrt{\sigma^\mathsf d})'(t).
\end{align}

The success of the method given in \cref{alg:PSGD_SOCP} is evident in \cref{fig:Results_OU}. 
The method manages to find controls $u_1$ and $u_2$ such that the trajectory of the mean value and the variance of the solution of the model equation follows the desired one.
In \cref{fig:Moments_trajectories_result}, the trajectories of all $M$ trials are plotted together with the calculated and desired moments.
In \cref{fig:Controls}, the control that is calculated with our optimization method is plotted together with the theoretical control $(u_1^*(t),u_2^*(t))$ that leads to a perfect tracking.
One obtains a very good agreement for $u_1(t)$. 
For $u_2(t)$ there is higher noise.
However, for more iterations of the stochastic gradient method, the noise gets smaller.
Finally, in \cref{fig:OU-convergence_history}, we plot the convergence history of the (relative) functional and the relative norm of the gradient over the optimization iterations $\ell$.
and the analogous holds for the relative gradient.
\begin{figure}
	\begin{subfigure}[l]{0.32\textwidth}
		\includegraphics[width=\textwidth]{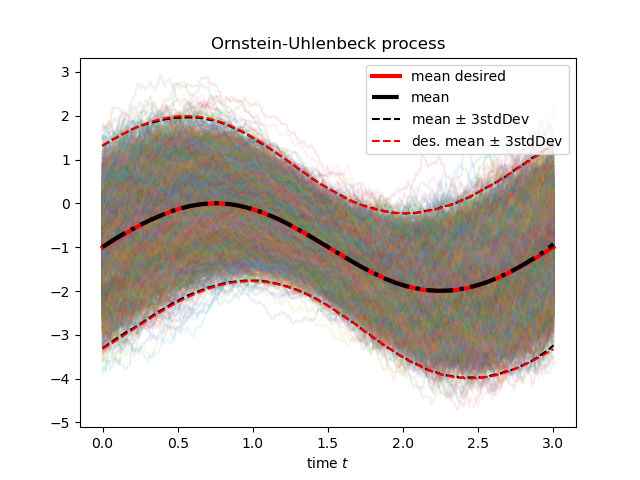}
		\caption{}
		\label{fig:Moments_trajectories_result}
	\end{subfigure}
	\begin{subfigure}[l]{0.32\textwidth}
		\includegraphics[width=\textwidth]{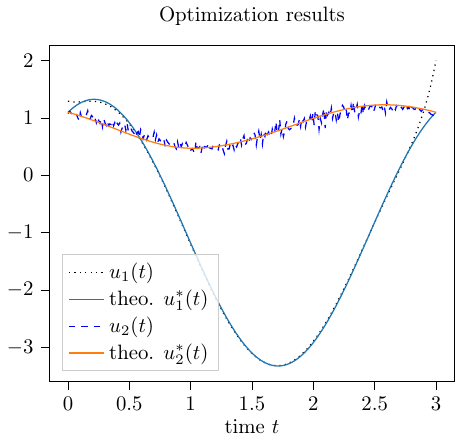}
		\caption{}
		\label{fig:Controls}
	\end{subfigure}
	\begin{subfigure}[l]{0.35\textwidth}
		\includegraphics[width=\textwidth]{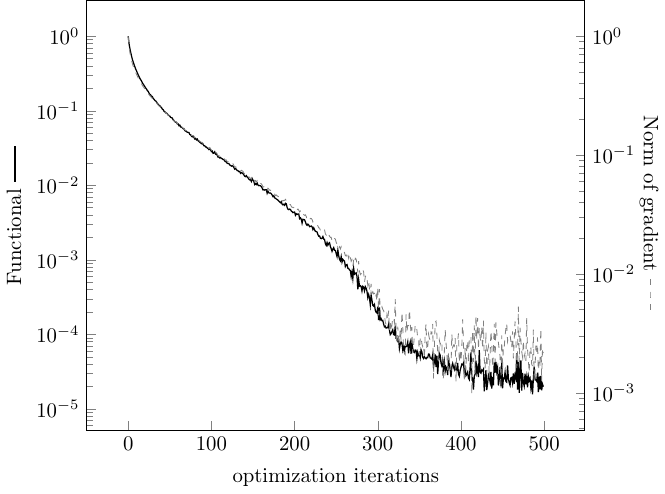}
		\caption{}
		\label{fig:OU-convergence_history}
	\end{subfigure}
	\caption{
		Results for the example of \cref{sec:OU}.
		(a) Trajectories of all trials (grayed out), desired moments (red) calculated moments (black) while applying the control calculated by Algorithm \ref{alg:PSGD_SOCP};
		(b) Calculated controls $u$ and $u$ together with controls $u_1^*$ and $u_2^*$ defined in \eqref{eq:perfect_controls_OU} for a perfect tracking;
		(c) Convergence history of relative functional (solid) and Euclidean norm of the relative gradient (dashed) over the optimization iterations (cf. \eqref{eq:relative_functional}).}
	\label{fig:Results_OU}
\end{figure} 
More specifically, the relative functional is given by
\begin{align}
	\hat{J}^h(\bU^{h,\ell})/|\hat{J}^h(\bU^{h,0})|,
	\label{eq:relative_functional}
\end{align}

\subsection{Stochastic Prandtl-Tomlinson (SPT) model}
\label{sec:SPT_example}

The SPT model is a non-equilibrium 
bath model in which a tracer particle, 
also called colloidal, is assumed to be immersed in a suspension of micelles; see, e.g., \cite{Jain2021twoStepRheolog,Jain2021microStochPrandtlTomlison,Muller2020NonLinBath} and the references therein and see \cite{PopovGray2014Prandtl,Prandtl1928gedankenmodell} for the origins of this models in the deterministic case.
In this setting, the internal forces between the particles lead to non-Markovian behavior of the movement of the colloidal due to the memory of the system. The goal is to study the properties of the suspension by tracing the colloidal.

To model the complex fluid, the colloidal particle is investigated with a coupling to one or more bath particles that describe the background. The system of differential equations describing the position of the tracer and $K$ bath particles is given by
\begin{align}
	\begin{aligned}
		\mathrm d X_1(t)
		&=\frac{1}{\gamma_1}\bigg(-\partial_x V_{\mathrm{ext}}(X_1(t),t)-\sum_{k=2}^{K+1} V_{\mathrm{int},k}'\big(X_1(t)-X_k(t)\big)\bigg)\,\mathrm dt+\frac{1}{\gamma_1}\,\mathrm dB_1(t),\\
		\mathrm d X_k(t)&=\frac{1}{\gamma_k}\,V_{\mathrm{int},k}'\big(X_1(t)-X_k(t)\big)+\frac{1}{\gamma_k}\,\mathrm dB_k(t)
		\quad
		\text{for } k=2,\ldots,K+1,
	\end{aligned}		
	\label{eq:Stochastic_Prantl_Tomlinson_intro}
\end{align}
for $t\in(0,T]$, where for $(x,t)\in\mathbb R\times[0,T]$
\begin{align*}
	V_{\mathrm{ext}}(x,t)=\frac{\kappa_{\mathrm{ext}}}{2}\,\big(x-v_0 t\big)^2
	\qquad\text{and}\qquad
	V_{\mathrm{int},k}(x)=V_{0,k}\cos\bigg( \frac{2\pi x}{d_k}\bigg)\text{ for } k=2,\ldots,K+1.
\end{align*}
The system \eqref{eq:Stochastic_Prantl_Tomlinson_intro} is completed with the initial condition $X_i(0) = X^{eq}_{i,\circ}$, $i=1,\ldots,K+1$, where $X^{eq}_{i,\circ} \in L^2(\Omega;\Rb)$ is the equilibrium distribution of the corresponding particle given the external potential $V_{\mathrm{ext}}(X,0)$. In system \eqref{eq:Stochastic_Prantl_Tomlinson_intro}, the variable $X_1$ denotes the position of the tracer particle and the $X_k$'s the position of the bath particle $p_k$, $k=2,\ldots,K+1$. The $\gamma_i>0$, $i=1,\ldots,K+1$, are the friction coefficients of the corresponding particle and are assumed to be known.
The external potential $V_{\mathrm{ext}}$ may be applied using optical traps (laser beams) of strength $\kappa_{\mathrm{ext}}>0$ to realize harmonic potentials. 
This potential has a trapping effect, which means that the particle cannot leave the potential easily. 
This trap is then pulled with a constant velocity $v_0$ through the surrounding medium. 
The internal forces between the tracer and the bath particles are denoted by $V_{\mathrm{int},k}$ for $k=2,\ldots,K+1$. 
The force
\begin{align*}
	\text{``$\xi_k(t)=\dot B_k(t)$''} \quad\text{for }t\in[0,T]\text{ and }k=1,\ldots,K+1
\end{align*}
resembles a random force for which the following mean value relation and fluctuation-dissipation theorem must hold
\begin{align}
	\Eb\big[\xi_i(t)\big] = 0,\quad\Eb\big[\xi_i(t)\xi_j(t')\big] = 2k_B\mathrm{T}\gamma_i\delta_{ij}\delta(t-t')
	\quad
	\text{for }i,j=1,\ldots,K+1,
	\label{eq:random_force}
\end{align}
where $k_B>0$ is Boltzmann's constant and $\mathrm{T}>0$ the temperature of the bath.
The requirements \eqref{eq:random_force} are fulfilled by Brownian motion, hence we can apply our setting above.

\begin{remark}
	\label{rem:initial_condition_equilibration}
	Notice that in \cite{Jain2021twoStepRheolog,Jain2021microStochPrandtlTomlison}, the authors did not specify any initial condition within the SPT model \eqref{eq:Stochastic_Prantl_Tomlinson_intro}. 
	The reason for this is that the system needs to be in equilibrium.
	To realize this numerically, we need to integrate over the transient phase in order to generate sensible data that will be used to calculate the simulation data.
	To generate $\bX^{eq}_{\circ}$, any initial condition can be used and the model needs to evolve for a certain equilibration time $t_{eq}>0$. 
	Then, we use the state of the system after $t_{eq}$ as $\bX^{eq}_{\circ}$.
	In order to account for a possible high equilibration time, one needs to resort to parallelization techniques such as multi-threading that we use in our simulations.
\end{remark}

The goal is to identify the pairs $\{(V_{0,k},d_k)\}_{k=2}^{K+1}\subset \Rb^2$ for each bath particle $p_k$, $k=2,\ldots,K+1$, within \eqref{eq:Stochastic_Prantl_Tomlinson_intro}, such that experimental data is matched as good as possible. For $r=2K$ we define the parameter vector $u=[u_1,\ldots,u_{2K}]^\top\in\Uc\colonequals\Rb^r$ with
\begin{align*}
	u_{2k-3}=V_{0,k}\quad\text{and}\quad u_{2k-2}=\frac{1}{d_k}\quad\text{for }k=2,\ldots,K+1.
\end{align*}
It follows that
\begin{align*}
	V_{\mathrm{int},k}(x)&=u_{2k-3}\cos\big(2\pi u_{2k-2}x\big)&&\text{for } k=2,\ldots,K+1,\\
	V_{\mathrm{int},k}'(x)&=-2\pi u_{2k-3}u_{2k-2}\sin\big(2\pi u_{2k-2}x\big)&&\text{for } k=2,\ldots,K+1.
\end{align*}
We suppose that $\ua\in\Uc$ satisfies $\ua>0$ component-wise in $\Rb^{2K}$. Further, the state variable consists of the tracer particle $X_1$ and the bath particles $X_2,\ldots,X_{K+1}$. 
Hence, for $d=K+1$ and $t\in[0,T]$ we define the following state vector
\begin{align*}
	\bX(t)\colonequals\big(X_1(t),\ldots,X_d(t)\big)^\top\in\Rb^d.
\end{align*}
The SDE for $\bX$ is given by \eqref{SDE}, where the coefficient functions $a$ and $b$ are given for every $(x,u,t)\in\Rb^d\times\Rb^r\times[0,T]$ as
\begin{subequations}
	\label{eq:coefficients_SPT_ab}
	\begin{align}
		\label{eq:coefficients_SPT_a}
		\begin{aligned}
			a_1(x,u,t)&=\frac{1}{\gamma_1}\bigg(-\partial_x V_{\mathrm{ext}}(x_1,t)-\sum_{k=2}^d V_{\mathrm{int},k}'(x_1-x_k)\bigg)
			\\
			&=\frac{1}{\gamma_1}\bigg(-\kappa_\mathrm{ext}(x_1-v_0 t)+2\pi\sum_{k=2}^du_{2k-2}u_{2k-3}\sin \big(2\pi u_{2k-2}(x_1-x_k)\big)\bigg)\\
			a_i(x,u,t)&=\frac{1}{\gamma_i}\,V_{\mathrm{int},i}'(x_1-x_i)=-\frac{2\pi}{\gamma_i}\,u_{2i-2}u_{2i-3}\sin\big(2\pi u_{2i-2}(x_1-x_i)\big)
		\end{aligned}
	\end{align}
	for $i=2,\ldots,d$ and
	\begin{align}
		\label{eq:coefficients_SPT_b}
		b_i(x,u,t)&=1/\gamma_i,
		\qquad\qquad
		i=1,\ldots,d.
	\end{align}
\end{subequations}
Notice that the $b_i$'s are constant and therefore independent of $(x,u,t)$. 
Moreover, we have that
\begin{align}
	a_x(x,u,t)=\left(\left(\partial_{x_j}a_i(x,u,t)\right)\right)\in\Rb^{d\times d}
	\label{eq:jacobian_a_x}
\end{align}
with
\begin{align*}
	\begin{split}
		\partial_{x_1}a_1(x,u,t)&=-\frac{\kappa_\mathrm{ext}}{\gamma_1}+\frac{4\pi^2}{\gamma_1}\sum_{k=2}^{K+1}u_{2k-2}^2u_{2k-3}\cos \big(2\pi u_{2k-2}(x_1-x_k)\big),
		\\
		\partial_{x_j}a_1(x,u,t)&=-\frac{4\pi^2}{\gamma_1}\, u_{2k-2}^2u_{2k-3}\cos \big(2\pi u_{2k-2}(x_1-x_j)\big),\quad j=2,\ldots,K+1,
		\\
		\partial_{x_1}a_i(x,u,t)&=-\frac{4\pi^2 }{\gamma_i}\,u_{2i-2}^2u_{2i-3}\cos\big(2\pi u_{2i-2}(x_1-x_i)\big),\quad i=2,\ldots,K+1,
		\\
		\partial_{x_i}a_i(x,u,t)&=\frac{4\pi^2}{\gamma_i}\,u_{2i-2}^2u_{2i-3}\cos\big(2\pi u_{2i-2}(x_1-x_i)\big),\quad i=2,\ldots,K+1.
	\end{split}
	\label{eq:SPT_derivative_x_a}
\end{align*}
We also find for
\begin{align*}
	a_u(x,u,t)=\left(\left(\partial_{u_j}a_i(x,u,t)\right)\right)\in\Rb^{d\times r},
\end{align*}
with the following components.
For the right-hand side $a_1$ corresponding to the tracer particle, we calculate
\begin{align}
	\begin{split}
		\partial_{u_{2k-3}} a_1(x,u,t) &= \frac{2\pi}{\gamma_1} u_{2k-2} \sin\big(2\pi u_{2k-2}(x_1-x_k) \big),\\
		\partial_{u_{2k-2}} a_1(x,u,t) &=\frac{2\pi}{\gamma_1}
		\bigg(u_{2k-3}\sin \big(2\pi u_{2k-2}(x_1-x_k)\big) 
		\\
		&\hspace{12mm}
		+ 2\pi u_{2k-3} u_{2k-2} (x_1-x_k)\cos\big(2\pi u_{2k-2}(x_1-x_k)\big)\bigg).
	\end{split}
	\label{eq:SPT_derivative_u_a1}
\end{align}
For the other right-hand sides $a_i$, $i=2,\ldots,d$, we calculate
\begin{align}
	\begin{split}
		\partial_{u_{2i-3}} a_i &= - \frac{2\pi}{\gamma_i} u_{2i-2} \sin \big( 2\pi u_{2i-2} (x_1-x_i) \big), 
		\\
		\partial_{u_{2i-2}} a_i &= - \frac{2\pi}{\gamma_i} \Big( 
		u_{2i-3}\sin\big(2\pi u_{2i-2}(x_1-x_i) \big) 
		\\
		&\qquad\qquad    
		+ 2\pi u_{2i-2} u_{2i-3} (x_1-x_i) \sin\big( 2\pi u_{2i-2}(x_1-x_i) \big)
		\Big).
	\end{split}
	\label{eq:SPT_derivative_u_ai}
\end{align}
All other components vanish, i.e. $\partial_{u_j} a_i = 0$ if $i\neq j$ for $j = 1,\ldots,d$, $i=2,\ldots,d $ since the position of each particle depends not on parameters of the other particles. Hence, $\partial_u a(x,u,t)$ has the structure
\begin{align*}
	\begin{pmatrix}
		\partial_{u_1} a_0 & \partial_{u_2} a_0 & \hdots & \hdots & \hdots & \hdots & \partial_{u_{2K}} a_0 \\
		\partial_{u_1} a_1 & 0 & 0  & \partial_{u_K} a_1  &0 &\hdots &0  \\
		0& \ddots & 0  & 0 &\ddots  & 0 & \vdots \\
		\vdots& 0 & \ddots  & 0  &0 &\ddots &0 \\
		0& \hdots  & 0  & \partial_{u_K} a_K  & 0& 0&\partial_{u_{2K}} a_\Npart &
	\end{pmatrix}.
\end{align*}

Notice that $\partial_u b =0$ and $\partial_x b =0$. We define the following continuous problem
\begin{subequations}
	\begin{align}
		&\min J(\bm X, u)=
		\frac{1}{2}\standardNorm{[\Cm(\bX)](\cdot) \,-\, \cd(\cdot)}_ {L^2(0,T)}^2+\frac{1}{2}\standardNorm{[\Cm(\bX)](T) \,-\, \cd(T)}_2^2,\\
		&\hspace{0.5mm}\text{s.t. }(\bm X,u) \in \X \times \Rb^{2\Npart}\text{ solves \eqref{eq:Stochastic_Prantl_Tomlinson_intro}},
	\end{align}
	\label{eq:MCD_Calibration}
\end{subequations}
where we have set
\begin{align}
	\label{eq:def_Xu}
	[\Cm(\bX)](t)\colonequals\frac{\Eb[Y_1(t)Y_1(0)]}{\Eb[Y_1^2(t)]}
	\quad\text{and}\quad 
	Y_1(t)=X_1(t)-\Eb[X_1(t)], 
\end{align}
The term $[\Cm(\bX)](t)$ stands for the (normalized) correlation function and is connected to the \emph{mean conditional displacement} that is used in \cite{Jain2021microStochPrandtlTomlison}.
Notice that it holds for all $t$ that 
\begin{align*}
	[\Cm(\bm X)](t) = \frac{\mathrm{Cov}(Y_1(t),Y_1(0))}{\Vb[Y_1(t)]},
\end{align*}
since $\Eb[Y_1(t)]=0$ by the definition of $Y_1$.

\begin{remark}
	\begin{enumerate}
		\item [\em 1)] The coefficients functions $a$ and $b$ are given by \eqref{eq:coefficients_SPT_ab}. Hence, they consist of functions that are either affine linear in $x$ and $u$ or smooth and bounded and therefore fulfill the Lipschitz property. Moreover, by linearity and boundedness, they also fulfill \eqref{Assumption:coefficients_integral}. Consequently, \cref{Assumption:coefficients} holds also for the SPT model.
		\item [\em 2)] Also the coefficients functions $a$ and $b$ are continuously differentiable functions as a composition of smooth functions. Their first partial derivatives are given by \eqref{eq:SPT_derivative_u_a1}, \eqref{eq:SPT_derivative_u_ai}, \eqref{eq:SPT_derivative_x_a}.
		Furthermore, the derivatives are bounded in $\bX$ due to the boundedness of the $\mathrm{sin}$ function and bounded in $u$ due to our assumptions on $\Uad$.
		Thus, \cref{Assumption:Diffbarkeit}-\ref{AssumptionState} and \cref{A:5.3} are fulfilled.
		\item [\em 3)] We do not have the convexity of $\mathcal{C}$ since it contains the variance. 
		But assuming that the variance of $X_u$ is not vanishing, we have continuous differentiability of the functional as a composition of continuous differentiable functions and hence that \cref{Assumption:Diffbarkeit}-\ref{AssumptionCost} is fulfilled. 
		Furthermore, the functional is also bounded from below by zero, since it consists of norms. Consequently, \cref{Assumption:continuity_C} is only partially satisfied.
	\end{enumerate}
\end{remark}

Now, we formulate the discretized model and the discretized adjoint model to the problem \eqref{eq:MCD_Calibration}.
For this, we define the discrete version of $\Cm$ as follows:
\begin{align*}
	\Cm_{\nu}^M(\bX^h) \colonequals\frac{\Eb^M[Y_{1,\nu}^hY_{1,0}^h]}{\Eb^M[(Y_{1,\nu}^h)^2]}
	\quad\text{for }\nu = 1,\ldots,N\text{ and }Y_{1,\nu}^h=X_{1,\nu}^h-\Eb^M\big[X_{1,\nu}^h\big].
\end{align*}
Recall the definition of $\Eb^M$ in \eqref{ExpValDisc}. Notice that the derivative of the functional with respect to $X_{1,\nu}^\mu$, $\nu \in \{1,\ldots,N\}$, $\mu \in \{1,\ldots,M\}$ is given by
\begin{align*}
	&\left(\Cm^M_\nu(\bX^h)-\cdIndex{\nu}\right)
	\bigg(\frac{1}{M^2}\sum_{\mu'=1}^M (Y^{\mu'}_{1,\nu})^2\,Y^\mu_{1,0} 
	-\frac{1}{M^2}\sum_{\mu'=1}^M 
	\big(Y^{\mu'}_{1,\nu}Y^{\mu'}_{1,0}\big)\,2Y^{\mu}_{1,\nu}\bigg)
	\bigg(\frac{1}{M}\sum_{\mu'=1}^M(Y^{\mu'}_{1,\nu})^2\bigg)^{-2} 
	\\
	&\quad=\frac{1}{M}\left(\Cm^M_\nu(\bX^h)-\cdIndex{\nu}\right) 
	\Eb^M\big[(Y^h_{1,\nu})^2\big]  
	\left(Y^\mu_{1,0}-2\,\Cm_\nu^M(\bX^h)Y^{\mu}_{1,\nu}
	\right)\big(\Eb^M\big[(Y^h_{1,\nu})^2\big]\Big)^{-2}\\
	&\quad=\frac{1}{M}\left(\Cm^M_\nu(\bX^h)-\cdIndex{\nu}\right)  
	\frac{Y^\mu_{1,0}-2\,\Cm_\nu^M(\bX^h)Y^{\mu}_{1,\nu}}{\Eb^M[(Y^h_{1,\nu})^2]}\\
\end{align*}
Thus, we can write the derivative of the functional with respect to the state variable $\bX^h = (X^h_i)_{i=1,\ldots,d}$ as
\begin{align}
	J^h_{\bX^h}(\bX^{\mu}_{\nu}) =
	\begin{cases}
		\frac{1}{M}\left(\Cm^M_\nu(\bX^h)-\cdIndex{\nu}\right)  
		\frac{Y^\mu_{1,0}-2\,\Cm_\nu^M(\bX^h)Y^{\mu}_{1,\nu}}{\Eb^M[(Y^h_{1,\nu})^2]}
		& \text{ if } i = 1,\\
		0 & \text{ otherwise}.
	\end{cases}
\end{align}

We introduce the adjoint variable $\bL^h=\{\bL^\mu_\nu\}\in\X^h$. Using \eqref{eq:adjoint_SDE_time} and \eqref{eq:jacobian_a_x}, we can formulate the discrete adjoint equation that evolves backwards in time for $\nu = N-1,\ldots,0$
\begin{subequations}
	\label{eq:SPT_discrete_adjoint}
	\begin{align}
		\bL^{\mu}_{\nu} = \left(I + \Delta t \, a_x^\nu(\bm X^\mu_{\nu},u) \right)^\top\bL^\mu_{\nu+1} +\Delta t \, J^h_{\bX^h}(\bX^\mu_\nu)\quad\text{for }\mu=1,\ldots,m
		\label{eq:AdjointEquation_SPT}
	\end{align}
	with the terminal condition
	\begin{align}
		\bL^\mu_{N}=J^h_{\bX^h}(\bX^\mu_N)\quad\text{for }\mu=1,\ldots,m.
	\end{align}
\end{subequations}
Furthermore, the discrete gradient is given by
\begin{align}
	\nabla \hat{J}^h(\mathrm u)= \frac{1}{M} \frac{1}{N-1}\sum_{\mu=1}^M \sum_{\nu=0}^{N-1}\left( a_u^\nu(\bm X^\mu_\nu,u) \right)^\top \bL_{\nu+1}^\mu.
	\label{eq:SPT_discrete_gradient}
\end{align}

For the first numerical experiment using the SPT model, we consider a single bath particle, i.e. we set $K=1$. 
In \cref{fig:SPT}, we present the results. 
The plot in \cref{fig:SPT_N1_merged} shows the simulation of \eqref{eq:Stochastic_Prantl_Tomlinson_intro} using the initial guess of the parameters and the simulation using the parameters obtained by Algorithm \ref{alg:PSGD_SOCP}.
\begin{figure}
	\begin{subfigure}[l]{0.4\textwidth}
		\includegraphics[width=\textwidth]{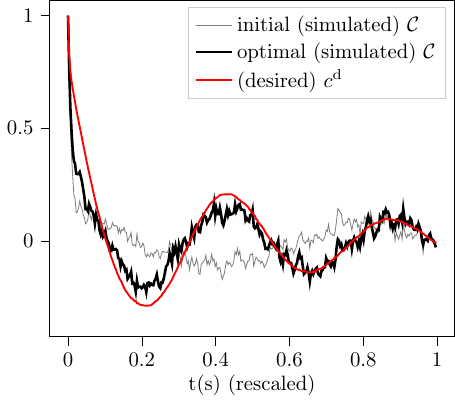}
		\caption{$K=1$}
		\label{fig:SPT_N1_merged}
	\end{subfigure}
	\hfill
	\begin{subfigure}[l]{0.4\textwidth}
		\includegraphics[width=\textwidth]{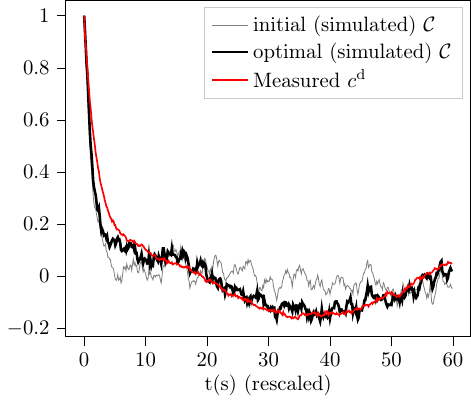}
		\caption{$K = 2$}
		\label{fig:SPT_N2_merged}
	\end{subfigure}		
	\caption{
		Results of numerical experiments for \eqref{eq:MCD_Calibration}. Simulation of $X_u$ defined in \eqref{eq:def_Xu} using the initial (grey) and optimal (black) guess of parameters for (a) $K=1$ and (b) $K=2$, each time compared with the desired $\cd$ (red).}
	\label{fig:SPT}
\end{figure}

For a second test, we consider \eqref{eq:MCD_Calibration} with two bath-particles, i.e. $K=2$. 
In this case, we observed that a longer calibration time is needed, also because the equilibration time $t_{eq}$ is longer (cf. \cref{rem:initial_condition_equilibration}). 
We plot the results in \cref{fig:SPT_N2_merged}.

It is evident, that our method manages to find optimal parameters in order to match the desired behavior of the model.


\section*{Acknowledgments}

We would like to express our gratitude to Juliana Caspers, Luis Reinalter as well as Clemens Bechinger and Matthias Kr\"uger for very helpful physical insights in the SPT model. We also thank Behzad Azmi for very fruitful discussions.
We would also like to express our gratitude to the anonymous referees for their helpful questions and remarks.

\section*{Funding}

This work was partially funded by the Deutsche Forschungsgemeinschaft within SFB 1432, Project-ID 425217212.

\section*{Competing interests}

The authors have no relevant financial or non-financial interests to disclose.

\section*{Author contributions}

All three authors contribute to the present manuscript in the same way.


\bibliographystyle{acm}

\begin{thebibliography}{10}
	
	\bibitem{AnnunziatoBorzi2018FPFrameworkSOC}
	{\sc Annunziato, M., and Borz\`{i}, A.}
	\newblock A {F}okker-{P}lanck control framework for stochastic systems.
	\newblock {\em EMS Surv. Math. Sci. 5}, 1-2 (2018), 65--98.
	
	\bibitem{Bartsch2021MOCOKI}
	{\sc Bartsch, J., and Borz\`\i, A.}
	\newblock M{OCOKI}: {A} {M}onte {C}arlo approach for optimal control in the
	force of a linear kinetic model.
	\newblock {\em Comput. Phys. Commun. 266\/} (2021), 108030.
	
	\bibitem{Bartsch2020OCPKS}
	{\sc Bartsch, J., Nastasi, G., and Borz\`{i}, A.}
	\newblock Optimal control of the {K}eilson-{S}torer master equation in a
	{M}onte {C}arlo framework.
	\newblock {\em J. Comput. Theor. Transp. 50}, 5 (2021), 454--482.
	
	\bibitem{BensoussanFrehseYam2013MeanFieldGamesControl}
	{\sc Bensoussan, A., Frehse, J., and Yam, P.}
	\newblock {\em Mean field games and mean field type control theory}.
	\newblock SpringerBriefs in Mathematics. Springer, New York, 2013.
	
	\bibitem{Bismut1973ConjugateConvexOSC}
	{\sc Bismut, J.-M.}
	\newblock Conjugate convex functions in optimal stochastic control.
	\newblock {\em J. Math. Anal. Appl. 44\/} (1973), 384--404.
	
	\bibitem{Bismut1978IntroOCPStoch}
	{\sc Bismut, J.-M.}
	\newblock An introductory approach to duality in optimal stochastic control.
	\newblock {\em SIAM Rev. 20}, 1 (1978), 62--78.
	
	\bibitem{BonnetFrankowska2021NecessaryOCPWasserstein}
	{\sc Bonnet, B., and Frankowska, H.}
	\newblock Necessary optimality conditions for optimal control problems in
	{W}asserstein spaces.
	\newblock {\em Appl. Math. Optim. 84\/} (2021), S1281--S1330.
	
	\bibitem{BottouCurtis2018ReviewStochOpt}
	{\sc Bottou, L., Curtis, F.~E., and Nocedal, J.}
	\newblock Optimization methods for large-scale machine learning.
	\newblock {\em SIAM Rev. 60}, 2 (2018), 223--311.
	
	\bibitem{BreitenbachBorzi2020PMPFPOCP}
	{\sc Breitenbach, T., and Borz\'{\i}, A.}
	\newblock The {P}ontryagin maximum principle for solving {F}okker-{P}lanck
	optimal control problems.
	\newblock {\em Comput. Optim. Appl. 76}, 2 (2020), 499--533.
	
	\bibitem{Buckdahn2010existenceStochasticControl}
	{\sc Buckdahn, R., Labed, B., Rainer, C., and Tamer, L.}
	\newblock Existence of an optimal control for stochastic control systems with
	nonlinear cost functional.
	\newblock {\em Stochastics 82}, 3 (2010), 241--256.
	
	\bibitem{ClevenhausTotzeckEhrhardt2023GradientCalibrationHeston}
	{\sc Clevenhaus, A., Totzeck, C., and Ehrhardt, M.}
	\newblock A gradient based calibration method for the heston model.
	\newblock {\em arXiv\/} (2023).
	
	\bibitem{ConlanWang1992ContDiscGronwall}
	{\sc Conlan, J., and Wang, C.-L.}
	\newblock A unified approach to continuous and discrete {G}ronwall-{B}ellman
	inequalities.
	\newblock {\em Appl. Anal. 44}, 3-4 (1992), 243--252.
	
	\bibitem{Cresson2018DerivationSDEBiology}
	{\sc Cresson, J., and Sonner, S.}
	\newblock A note on a derivation method for {SDE} models: applications in
	biology and viability criteria.
	\newblock {\em Stoch. Anal. Appl. 36}, 2 (2018), 224--239.
	
	\bibitem{Daudin2023OCPStateConstraintsFP}
	{\sc Daudin, S.}
	\newblock Optimal control of the {F}okker-{P}lanck equation under state
	constraints in the {W}asserstein space.
	\newblock {\em J. Math. Pures Appl. (9) 175\/} (2023), 37--75.
	
	\bibitem{Dhont1996introductionColloids}
	{\sc Dhont, J.~K.}
	\newblock {\em An introduction to dynamics of colloids}.
	\newblock Elsevier, 1996.
	
	\bibitem{ElKarouiNguyen1987ExistenceOSCP}
	{\sc El~Karoui, N., H\.{u}\.{u}~Nguyen, D., and Jeanblanc-Picqu\'{e}, M.}
	\newblock Compactification methods in the control of degenerate diffusions:
	existence of an optimal control.
	\newblock {\em Stochastics 20}, 3 (1987), 169--219.
	
	\bibitem{Estevao2000FunctionalCalibration}
	{\sc Estevao, V.~M., and Sa{\`E}rndal, C.-E.}
	\newblock A functional form approach to calibration.
	\newblock {\em J Off Stat 16}, 4 (2000), 379.
	
	\bibitem{Evans2013SDE}
	{\sc Evans, L.~C.}
	\newblock {\em An introduction to stochastic differential equations}.
	\newblock American Mathematical Society, Providence, RI, 2013.
	
	\bibitem{Fabbri2017StochasticOptimalControlInInfi}
	{\sc Fabbri, G., Gozzi, F., and \'{S}wie\'ch, A.}
	\newblock {\em Stochastic optimal control in infinite dimension}, vol.~82 of
	{\em Probability Theory and Stochastic Modelling}.
	\newblock Springer, Cham, 2017.
	\newblock Dynamic programming and HJB equations, With a contribution by Marco
	Fuhrman and Gianmario Tessitore.
	
	\bibitem{FabrizioMonetti2015MethodologiesCalibration}
	{\sc Fabrizio, E., and Monetti, V.}
	\newblock Methodologies and advancements in the calibration of building energy
	models.
	\newblock {\em Energies 8}, 4 (2015), 2548--2574.
	
	\bibitem{Frankowska2019NecessaryConditionsSDEStateConstraint}
	{\sc Frankowska, H., Zhang, H., and Zhang, X.}
	\newblock Necessary optimality conditions for local minimizers of stochastic
	optimal control problems with state constraints.
	\newblock {\em Trans. Amer. Math. Soc. 372}, 2 (2019), 1289--1331.
	
	\bibitem{Gilli2011calibratingPricingModels}
	{\sc Gilli, M., and Schumann, E.}
	\newblock Calibrating option pricing models with heuristics.
	\newblock In {\em Natural computing in computational finance}. Springer, 2011,
	pp.~9--37.
	
	\bibitem{Gong2017efficientSGD}
	{\sc Gong, B., Liu, W., Tang, T., Zhao, W., and Zhou, T.}
	\newblock An efficient gradient projection method for stochastic optimal
	control problems.
	\newblock {\em SIAM J Numer Anal 55}, 6 (2017), 2982--3005.
	
	\bibitem{GrahamTalay2013StochSimMC}
	{\sc Graham, C., and Talay, D.}
	\newblock {\em Stochastic simulation and {M}onte {C}arlo methods}, vol.~68 of
	{\em Stochastic Modelling and Applied Probability}.
	\newblock Springer, Heidelberg, 2013.
	\newblock Mathematical foundations of stochastic simulation.
	
	\bibitem{Gross2015applications}
	{\sc Gro{\ss}, B.}
	\newblock {\em Applications of the Adjoint Method in Stochastic Financial
		Modelling}.
	\newblock PhD thesis, University of Trier, 2015.
	
	\bibitem{Higham2001IntroNumSDE}
	{\sc Higham, D.~J.}
	\newblock An algorithmic introduction to numerical simulation of stochastic
	differential equations.
	\newblock {\em SIAM Rev. 43}, 3 (2001), 525--546.
	
	\bibitem{HPUU08}
	{\sc Hinze, M., Pinnau, R., Ulbrich, M., and Ulbrich, S.}
	\newblock {\em Optimization with {PDE} constraints}, vol.~23 of {\em
		Mathematical Modelling: Theory and Applications}.
	\newblock Springer, New York, 2009.
	
	\bibitem{Jagla2018SPT}
	{\sc Jagla, E.~A.}
	\newblock The {P}randtl-{T}omlinson model of friction with stochastic driving.
	\newblock {\em J. Stat. Mech. Theory Exp.}, 1 (2018), 013401, 14.
	
	\bibitem{Jain2021twoStepRheolog}
	{\sc Jain, R., Ginot, F., Berner, J., Bechinger, C., and Kr{\"u}ger, M.}
	\newblock Two step micro-rheological behavior in a viscoelastic fluid.
	\newblock {\em J. Chem. Phys. 154}, 18 (2021), 184904.
	
	\bibitem{Jain2021microStochPrandtlTomlison}
	{\sc {Jain}, R., {Ginot}, F., and {Kr{\"u}ger}, M.}
	\newblock {Micro-rheology of a particle in a nonlinear bath: Stochastic
		Prandtl-Tomlinson model}.
	\newblock {\em Phys. Fluids 33}, 10 (Oct. 2021), 103101.
	
	\bibitem{Kaebe2009AdjointMonteCarloCalibra}
	{\sc Kaebe, C., Maruhn, J.~H., and Sachs, E.~W.}
	\newblock Adjoint-based {M}onte {C}arlo calibration of financial market models.
	\newblock {\em Finance Stoch. 13}, 3 (2009), 351--379.
	
	\bibitem{Karatzas1988BMSDE}
	{\sc Karatzas, I., and Shreve, S.~E.}
	\newblock {\em Brownian motion and stochastic calculus}, vol.~113 of {\em
		Graduate Texts in Mathematics}.
	\newblock Springer-Verlag, New York, 1988.
	
	\bibitem{Kawasaki1973DerivationGLE}
	{\sc {Kawasaki}, K.}
	\newblock {Simple derivations of generalized linear and nonlinear Langevin
		equations}.
	\newblock {\em J. Phys. A Math. Theor. 6}, 9 (Sept. 1973), 1289--1295.
	
	\bibitem{KosmolPavon1993LagrangeOCP}
	{\sc Kosmol, P., and Pavon, M.}
	\newblock Lagrange approach to the optimal control of diffusions.
	\newblock {\em Acta Appl. Math. 32}, 2 (1993), 101--122.
	
	\bibitem{Kubo1966FluctuationDissipation}
	{\sc Kubo, R.}
	\newblock The fluctuation-dissipation theorem.
	\newblock {\em Rep. Prog. Phys. 29}, 1 (1966), 255.
	
	\bibitem{LoHaslam2008algorithmParamEstSDE}
	{\sc L{\'o}, B.~P., Haslam, A., and Adjiman, C.}
	\newblock An algorithm for the estimation of parameters in models with
	stochastic differential equations.
	\newblock {\em Chem. Eng. Sci. 63}, 19 (2008), 4820--4833.
	
	\bibitem{LuZhang2014GeneralPMPBSEE}
	{\sc L\"{u}, Q., and Zhang, X.}
	\newblock {\em General {P}ontryagin-type stochastic maximum principle and
		backward stochastic evolution equations in infinite dimensions}.
	\newblock SpringerBriefs in Mathematics. Springer, Cham, 2014.
	
	\bibitem{Lue69}
	{\sc Luenberger, D.~G.}
	\newblock {\em Optimization by Vector Space Methods}.
	\newblock John Wiley \& Sons, Inc., New York, NY, 1969.
	
	\bibitem{MaYong1999FBSDEapplic}
	{\sc Ma, J., and Yong, J.}
	\newblock {\em Forward-backward stochastic differential equations and their
		applications}, vol.~1702 of {\em Lecture Notes in Mathematics}.
	\newblock Springer-Verlag, Berlin, 1999.
	
	\bibitem{Mao2008SDE}
	{\sc Mao, X.}
	\newblock {\em Stochastic differential equations and applications}, second~ed.
	\newblock Horwood Publishing Limited, Chichester, 2008.
	
	\bibitem{Maruyama1955SDE}
	{\sc Maruyama, G.}
	\newblock Continuous {M}arkov processes and stochastic equations.
	\newblock {\em Rend. Circ. Mat. Palermo (2) 4\/} (1955), 48--90.
	
	\bibitem{Milstein-Tretyakov21}
	{\sc Milstein, G.~N., and Tretyakov, M.~V.}
	\newblock {\em Stochastic numerics for mathematical physics}, second~ed.
	\newblock Scientific Computation. Springer, Cham, [2021] \copyright 2021.
	
	\bibitem{Muller2020NonLinBath}
	{\sc M{\"u}ller, B., Berner, J., Bechinger, C., and Kr{\"u}ger, M.}
	\newblock Properties of a nonlinear bath: experiments, theory, and a stochastic
	prandtl--tomlinson model.
	\newblock {\em New J. Phys. 22}, 2 (2020), 023014.
	
	\bibitem{Oksendal2000SDEIntroApplication}
	{\sc {\O}ksendal, B.}
	\newblock {\em Stochastic differential equations: an introduction with
		applications}.
	\newblock Springer Science \& Business Media, 2013.
	
	\bibitem{PaulTrelat2024microscopicMacroMeso}
	{\sc Paul, T., and Tr{\'e}lat, E.}
	\newblock From microscopic to macroscopic scale dynamics: mean field,
	hydrodynamic and graph limits.
	\newblock {\em arXiv preprint arXiv:2209.08832\/} (2024).
	
	\bibitem{Pham2009StochasticModellingOptimization}
	{\sc Pham, H.}
	\newblock {\em Continuous-time stochastic control and optimization with
		financial applications}, vol.~61 of {\em Stochastic Modelling and Applied
		Probability}.
	\newblock Springer-Verlag, Berlin, 2009.
	
	\bibitem{PopovGray2014Prandtl}
	{\sc Popov, V.~L., and Gray, J. A.~T.}
	\newblock {\em Prandtl-Tomlinson Model: A Simple Model Which Made History}.
	\newblock Springer Berlin Heidelberg, Berlin, Heidelberg, 2014, pp.~153--168.
	
	\bibitem{Prandtl1928gedankenmodell}
	{\sc Prandtl, L.}
	\newblock Ein gedankenmodell zur kinetischen theorie der festen k{\"o}rper.
	\newblock {\em ZAMM-Journal of Applied Mathematics and Mechanics/Zeitschrift
		f{\"u}r Angewandte Mathematik und Mechanik 8}, 2 (1928), 85--106.
	
	\bibitem{Protter05}
	{\sc Protter, P.~E.}
	\newblock {\em Stochastic integration and differential equations}, second~ed.,
	vol.~21 of {\em Stochastic Modelling and Applied Probability}.
	\newblock Springer-Verlag, Berlin, 2005.
	\newblock Corrected third printing.
	
	\bibitem{RenardyRogers2004IntroPDE}
	{\sc Renardy, M., and Rogers, R.~C.}
	\newblock {\em An introduction to partial differential equations}, second~ed.,
	vol.~13 of {\em Texts in Applied Mathematics}.
	\newblock Springer-Verlag, New York, 2004.
	
	\bibitem{SachsSchu2013GradienComputationModelCalibration}
	{\sc Sachs, E.~W., and Schu, M.}
	\newblock Gradient computation for model calibration with pointwise
	observations.
	\newblock In {\em Control and optimization with {PDE} constraints}, vol.~164 of
	{\em Internat. Ser. Numer. Math.} Birkh\"{a}user/Springer Basel AG, Basel,
	2013, pp.~117--136.
	
	\bibitem{Sekimoto1998MarkovianLangevin}
	{\sc {Sekimoto}, K.}
	\newblock {Langevin Equation and Thermodynamics}.
	\newblock {\em Progress of Theoretical Physics Supplement 130\/} (Jan. 1998),
	17--27.
	
	\bibitem{TanimuraYoshitakaWolynes1991QuantumClassicalFPMarkov}
	{\sc {Tanimura}, Y., and {Wolynes}, P.~G.}
	\newblock {Quantum and classical Fokker-Planck equations for a
		Gaussian-Markovian noise bath}.
	\newblock {\em Phys. Rev. A 43}, 8 (Apr. 1991), 4131--4142.
	
	\bibitem{Treoltzsch2010OCPPDE}
	{\sc Tr{\"o}ltzsch, F.}
	\newblock {\em Optimal control of partial differential equations: Theory,
		methods, and applications}, vol.~112 of {\em Graduate Studies in
		Mathematics}.
	\newblock American Mathematical Society, Providence, RI, Providence, RI, 2010.
	
	\bibitem{Kampen1976SDE}
	{\sc van Kampen, N.~G.}
	\newblock Stochastic differential equations.
	\newblock {\em Phys. Rep. 24}, 3 (1976), 171--228.
	
	\bibitem{Kampen1981StochasticProcesses}
	{\sc van Kampen, N.~G.}
	\newblock {\em Stochastic processes in physics and chemistry}, vol.~888 of {\em
		Lecture Notes in Mathematics}.
	\newblock North-Holland Publishing Co., Amsterdam-New York, 1981.
	
	\bibitem{YongZhou1999StochasticControls}
	{\sc Yong, J., and Zhou, X.~Y.}
	\newblock {\em Stochastic Controls}, vol.~43 of {\em Applications of
		Mathematics (New York)}.
	\newblock Springer-Verlag, New York, 1999.
	\newblock Hamiltonian systems and HJB equations.
	
	\bibitem{Zwanzig1973NonlinearGeneralizedLangevin}
	{\sc Zwanzig, R.}
	\newblock Nonlinear generalized {L}angevin equations.
	\newblock {\em J. Stat. Phys. 9}, 3 (1973), 215--220.
	
\end{thebibliography}

\end{document}